\documentclass[10pt,reqno]{amsart}

\usepackage{algorithm}
\usepackage{algpseudocode}
\usepackage{todonotes}
\usepackage{a4wide}
\usepackage{siunitx}
\usepackage{afterpage}
\usepackage[section]{placeins}

\usepackage{amsmath,amssymb}
\usepackage{rotating}

\newtheorem{thm}{Theorem}[section]
\newtheorem{cor}[thm]{Corollary}
\newtheorem{hyp}[thm]{Hypothesis}
\newtheorem{lemma}[thm]{Lemma}
\newtheorem{prop}[thm]{Proposition}
\newtheorem{example}[thm]{Example}

\theoremstyle{definition}

\theoremstyle{remark}
\newtheorem{rem}{Remark}[section]
\numberwithin{equation}{section}

\newcommand{\Id}{\mathrm{I}}

\DeclareMathOperator*{\esssup}{ess\,sup}
\DeclareMathOperator{\Div}{\mathrm{div}}

\DeclareMathOperator{\Grad}{\nabla}
\DeclareMathOperator{\trace}{\mathrm{tr}}

\newcommand{\suml}[2]{\ensuremath{\sum\limits_{#1}^{#2}}}
\newcommand{\ssum}{\textstyle \sum} 
\newcommand{\tr}[1]{\trace\left(#1\right)}
\newcommand{\sig}[1]{\sigma\left(#1\right)}
\newcommand{\e}[1]{\varepsilon\left(#1\right)}

\newcommand{\inner}[2]{\langle #1,#2\rangle}
\newcommand{\dual}[3]{\inner{#1}{#2}_{#3}}
\newcommand{\foralls}{\forall \,}

\newcommand{\nrmbar}[1]{\vert\vert #1 \vert\vert}
\newcommand{\nrm}[2]{\nrmbar{#1}_{#2}}

\newcommand{\ls}{\lesssim}

\usepackage{booktabs}
\usepackage{subcaption}
\usepackage[linkbordercolor=white, colorlinks=true, linkcolor=blue, allcolors=blue]{hyperref}
\usepackage{cleveref}
\usepackage[numbers]{natbib}

\newcommand{\R}{\mathbb{R}}
\newcommand{\dx}{\, \mathrm{d}x}
\newcommand{\ds}{\, \mathrm{d}s}
\newcommand{\alphafac}{\alpha_{\eta}}

\newcommand{\changed}[1]{\textcolor{black}{#1}}

\begin{document}

\title[A posteriori estimation for multiple-network poroelasticity]{A posteriori error estimation and adaptivity for multiple-network poroelasticity}
\author{E.~Eliseussen \and M.~E.~Rognes \and T.~B.~Thompson}
\address[E.~Eliseussen, e.e.odegaard@medisin.uio.no]{Institute of Basal Medical Sciences, University of Oslo, Oslo, Norway}
\address[M.~E.~Rognes, meg@simula.no]{Department of Scientific Computing and Numerical Analysis, Simula Research Laboratory, Oslo, Norway; Department of Mathematics, University of Bergen, Bergen, Norway.}
\address[T.~B.~Thompson, travis.thompson@ttu.edu]{Department of Mathematics and Statistics, Texas Tech University, Lubbock, Texas, United States.}

\thanks{This project has received funding from the European Research
  Council (ERC) under the European Union's Horizon 2020 research and
  innovation programme under grant agreement 714892 and by the
  Research Council of Norway under the FRINATEK Young Research Talents
  Programme through project \#250731/F20 (Waterscape). Parts of the
  computations were performed on resources provided by UNINETT Sigma2
  - the National Infrastructure for High Performance Computing and
  Data Storage in Norway.}

\date{\today}


\begin{abstract}
  The multiple-network poroelasticity (MPET) equations describe
  deformation and pressures in an elastic medium permeated by
  interacting fluid networks. In this paper, we (i) place these
  equations in the theoretical context of coupled elliptic-parabolic
  problems, (ii) use this context to derive residual-based
  a posteriori error estimates and indicators for fully discrete MPET
  solutions and (iii) evaluate the performance of these error
  estimators in adaptive algorithms for a set of test cases: ranging
  from synthetic scenarios to physiologically realistic simulations of
  brain mechanics.
\end{abstract}


\maketitle

\section{Introduction}
\label{sec:introduction}

At the macroscale, the brain and other biological tissues can often be
viewed as a poroelastic medium: an elastic structure permeated by one
or more fluid networks. Such structures can be modeled via Biot's
equations in the case of a single fluid network~\citep{biot-1941,
  biot-1955, terzaghi-book-1943} or by their generalization to the
equations of multiple-network poroelasticity (MPET) which describe the
case of two or more interacting fluid networks~\citep{aifantis-1980,
  bai1993multiporosity,barenblatt-1963,barenblatt-1960,vardakis-ventikos-2018,
  aifantis-1984,tully-ventikos-2011, vardakis-ventikos2016,aifantis-1982}. 
However, the computational expense associated with the numerical solution of 
these equations, over complex domains such as the human brain, is substantial.  
A natural question is therefore whether numerical error estimation and 
adaptivity can yield more accurate simulations of the MPET equations within a 
limited set of computational resources.

The quasi-static MPET equations read: given a domain $\Omega$, a
finite final time $T > 0$ and a set of $J$ fluid networks, find the
displacement field $u:[0,T] \times \Omega \rightarrow \R^d$ and
pressure fields $p_1, p_2 \dots p_{J}:[0,T] \times \Omega \rightarrow
\R$ such that
\begin{subequations}
  \label{eq:mpet}
  \begin{align}
    \label{eq:mpet-a}
    -\Div \sigma(u) + \suml{j=1}{J} \alpha_j \Grad p_j &= f, \\
    \label{eq:mpet-b}
    \partial_t \left(s_j p_j +  \alpha_j \Div u \right) - \Div \left(\kappa_j \Grad p_j \right) + T_j &= g_j.
  \end{align}
\end{subequations}
The quantity $\sig{u} = 2\mu\e{u} + \lambda \tr{\e{u}}\Id$
in~\eqref{eq:mpet-a} is the elastic stress tensor and involves
the displacement $u$, the linearized strain tensor $\e{u} =
\frac12\left(\Grad u + \Grad u^{T}\right)$, the $d\times d$ identity
matrix $\Id$ and the material Lam{\'e} coefficients $\mu$ and
$\lambda$. Each one of the $J$ fluid networks is associated with a Biot-Willis 
coefficient $\alpha_j$, a storage coefficient $s_j$, and a hydraulic 
conductivity $\kappa_j$.  An interpretation of the Biot-Willis and storage 
coefficients, in the MPET context, appears in ~\cite[Section 
\S3]{bai1993multiporosity}.  We use transfer terms~$T_{j}$ in~\eqref{eq:mpet-b} 
of the form
\begin{equation}
  \label{eq:def:T_j}
  T_j = \sum_{i=1}^J T_{ji}, \quad T_{ji} = \gamma_{ji} (p_j - p_i) .
\end{equation}
The coefficients $\gamma_{ji}$ regulate the interplay between network
$i$ and network $j$ and $T_j$ is the total transfer out of network $j$
(into the other networks). The transfer term vanishes when $J=1$ and
\eqref{eq:mpet} coincides with Biot's equations for a single fluid in
a poroelastic medium.  We also note that the fluid (Darcy) velocity
$v_j$ in network $j$ is defined by
\begin{equation}
  v_j = - \kappa_j \Grad p_j .
  \label{eq:velocity}
\end{equation}

Over the last decade, several authors have studied a posteriori error
estimation and adaptivity related to \eqref{eq:mpet} in the case of $J=1$; 
that is, for Biot's equations of poroelasticity.  Depending on the application 
of interest, different formulations of Biot's equations have been used which 
introduce additional solution fields such as the Darcy velocity, the total 
pressure or the effective stress. In each case, a posteriori methods have 
been developed to facilitate adaptive refinement strategies. In 
\cite{riedlbeck2017}, the authors consider the standard two-field formulation 
of Biot's equation in two spatial dimensions, develop an a posteriori error 
analysis based on $H(\Div)$ reconstructions of the flux and effective stress 
and apply the resulting estimators to construct a time-space adaptive algorithm.  
A posteriori estimators have also been used to provide error 
estimates for the popular fixed-stress iterative solution scheme applied to 
the two-field formulation \cite{kumar2021}. Formulations with additional fields 
have also been considered for Biot's equations. The total pressure formulation 
\cite{baier2016} is a locking-free, three-field formulation, ideal for a 
nearly-incompressible poroelastic material. A priori estimates, and an 
adaptive refinement strategy, for this formulation are constructed in 
\cite{khan2020robust} for quadrilateral and simplicial meshes.  
Residual-based a posteriori error estimates have also been advanced  
\cite{li2019residual} for a lowest-order discretization of the standard 
Darcy-flux three-field formulation which, as shown in \cite{rodrigo2018}, 
robustly preserves convergence in the presence of vanishingly small hydraulic 
conductivity.  Finally, a four-field formulation, with symmetric stress and a 
Darcy velocity, of Biot's equations has been used to develop  
\cite{ahmed2019adaptive} a posteriori error estimates, and an adaptive 
refinement, based on post-processed pressure and displacement fields.

The a posteriori landscape for the more general MPET system \eqref{eq:mpet} is 
considerably sparse.  A posteriori error estimates for the two-field 
formulation of the Barenblatt-Biot equations (corresponding to the $J = 2$ 
case of~\eqref{eq:mpet}) have indeed been obtained 
by \citet{nordbotten2010posteriori}.  In general, though, there has been 
little work on the development of a posteriori error estimators for 
\eqref{eq:mpet}, for formulations with any number of fields, in the case of 
more than one fluid network (i.e.~$J > 1$).  However, the recent work of  
\cite{ern-meunier-2009} developed an abstract framework for a posteriori 
error estimators for a general class of coupled elliptic-parabolic 
problems. 

In this manuscript, our focus is three-fold. First, we rigorously
place the MPET equations in the context of coupled elliptic-parabolic
problems. In particular, we consider extended spaces, bilinear forms
and augmentation with a semi-inner product arising from the additional
transfer terms. Second, we use this context to derive specific
a posteriori error estimates and error indicators for the space-time
finite element discretizations of the multiple-network poroelasticity
equations in general. In biomedical applications, two-field
variational formulations are often used to numerically approximate the
multiple-network poroelasticity equations~\citep{tully-ventikos-2011,
  vardakis-ventikos2016}, and we therefore focus on such here. Third,
we formulate a physiological modelling and simulation-targeted
adaptive strategy and evaluate this strategy on a series of test cases
including a clinically-motivated simulation of brain mechanics.

\section{Notation and preliminaries}
\label{sec:preliminaries}

This section provides a brief account of the notation and relevant
results employed throughout the remainder of the manuscript.

\subsection{Domain, boundary and meshes}
It is assumed that the poroelastic domain $\Omega \subset \R^d$ with
$d \in \{1, 2, 3\}$ is a bounded, convex domain with $\partial \Omega$
Lipschitz continuous. We consider a family of mesh discretization
$\left\{\mathcal{T}_h\right\}_{h>0}$ of $\Omega$ into simplices;
triangles when $d=2$ and tetrahedra when $d=3$. Here, $h>0$ is a
characteristic mesh size such as the maximum diameter over all
simplices. Furthermore, we assume that each mesh $\mathcal{T}_h$ in
the family is quasi-uniform.

\subsection{Material parameters}
\label{sec:introduction:material-parameters}

For simplicity, we assume that all material parameters are constant
(in space), and that the following (standard) bounds are satisfied:
$\mu > 0$, $2 \mu + \lambda > 0$, $\kappa_j > 0$, $\alpha_j \in (0,
1]$, $s_j > 0$ for $j = 1, \dots, J$, and $\gamma_{ji} = \gamma_{ij}
  \geq 0$ for $i, j = 1, \dots, J$ with $\gamma_{ii} = 0$.  The
  analysis can be extended to the case where the parameters, above,
  vary with sufficient regularity in space and time provided the above
  bounds hold uniformly.  For each parameter, $\xi_i$ or $\xi_{ij}$
  above, the minimum and maximum notation
\[
\xi_{\max} = \max\limits_{i} \xi_i \text{ or }  \max\limits_{ij}\xi_{ij} 
\quad\text{and}\quad 
\xi_{\min} = \min\limits_{i} \xi_i\text{ or }  \min\limits_{ij}\xi_{ij},
\] 
will be used throughout the manuscript; the notational extension to the 
case of smoothly varying parameters, on a bounded domain with 
compact closure, is clear.

\subsection{Norms and function spaces}
Let $f$, $g$ denote real-valued functions with domain $\Omega
\subset\R^d$.  If there exists a generic constant $C$ with $f \leq C
g$ then we write
\begin{equation*}
  f \ls g. 
\end{equation*}
The notation $\inner{f}{g}$ signifies the usual Lebesgue inner product defined 
by
\begin{equation*}
  \inner{f}{g}=\int_{\Omega} fg \dx, 
\end{equation*} 
and $\nrm{f}{} = \inner{f}{f}^{1/2}$ is the corresponding norm on the Hilbert 
space of square-integrable functions
\begin{equation*}
  L^2(\Omega) =
  \left \{ f: \Omega \rightarrow \R \, | \, \nrm{f}{} <  \infty \right \} .
\end{equation*}
When the context is evident in praxis the domain, $\Omega$, is
suppressed in the above expressions. Given $w$ a positive constant,
positive scalar field, or positive-definite tensor field, the
symbolics
\begin{equation*}
  \dual{f}{g}{w} = \inner{wf}{g},\quad %
  \nrm{f}{w} = \dual{f}{f}{w}^{1/2},
\end{equation*}
refer to \changed{a $w$-weighted} inner product and norm, respectively.

The Sobolev space $H^1(\Omega)$, often abbreviated as simply $H^1$,
consists of those functions $f \in L^2$ whereby $\partial_{x_j} f$
exists, in the sense of distributions, for every $j=1,2,\dots,d$ and
$\partial_{x_j}f \in L^2$.  The associated norm is given by the
expression
\begin{equation*}
  \nrm{f}{H^1} = \left(\nrm{f}{}^2 + %
  \suml{j=1}{d}\nrm{\partial_{x_j}f}{}^2\right)^{1/2}.
\end{equation*}
The subset $H_0^1 \subset H^1$ signifies functions with zero trace on the 
boundary; that is, those functions $f\in H^1$ such that $f(x) = 0$ for almost
every $x\in\partial\Omega$. %
%
In addition, given a Hilbert space $X$ with inner product
$\inner{\cdot}{\cdot}_{X}$, the notation $[X]^d$ refers to vectors $f
= [f_1,f_2, \dots,f_d ]^T$ whereby $f_j \in X$ for each $j=1, 2,
\dots, d$. The natural inner product on $[X]^d$, in which $[X]^d$ is
also a Hilbert space, is then
\begin{equation*}
  \inner{f}{g} = \suml{j=1}{d}\inner{f_j}{g_j}_{X},
\end{equation*}
with resulting norm
\begin{equation*}
  \nrm{f}{X^d} = \left(\suml{j=1}{d}\nrm{f_j}{X}^2 \right)^{1/2}.
\end{equation*} 

The additional decoration of the inner product, for the case of a Hilbert space 
$X$, will be omitted when the context is clear.  For $X$ any Banach space, the 
notation $X^{\ast}$ denotes the dual space,and we also write 
$\dual{x^*}{x}{X'\times X}$ for the duality pairing. Accordingly, the operator 
norm on $X^{\ast}$ is denoted
\begin{equation*}
  \nrm{x^*}{X^{\ast}} = \sup\limits_{\nrm{x}{X}=1} |\dual{x^*}{x}{X^{\ast}\times X}|.
\end{equation*}
Unlike the inner product case, the decoration of the duality pairing bracket 
notation will always be made explicit and never omitted.   

We also recall the canonical definition \citep{evans-book-2010} of some useful 
time-dependent spaces whose codomain is also a given Hilbert space $X$.  
With $X$ selected we consider a strongly measurable function 
$f:[0,T]\rightarrow X$.  Then $f\in L^2(0,T;X)$ means that 
\begin{equation*}
	\nrm{f}{L^2(0,T;X)} = %
	\left(\int_{0}^{T} \nrm{f(t)}{X}^2\, dt\right)^{1/2} < \infty,
\end{equation*}
whereas $f\in L^\infty(0,T;X)$ implies
\begin{equation*}
	\nrm{f}{L^2(0,T;X)} = %
	\esssup\limits_{0\leq t\leq T} \nrm{f(t)}{X} < \infty,
\end{equation*}
and $f\in C(0,T;X)$ means that
\begin{equation*}
	\nrm{f}{C(0,T;X)} = %
	\max\limits_{0\leq t \leq T}\nrm{f(t)}{X} < \infty.
\end{equation*} 

We now discuss those strongly measurable functions, $f:[0,T]\rightarrow X$, 
which possess weakly differentiability in time. The space $H^1(0,T;X)$ denotes 
the collection of functions, $f\in L^2(0,T;X)$, such that $\partial_t f$ 
exists, in the weak sense, and also resides in $L^2(0,T;X)$.  This is similar
to the usual definition of $H^1\left(\Omega\right)$, given above, and the norm
corresponding to this space is also similar; it is given by
\begin{equation*}
	\nrm{f}{H^1(0,T;X)} = \left( \int_{0}^{T} \nrm{f(t)}{X}^2 + %
	\nrm{\partial_t f(t)}{X}^2\, dt \right).
\end{equation*}
Likewise, $f\in C^k(0,T;X)$ implies that $f$ and its first $k$ weak derivatives 
in time, $\partial_t^j f$ for $j=1,2,\dots,k$, all reside in $C(0,T;X)$.

\subsection{Mesh elements and discrete operators}

For a fixed $h$, the mesh $\mathcal{T}_h$ is composed of simplices,
denoted $T\in \mathcal{T}_h$, and faces (edges in 2D) $e \in \partial
T$.  Let $\Gamma$ denote the complete set of faces of simplices
$T\in\mathcal{T}_h$; then $\Gamma$ can be written as the disjoint
union
\begin{equation*}
  \Gamma = \Gamma_{int} \cup \Gamma_{bd},
\end{equation*} 
where $e\in\Gamma_{int}$ if $e$ is an interior edge and $e\in\Gamma_{bd}$ if
$e$ is a boundary edge.

Let $f$ be a scalar or vector valued function and suppose $e$ is an
interior edge $e \in T_+ \cap T_-$ where $T_+$ and and $T_-$ are two
simplices with an arbitrary but fixed choice of labeling for the
pairing. We denote by $n_e$ the outward facing normal associated to
$T_+$. We use an explicit jump operator defined, for
$e\in\Gamma_{int}$, by
\begin{equation}\label{eqn:jump-operator}
[f] = f_+ - f_-
\end{equation}
where $f_+$ denotes $f$ restricted to $e\in T_+$ and $f_-$ denotes $f$
restricted to $e\in T_-$.  For an edge $e\in\Gamma_{bd}$ we have that
there exists only one $T_+ = T \in \mathcal{T}_h$ such that $e\in
\partial T$ and in this case we define
\begin{equation*}
[f] = f_+ .
\end{equation*}

\subsection{Boundary and initial conditions}

We assume homogeneous boundary conditions for the displacement and
pressures; though, as in \citep{ern-meunier-2009}, these conditions can
easily be generalized \citep{showalter-2000}.

\section{Coupled elliptic-parabolic problems as a setting for poroelasticity}
\label{sec:general-setting}

To consider the a posteriori error analysis of the generalized
poroelasticity equations~\eqref{eq:mpet}, we follow the general
framework for a posteriori error analysis for coupled
elliptic-parabolic problems presented by~\citet{ern-meunier-2009}. In
Section~\ref{sec:general-setting:basic} below, we briefly overview
this general framework and its application to Biot's equations. Next,
we show that \eqref{eq:mpet} can be addressed using this general
framework, also for the case where $J > 1$ under appropriate
assumptions on the transfer terms $T_{m\rightarrow n}$, in
Section~\ref{sec:general-setting:extension}. Based on the general
framework, Ern and Meunier derive and analyze several a posteriori
error estimators. These estimators, and their corresponding extensions
to the generalized poroelasticity equations, will be discussed in
Section \ref{sec:aposteriori}.

\subsection{The coupled elliptic-parabolic problem framework}
\label{sec:general-setting:basic}
The setting introduced by Ern and Meunier~\citep{ern-meunier-2009} for
coupled elliptic-parabolic problems provides a natural setting also
for generalized poroelasticity. The general coupled elliptic-parabolic
problem reads as: find $(u, p) \in H^1(0,T;V_a) \times
H^1(0,T;V_d)$ that satisfy (for almost every $t \in [0, T]$):

\begin{subequations}
  \label{eq:general}
  \begin{align}
    a(u, v) - b(v,p) &= \dual{f}{v}{V_a^{\ast} \times V_a},
    \quad \foralls v\in V_a,\label{eq:general-a}\\
    c(\partial_t p,q) + b(\partial_t u,q) + d(p,q) &=
    \dual{g}{q}{V_d^{\ast} \times V_d} .
    \quad\foralls q\in V_d.\label{eq:general-b}
  \end{align}
\end{subequations}
The data, $f$ and $g$ in~\eqref{eq:mpet}, are general and assumed to satisfy 
$f\in H^1(0,T;V_a^*)$ and $g \in H^1(0,T;V_d^*)$.  The initial pressure is
assumed to satisfy $p(0) \in V_d$. Moreover, it is assumed that
\begin{enumerate}
\item
  $V_a$ and $V_d$ are Hilbert spaces. 
\item
  $a: V_a \times V_a \rightarrow \R$ and $d: V_d \times V_d
  \rightarrow \R $ are symmetric, coercive, and continuous bilinear
  forms, thus inducing associated inner-products and norms (denoted by
  $\|\cdot\|_a$ and $\| \cdot \|_d$) on their respective spaces.
\item
  There exist Hilbert spaces $L_a$ and $L_d$ with $V_a \subset L_a$
  and $V_d \subset L_d$, where the inclusion is dense and such that
  $\nrm{f}{L_a} \ls \nrm{f}{a}$ for $f \in V_a$, and $\nrm{g}{L_d} \ls
  \nrm{g}{d}$ for $g \in V_d$.
\item
   $c : L_d \times L_d \rightarrow \R$ is symmetric, coercive and
  continuous; thereby defining an equivalent norm $\|\cdot\|_{c}$, on
  $L_d$.
\item
  There exists a continuous bilinear form $b: V_a \times L_d
  \rightarrow \R$ such that $|b(f, g)| \ls \|f\|_{a} \|g\|_{c}$ for
  $f \in V_a$ and $g\in L_d$.
\end{enumerate}

\begin{example}
  Biot's equations of poroelasticity (i.e.~\eqref{eq:mpet} for $J=1$
  fluid networks) fit the coupled elliptic-parabolic framework with
  \begin{equation*}
    V_a = \left[H_0^1\right]^d, \quad
    L_a = \left[L^2\right]^d, \quad
    V_d = H_0^1, \quad
    L_d = L^2.
  \end{equation*}
and 
\begin{align*}
  a(u,v) &= \inner{\sigma(u)}{\varepsilon(v)}, \\ 
  b(u, p) &= \inner{\alpha_1 p}{\Div u}, \quad 
  c(p,q) = \inner{c_1 p}{q}, \quad
  d(p,q) = \inner{\kappa_1 \Grad p}{\Grad q},
\end{align*}
with the standard (vector) $H^1_0$-inner product and norm on $V_a$ and
$V_d$, and $L^2$-inner product and norm on $L_a$ and $L_d$. It is
readily verifiable that the general conditions described above are
satisfied under these choices of spaces, norms, inner products and
forms~\citep{ern-meunier-2009}.  \changed{The existence and uniqueness of 
solutions to Biot's equations of poroelasticity ($J=1$ in \eqref{eq:mpet} and 
\eqref{eq:general} with the above bilinear forms) is now a classical result 
\cite{showalter-2000}.}
\end{example}

\subsection{Generalized poroelasticity as a coupled elliptic-parabolic problem}
\label{sec:general-setting:extension}
In this section, we derive a variational formulation of the
generalized poroelasticity equations~\eqref{eq:mpet} for
the case of several fluid networks (i.e.~$J \geq 1$) and show
how this formulation fits the general framework presented above. The
extension from Biot's equations to generalized poroelasticity is
natural in the sense it coincides with the original application of the
general framework to Biot's equations when $J=1$.  %
Suppose that the total number of networks $J$ is arbitrary
but fixed. We define the spaces
\begin{equation}
  \label{eq:spaces}
  V_{a} = \left[H_0^1\right]^d, \quad
  L_{a} = \left[L^2\right]^d, \quad
  V_{d} = \left[H_0^1\right]^{J}, \quad 
  L_{d} = \left[L^2\right]^{J}.
\end{equation}
We consider data such that $f \in H^1(0,T; L_{a})$ and $(g_1, g_2,
\dots, g_{J}) \in H^1(0,T; L_d)$ with given initial network pressures
determined by $p(0) \in V_d$. A standard multiplication, integration
and integration by parts yield the following variational formulation
of~\eqref{eq:mpet}: find $u \in H^1(0, T; V_a)$ and $p = (p_1, \dots,
p_J) \in H^1(0, T; V_d)$ such that for a.e.~$t \in (0, T]$:
\begin{subequations}
  \label{eq:weak}
  \begin{align}
  \inner{\sigma(u)}{\varepsilon(v)} {-} \ssum_{j=1}^J \inner{\alpha_j p_j}{\Div v}
  &= \inner{f}{v}_{},\label{eq:weak:a} \\
  \ssum_{j=1}^J \inner{\partial_t s_j p_j}{q_j} + \inner{\partial_t \alpha_j \Div u}{q_j} + \inner{\kappa_j \Grad p_j}{\Grad q_j} + \inner{T_j}{q_j}
  &= \ssum_{j=1}^J \inner{g_j}{q_j}_{}. \label{eq:weak:b}
  \end{align}
\end{subequations}
for $v \in V_a$, $q = (q_1, \dots, q_J) \in V_d$. As noted in 
\citep{ern-meunier-2009}, \eqref{eq:weak:a} holds up to time $t=0$ so that $u_0$ 
is determined by the initial data $p(0)$ and initial right-hand side $f(0)$.  %
By labeling the forms
\begin{subequations}
  \label{eq:forms}
  \begin{align}
  a(u, v) &= \inner{\sigma(u)}{\varepsilon(v)}, \label{eq:forms-a}\\
  b(u, p) &= \ssum_{j=1}^J \inner{\alpha_j p_j}{\Div u}, 
  \label{eq:b} \\
  c(p, q) &= \ssum_{j=1}^J \inner{s_j p_j}{q_j}, \label{eq:forms-c}\\
  d(p, q) &= \ssum_{j=1}^J \inner{\kappa_j \Grad p_j}{\Grad q_j} + \inner{T_j}{q_j},
  \label{eq:d}
\end{align}
\end{subequations}
we observe that the weak formulation~\eqref{eq:weak} of~\eqref{eq:mpet} 
takes the form~\eqref{eq:general} where $T_j$, in \eqref{eq:d}, is given by 
\eqref{eq:def:T_j}.  

\changed{%
\begin{rem}
Existence and uniqueness of solutions to \eqref{eq:weak}, with forms \eqref{eq:forms}, for the case 
of $J=2$ (Barenblatt-Biot) have been established \cite{showalter-2002}.  However, to our knowledge, 
a rigorous treatment of existence and uniqueness of solutions to \eqref{eq:mpet}, \eqref{eq:weak} 
for $J>2$ remains an open problem despite their otherwise successful use in applications 
\cite{vardakis-ventikos-2018,guo2019validation,tully-ventikos-2011,vardakis-ventikos2016}.
\end{rem}
}

We now show that the associated assumptions on these forms
and spaces hold, beginning with properties of the form $d$ in
Lemma~\ref{lemma:d} below.
\begin{lemma}
  \label{lemma:d}
  The form $d$ given by~\eqref{eq:d} defines an inner product over
  $[H^1_0(\Omega)]^J$ with associated norm 
  \begin{equation}
    \label{eqn:general-setting:extension-D-norm}
    \| q \|_d^2 = d(q, q) = \ssum_{j=1}^{J} \| \Grad q_j \|_{\kappa_j}^2 + | q |_{T}^2, \quad \foralls q \in V_d,
  \end{equation}
  where $|\cdot|_{T}$ is defined by~\eqref{eq:def:T}, which is such that
  \begin{equation}
    \label{eq:dnorm:bound}
    \|q\|_{d} \ls \|q\|_{H^1_0}, \quad \foralls q \in [H^1_0]^J,
  \end{equation}
  with inequality constant depending on $J$, $\kappa_{\max}$,
  $\gamma_{\max}$ and $\Omega$.
\end{lemma}
\begin{proof}
By definition~\eqref{eq:def:T_j} and the assumption of symmetric
transfer $\gamma_{ji} = \gamma_{ij} \geq 0$, we have 
\begin{equation}
  \label{eq:def:T}
  \ssum_{j=1}^J \inner{T_j}{q_j} =
  \ssum_{j=1}^J \ssum_{i=1}^J \inner{\gamma_{ji} (p_j - p_i)}{q_j} = \frac{1}{2} \ssum_{j=1}^J \ssum_{i=1}^J \inner{\gamma_{ji} (p_j - p_i)}{(q_j - q_i)} .
\end{equation}
Given $p$, $q\in V_d$, the bilinear form defined by \eqref{eq:def:T}, that is 
\[
\inner{p}{q}_{T} = \ssum_{j=1}^J \inner{T_j}{q_j},
\] 
is clearly symmetric and satisfies the requirements of a (real) semi-inner 
product on $L_{d} \times L_{d}$ in the sense of \citep{conway-book-1997}.
It follows that
\begin{equation}
  \label{eq:def:seminorm}
  |q|_{T} \equiv \inner{q}{q}_T^{1/2} = \left( \frac{1}{2} \ssum_{j=1}^J \ssum_{i=1}^J \inner{\gamma_{ji} (q_j - q_i)}{(q_j - q_i)} \right)^{1/2}
\end{equation}
defines a semi-norm on $L_{d} \times L_{d}$ and that the corresponding
Cauchy-Schwarz inequality holds. Using the triangle
inequality, the definition \eqref{eq:def:T}, the bounds for $\gamma_{ji}$ and the
Poincar\'e inequality, we have that
\begin{equation}
  \label{eq:def:seminorm-poincare}
  | q |_{T} \ls \| q \|_{[L^2]^J} \ls \| q \|_{[H^1_0]^J}
\end{equation}
with constant depending on $\gamma_{\max}$, $J$, and the domain via
the Poincar\'e constant. Under the assumption that $\kappa_{\min} >
0$, we observe that as a result $d$ defines an inner product and norm
on $[H^1_0]^J \times [H^1_0]^J$. Similarly, \eqref{eq:dnorm:bound}
holds with with constant depending on $\kappa_{\max}$ in addition to
$\gamma_{\max}$, $J$, and the domain $\Omega$.
\end{proof}
Lemma~\ref{lemma:d} will be used in the subsequent sections. %
We next show that the choices of spaces \eqref{eq:spaces} and
forms~\eqref{eq:forms} satisfy the abstract assumptions of the
framework as overviewed in Section \ref{sec:general-setting:basic},
and summarize this result in Lemma~\ref{lemma:mpet}.
\begin{lemma}
  \label{lemma:mpet}
  The problem \eqref{eq:weak}, arising from the equations of
  generalized poroelasticity \eqref{eq:mpet} with material parameters as in 
  Section~\ref{sec:introduction:material-parameters}, posed on the spaces
  \eqref{eq:spaces} with bilinear forms defined via \eqref{eq:spaces}
  is a coupled elliptic-parabolic problem and satisfy the assumptions
  set forth in \citep{ern-meunier-2009}.
\end{lemma}
\begin{proof}
  We consider each assumption in order. These are standard results,
  but explicitly included here for the sake of future reference.
  \begin{enumerate}
  \item
    $V_a$ and $V_d$ defined by \eqref{eq:spaces} are clearly Hilbert
    spaces with natural Sobolev norms $\|\cdot\|_{H^1_0}$.
  \item
    $a$ is symmetric, coercive on $V_a$ by Korn's inequality and the
    lower bounds on $\mu$, $2 \mu + d \lambda$, and continuous (with
    continuity constant depending on $\mu_{\max}$ and
    $\lambda_{\max}$). $d$ is clearly symmetric by the transfer
    symmetry assumption and~\eqref{eq:def:T}, coercive by $|\cdot|_{T}
    \geq e 0$ and the assumption that $\kappa_{\min} > 0$:
    \begin{equation*}
      d(q, q) \geq \ssum_{j=1}^{J} \inner{\Grad q_j}{\Grad q_j}_{\kappa_j}
      \geq \kappa_{\min} \ssum_{j=1}^{J} \| q_j\|_{H_0^1}^2 ,
    \end{equation*}
    and continuous by Lemma~\ref{lemma:d}.
  \item
    The embedding of $(V_a, \|\cdot \|_{a})$ into $L_a$ follows from
    Poincare's inequality and the coercivity of $a$ over $V_a$ and
    similarly for $V_d \hookrightarrow L_d$.
  \item
    $c$ is symmetric by definition, continuous over $L_d$ with
    continuity constant depending on $c_{\max}$, and coercive with
    coercivity constant depending on $c_{\min} > 0$.
  \item
    The form $b$ given by~\eqref{eq:b} is clearly bilinear and
    continuous on $V_{a} \times L_{d}$ as
    \begin{align*}
      |b(u,p)| &= %
      |\suml{j=1}{J}\dual{p_j}{\Div u}{\alpha_j}|	%
      \ls \nrm{u}{H^1}\left(\suml{j=1}{J}\nrm{p_j}{}^2\right)^{\frac12} %
      = \nrm{u}{H^1}\nrm{p}{[L^2]^J} 
      \ls \|u\|_{a} \|p\|_c
    \end{align*}
    by applying Cauchy-Schwarz and H{\"o}lder's inequality, with
    constant depending on $\alpha_{\min} > 0$ and the coercivity
    constants of $a$ and $c$. \qedhere
  \end{enumerate}
\end{proof}

In light of Lemma~\ref{lemma:mpet}, the generalized poroelasticity system 
\eqref{eq:weak} is of coupled elliptic-parabolic type and takes the form 
of \eqref{eq:general} with bilinear forms defined by \eqref{eq:forms}. 
\begin{cor}\label{cor:energy-est}
The following energy estimates hold for almost every $t\in [0,T]$
\begin{align*}
\nrm{u(t)}{a}^2 + &\ssum\limits_{j=1}^J s_j \nrm{p_j(t)}{}^2 + %
\ssum\limits_{j=1}^{J}\int_{0}^{t}\kappa_j\nrm{\Grad p(s)}{}^2\,\text{d}s + %
\ssum\limits_{i=1}^{J}\ssum\limits_{j=1}^{J}\int_{0}^{t} %
\gamma_{ij}\nrm{p_j(s) - p_i(s)}{}^2\, \text{d}s \\ %
&\ls \left(\sup\limits_{s\in[0,T]} \nrm{f(s)}{} + 
\int_{0}^{T}\nrm{\partial_t f(s)}{}\, \text{d}s\right)^2 + 
\int_{0}^{T} \nrm{g(s)}{}^2\,\text{d}s + \nrm{u_0}{a}^2 + %
\ssum\limits_{j=1}^{J} s_j\nrm{p_0}{}^2
\end{align*} 
\end{cor}
\begin{proof}
The proof follows directly from Lemma~\ref{lemma:mpet}, the corresponding 
energy estimates for coupled elliptic-parabolic systems 
\cite[Prop. 2.1]{ern-meunier-2009} and the definition of the norms arising 
from the forms \eqref{eq:forms}.  Moreover, the proportionality constant in 
the estimates is independent of all material parameters and the number of 
networks.
\end{proof}

\begin{rem}
The elliptic-parabolic MPET energy estimates, of Corollary~\ref{cor:energy-est}, 
are similar to those of the total pressure formulation \cite[Theorem 3.3]{lee2019mixed} when 
the second Lam{\'e} coefficient, $\lambda$, is held constant in the latter.  
The primary difference is that \citep{lee2019mixed} separates the estimates of 
$u$ from that of the solid pressure, $\lambda \Div u$, by including the latter 
term into a `total pressure' variable.  This allows for $\nrm{u}{1}$ to be 
estimated directly, in \citep{lee2019mixed}, regardless of the value of 
$\lambda$ used in the definition of $\nrm{u}{a}$.
\end{rem}

\begin{rem}
The conditions of Section~\ref{sec:general-setting:basic}, i.e.~conditions 
(2)-(5), can place restrictions on the generalized poroelastic setting.  As an 
example, the assumption of a vanishing storage coefficient has appeared in the 
literature as a modeling simplification \citep{lotfian2018,young-riviere-2014}.
However, the coercivity requirement of condition (4) precludes the use of 
a vanishing specific storage coefficient, $s_j$ in \eqref{eq:forms-c}, for any 
network number $j=1,2,\dots,J$.  Care should be taken to ensure 
that any modeling simplifications produce forms that satisfy the conditions 
of Section~\ref{sec:general-setting:basic} in order for the results of 
Corollary~\ref{cor:energy-est} and 
Section~\ref{sec:discretizations-and-apriori-est} to hold.
\end{rem}

\section{Discretization and a priori error estimates}
\label{sec:discretizations-and-apriori-est}
\subsection{An Euler-Galerkin discrete scheme}
\label{sec:general-setting:discrete-scheme}

We now turn to an Euler-Galerkin discretization of~\eqref{eq:mpet} in
the context of such discretizations of coupled elliptic-parabolic
problems in general~\citep{ern-meunier-2009}. We employ an implicit
Euler discretization in time and conforming finite elements in
space.

We consider a family of simplicial meshes
$\left\{\mathcal{T}_{h}\right\}_{h>0}$ with $h$ a characteristic mesh
size such as the maximal element diameter
\begin{equation*} 
h = \max\left\{h_K = \textrm{diam}(K) \,| \, K\in \mathcal{K}_h\right\}.
\end{equation*}
Furthermore, let $\{V_{a,h}\}_{h}$ and $\{V_{d,h}\}_{h}$ denote two
families of finite dimensional subspaces of $V_{a}$ and $V_{d}$, as in
\eqref{eq:spaces}, respectively, defined relative to
$\{\mathcal{T}_h\}_h$.  For a final time $T > 0$ we let $0 = t_0 < t_1
< \dots < t_N = T$ denote a sequence of discrete times and set $\tau_n
= t_n - t_{n-1}$. For functions and fields, we use the superscript $n$
to refer values at time point $t_n$. We also utilize the discrete time
differential notation $\delta_t$ where
\begin{equation}
  \delta_t u_h^n = \tau_n^{-1} \left(u_h^n - u_h^{n-1} \right).
\end{equation}

With this notation, the discrete problem is to seek $u_{h}^{n} \in
V_{a,h}$ and $p_{h}^{n} = (p_{1,h}^n, p_{2,h}^n ,\dots,p_{J,h}^n) \in
V_{d,h}$ such that for all time steps $t_n$ with $n \in
\left\{1, 2, \dots, N \right\}$:
\begin{subequations}
  \label{eq:discrete}
  \begin{alignat}{3}
    a(u_h^n,v_h) - b(v_h,p_{h}^n) &= \inner{f_h^n}{v_h}  && \quad \foralls v_h \in V_{a,h},\label{eq:discrete:a}\\
    c(\delta_t p_{h}^n,q_{h}) + b(\delta_t u_h^n,q_{h}) + d(p_{h}^n,q_{h}) &= \inner{g_{h}^n}{q_{h}} && \quad\foralls q_{h} \in V_{d,h},\label{eq:discrete:b}
  \end{alignat}
\end{subequations}
where the spaces and forms are defined by~\eqref{eq:spaces}
and~\eqref{eq:forms}.  The right-hand sides, above, express the inner product 
of the discrete approximations $f_h^{n}\in L_{a,h}$, to $f$ and 
$g_{h}^{n}\in L_{d,h}$, to $g$, at time $t_n$. By Lemma~\ref{lemma:mpet} and \cite[Lemma 2.1]{ern-meunier-2009}, the discrete 
system~\eqref{eq:discrete} is well-posed.

\subsection{A priori error estimates}
Now, let $V_{a,h}$ and $V_{d,h}$ be spatial discretizations arising
from continuous Lagrange elements of order $k_a$ and $k_d$,
respectively, where $k_d = k_a - 1$; this relation on relative degree
results directly from the framework hypotheses \cite[Section
  2]{ern-meunier-2009}. Let $\mathcal{P}_k(T)$ denote polynomials of order 
$k$ on a simplex $T \in \mathcal{T}_h$. We consider the continuous Lagrange 
polynomials of order $k_a$ and $k_{a-1}$ defined by 
\begin{align}
V_{a, h} &= \left\{ v_h \in C^0(\bar{\Omega}) \,|\, v_{h|T} \in \mathcal{P}_{k_a}(T) \text{ for every } T\in \mathcal{T}_h\right\},\label{eqn:disc-spaces:vah}\\ 
V_{d, h} &= \left[ \left\{ q_h \in C^0(\bar{\Omega}) \,|\, q_{h|T} \in \mathcal{P}_{k_{a}-1}(T) \text{ for every } T\in \mathcal{T}_h\right\}\right]^J \label{eqn:disc-spaces:vdh}
\end{align}
as the discrete spaces for the displacement and network pressures, 
respectively.  
When $J=1$ this choice coincides with the previous \citep{ern-meunier-2009} 
discretization considered for Biot's equations.  %

The general framework stipulates that three hypotheses
\citep[Section~2.5]{ern-meunier-2009}, restated here for completeness,
should be satisfied for the discretization.
\begin{hyp}\label{general-setting-discrete-scheme:hyp-one}
There exists positive real numbers, denoted $s_a$ and $s_d$, and subspaces, 
$W_a \subset V_a$ and $W_d \subset V_d$ equipped with norms $\nrm{\cdot}{W_a}$ 
and $\nrm{\cdot}{W_d}$, such that the following estimates hold 
independently of $h$ 
\begin{align}
\forall v \in W_a,\quad & %
\inf_{v_h \in V_{a,h}}\nrm{v - v_h}{a} \ls h^{s_a}\nrm{v}{W_a}, %
\label{eqn:general-setting:discrete-scheme:approx-hyp-a}\\
\forall q \in W_d,\quad & %
\inf_{q_h \in V_{d,h}}\nrm{q - q_h}{d} \ls h^{s_d}\nrm{q}{W_d}. %
\label{eqn:general-setting:discrete-scheme:approx-hyp-b}
\end{align} 
\end{hyp}
\begin{hyp}\label{general-setting-discrete-scheme:hyp-two}
There exists a real number $\delta$ such that for every $r\in L_d$, the unique 
solution $\phi \in V_d$ for the dual problem
\[
	d(q,r) = c(r,q)\quad \forall q \in V_d,
\]
is such that there exists $\phi_h \in V_{d,h}$ satisfying
\[
	\nrm{\phi-\phi_h}{d} \ls h^{\delta}\nrm{r}{c} .
\]
\end{hyp}
\begin{hyp}\label{general-setting-discrete-scheme:hyp-three}
$s_a$ = $s_d + \delta$
\end{hyp}

We now state the primary result of this section.
\begin{lemma}
  \label{lemma:discrete}
  The discrete two-field variational formulation of the MPET
  equations~\eqref{eq:discrete} with the choice of discrete spaces
  $V_{a, h}$~\eqref{eqn:disc-spaces:vah} and $V_{d,
    h}$~\eqref{eqn:disc-spaces:vdh} satisfy the elliptic-parabolic
  framework hypotheses
  \ref{general-setting-discrete-scheme:hyp-one}--\ref{general-setting-discrete-scheme:hyp-three},
  above.
\end{lemma}
\begin{proof}
Choose $s_a = k_a$ and $s_d = k_{a-1}$.  Then the conditions of 
Hypothesis~\ref{general-setting-discrete-scheme:hyp-one} 
follow, as in \citep{ern-meunier-2009}, from choosing  
$W_a = \left[H_0^1 \cap H^{k_{a+1}}(\mathcal{T}_h)\right]^3$ and
$W_d = \left[H_0^1 \cap H^{k_a}(\mathcal{T}_h)\right]^J$ where 
$H^{k}(\mathcal{T}_h)$ denotes the broken Sobolev
space of order $k$ on the mesh $\mathcal{T}_h$.  The estimate
\eqref{eqn:general-setting:discrete-scheme:approx-hyp-a} follows,
without extension, directly from classical results in approximation
theory \citep{guermond-book-2004}; precisely as discussed in
\citep{ern-meunier-2009}.  Similarly, the estimate
\eqref{eqn:general-setting:discrete-scheme:approx-hyp-b} follows from
the properties of $d$, standard interpolation estimates
\citep{guermond-book-2004}, and the product structure of $V_d$ and 
$V_{d,h}$.  

For hypothesis~\ref{general-setting-discrete-scheme:hyp-two}, we use
the elliptic regularity, c.f.~standard well posedness and interior
regularity arguments in \cite[Chp. 6]{evans-book-2010}, of the
solution to the coupled linear diffusion-reaction equation of finding
$\phi \in V_d$ such that
\begin{equation*}
  A\phi + \Gamma \phi = R\,\quad \text{for }R\in L_d,
\end{equation*}
where $A$ is the $J \times J$ diagonal Laplacian matrix
\begin{equation}
\left( 
\begin{array}{cccc}
-\Delta & 0 	  & \dots & 0 \\
0 	& -\Delta & \dots & 0 \\
\vdots	&  \ddots & \ddots& \vdots \\
0	& 0	  & 0	  & -\Delta
\end{array}
\right),
\end{equation}
and $\Gamma$ is a matrix composed of the transfer coefficients
$\gamma_{ij}$: $\Gamma_{ii} = \sum_{j} \gamma_{ji}$, and $\Gamma_{ij}
= - \gamma_{ij}$ for $j \not = i$. %
It follows from Lemma~\ref{lemma:d}, and $\Gamma$ symmetric and positive-semi 
definite, that %
the solution $\phi \in V_d$ to the dual problem
\begin{equation}
  d(w,\phi) = c(R,w) = %
  \dual{R}{w}{L_d',V_d},\quad\text{for all } w\in V_d,
\end{equation} 
lies in $\left[H^2\right]^{J}$ with $\nrm{\phi}{H^2} \ls \nrm{R}{L_d}
= \nrm{R}{c}$. Using this and standard interpolation results we have
$\phi_{h} \in V_{d,h}$ with
\begin{equation*}
	\nrm{\phi-\phi_{h}}{d} \ls h^{\delta}\nrm{R}{c} = h^{\delta}\nrm{R}{L_d},
\end{equation*}
where $\delta = 1$; exactly as in \citep{ern-meunier-2009}.  Finally, with 
$\delta = 1$ and the choices $s_a = k_a$ and $s_d = k_{a}-1$, 
hypothesis~\ref{general-setting-discrete-scheme:hyp-three} also holds.  
\end{proof}

A priori estimates for the Euler-Galerkin
discretization~\eqref{eq:discrete} of the generalized poroelasticity
equations~\eqref{eq:mpet} then follow directly from \cite[Theorem
  3.1]{ern-meunier-2009}. These estimates will be used in the
a posteriori analysis and are restated from \citep{ern-meunier-2009},
subject to the extended spaces and forms of \eqref{eq:spaces} and
\eqref{eq:forms}.
\begin{cor}[A priori estimates for generalized poroelasticity]
Let 
$I_n = [t_{n-1},t_n]$ denote the $n^{th}$ time sub-interval of $[0,T]$ 
for $n=1,2,\dots N$ of length
$\tau_n = t_n - t_{n-1}$. Suppose the exact solution $(u,p)$ to 
\eqref{eq:weak} satisfies 
$u\in C^1(0,T;W_a) \cap C^2(0,T;V_a)$ and 
$p\in C^1(0,T;W_d)\cap C^2(0,T;L_d)$, where $W_a$ and $W_d$ are 
given above with $V_a$ and $L_d$ as in 
\eqref{eq:spaces}.  It is also assumed that
the initial data satisfies 
\begin{equation*}
\nrm{u_0 - u_{0,h}}{a} \ls h^{k_a}\nrm{u_0}{W_a} \quad \text{and} \quad 
\nrm{p_{0} - p_{0,h}}{c} \ls h^{k_a}\nrm{p_{0}}{W_d} .
\end{equation*}
Define 
\begin{align*}
C_1^n(u,p) &= \nrm{\partial_t p(s)}{L^{\infty}(I_n;W_d)}^2 + %
	\nrm{\partial_t u(s)}{L^{\infty}(I_n;W_a)}^2,\\
C_2^n(u,p) &= \nrm{\partial_{tt}^2 p(s)}{L^{\infty}(I_n;L_d)}^2 + %
	\nrm{\partial_{tt}^2 u(s)}{L^{\infty}(I_n;V_a)}^2,\\
C^n(f,g)   &= \nrm{f^n - f_h^n}{a}^2 + \tau_n\nrm{g^n - g_h^n}{d}^2,\\
C(u_0,p_0) &= \nrm{u_0}{W_a}^2 + \nrm{p_0}{W_d}^2.
\end{align*}
Setting, for simplicity, $s=k_a$ then $s=k_d+1$, by the selection of the 
discrete spaces, and we have that for each $n\in\left\{1,2,\dots,N\right\}$
\begin{equation}
\begin{split}
  \nrm{u^n - u_h^n}{a}^2 + \nrm{p^n - p_{h}^n}{c}^2  
  \ls h^{2s}C(u_0,p_0) + \suml{m=1}{n}C^{m}(f,g)  +  
  \suml{m=1}{n}\tau_m h^{2s}C_1^m(u,p) \\ 
  \quad\quad + \suml{m=1}{n}\tau_m^3 C_2^m(u,p) + h^{2s}\left(\nrm{u^n}{W_a}^2 + 
  \nrm{p^n}{W_d}^2\right),
\end{split}
\end{equation}
and
\begin{equation}
\begin{split}
\suml{m=1}{n}\tau_m \nrm{p^m-p_h^m}{d}^2 \ls h^{2s}C(u_0,p_0) 
  + \suml{m=1}{n}C^{m}(f,g) + \suml{m=1}{n}\tau_m h^{2s}C_1^{m}(u,p) \\ 
  \quad\quad+ \suml{m=1}{n}\tau_m^3 C_2^m(u,p) 
  + \suml{m=1}{n}\tau_m h^{2s-1}\nrm{p^m}{W_d}^2 .
\end{split}
\end{equation}
\end{cor}

\section{A posteriori error estimation for generalized poroelasticity}
\label{sec:aposteriori}
We now turn to discuss the implications to a posteriori error
estimates for generalized poroelasticity as viewed through the lens of
the coupled elliptic-parabolic problem framework. Our focus is to
derive, apply and evaluate residual-based error estimators and
indicators in the context of generalized poroelasticity. We will
therefore present an explicit account of abstractly defined quantities
presented in \cite[Sec.~4.1]{ern-meunier-2009}, including e.g.~the
Galerkin residuals, applied in our context.

\subsection{Time interpolation}
\label{sec:aposteriori:time-interp}

We \changed{now recall} additional notation for time interpolation 
\changed{(from \cite[Sec. 4.1]{ern-meunier-2009})}, and rewrite
\eqref{eq:discrete}. Let $u_{h\tau}$
denote the continuous and piecewise linear function in time,
$u_{h\tau}\in H^1(0,T;V_{a,h})$ such that $u_{h\tau}(t_n) = u_{h}^n$.
Similarly $(p_{1h\tau},p_{2,h\tau},\dots, p_{J,h\tau}) = p_{h\tau}$ is
defined by $p_{h\tau}(t_n) = p_h^n$ and extended linearly in time. As
a result, $\partial_t u_{h\tau}$, $\partial_t p_{h\tau}$ are defined
for almost every $t\in(0,T)$. Define the corresponding
continuous, piecewise linear in time variants of the data, $f_{h\tau}$
and $g_{h\tau}$, by the same approach; i.e.~$f_{h\tau}(t_n) = f_h^n$
and $g_{h\tau}(t_n) = g_{h}^n$.

Before rephrasing \eqref{eq:discrete} using the time-interpolated
variables we define piecewise constant functions in time for the
pressure and right-hand side data. These are defined as
$\pi^0p_{h\tau} = p_h^n$ and $\pi^0 g_{h\tau} = g_{h}^n$ on $I_n =
(t_{n-1}, t_{n})$.  Using the above notation, the discrete scheme for
almost every $t\in (0,T)$ becomes
\begin{subequations}
  \label{eq:discrete-time-interpolated}
  \begin{align}
    a(u_{h\tau},v_h) - b(v_h,p_{h\tau}) &= %
    \inner{f_{h\tau}}{v_h},%
    \quad\quad \forall v\in V_{a,h},
    \label{eq:discrete-time-interpolated:a}\\
    c(\partial_{t}p_{h\tau},q_{h}) %
    + b(\partial_t u_{h\tau},q_{h})%
    + d(\pi^0p_{h\tau},q_{h}) &= %
    \inner{\pi^0g_{h\tau}}{q_{h}},%
	  \quad\forall q_{h} \in \changed{V_{d,h}}.
    \label{eq:discrete-time-interpolated:b}
  \end{align}
\end{subequations}

\begin{rem}
Using the linear time interpolations defined above, such as $u_{h\tau}$ or 
$p_{h\tau}$, we have the following identity
\[
\partial_t u_{h\tau} = \delta u_h^n \quad \text{for all } t\in I_n = (t_{n-1},t_{n}),
\]
so that the left-hand sides of and \eqref{eq:discrete:b} and 
\eqref{eq:discrete-time-interpolated:b} are identical.  However, as noted 
in \citep{ern-meunier-2009}, the interpolants of the data, $f_{h\tau}$ and 
$g_{h\tau}$, are continuous and facilitate the definition of the continuous-time 
residuals \eqref{eq:residual:a} and \eqref{eq:residual:d}.
\end{rem}

\subsection{Galerkin residuals}
\label{sec:aposteriori:gal-residuals}

The Galerkin residuals \cite[Section 4.1]{ern-meunier-2009} are functions of 
time whose co-domain lies in the dual of either $V_a$ or $V_d$.  More 
specifically, the residuals are continuous, 
piecewise-affine functions $\mathcal{G}_a: [0,T]\rightarrow V_a^*$ and 
$\mathcal{G}_d: [0,T]\rightarrow V_d^*$.  In our context, of generalized 
poroelasticity, the Galerkin residual $\mathcal{G}_a$ is, given any $v\in V_a$, 
defined by the relation
\begin{equation}
  \label{eq:residual:a}
  \inner{\mathcal{G}_a}{v}
  \equiv \inner{f_{h\tau}}{v} - a(u_{h\tau},v) + b(v,p_{h\tau})
  = \inner{f_{h\tau}}{v} - a(u_{h\tau}, v) + \suml{j=1}{J} \inner{p_{j,h\tau}}{\Div v}_{\alpha_j}.
\end{equation}
Similarly, $\mathcal{G}_d$ is, 
given any $q=(q_1, q_2, \dots, q_{J}) \in V_d$, defined by the relation
\begin{equation}
  \begin{split}
  \label{eq:residual:d}
  \inner{\mathcal{G}_d}{q} 
  &\equiv\inner{\pi^0 g}{q} - c(\partial_t p_{h\tau},q) - 
  b(\partial_t u_{h\tau},q) - d(\pi^0 p_{h\tau},q) \\ 
  &=
  \suml{j=1}{J}\inner{\pi^0 g_{j,h\tau}}{q_j} 
  - \suml{j=1}{J}\inner{s_j \partial_t p_{j,h\tau}}{q_j} 
  - \suml{j=1}{J}\inner{\alpha_j \partial_t\Div u_{h\tau}}{q_j} \\
  &\qquad - \suml{j=1}{J}\inner{\kappa_j \Grad p_j}{\Grad q_j} + 
  \frac12 \suml{j=1}{J}\suml{i=1}{J} %
  \inner{\gamma_{ij}(p_{j,h\tau}-p_{i,h\tau})}{(q_j - q_i)},
  \end{split}
\end{equation}
Again, we note that \eqref{eq:residual:a} and \eqref{eq:residual:d}
generalize the corresponding \cite[Sec.~4.1]{ern-meunier-2009}
residuals for the case of single-network poroelasticity studied therein.

\subsection{Data, space and time estimators}
\label{sec:aposteriori:estimators}
The general coupled elliptic-parabolic problem framework gives a
posteriori error estimates, and in particular so-called data, space
and time estimators for the discrete solutions. In our context of
generalized poroelasticity, these can be expressed explicitly as
follows. We have terms for the data $f$ and $g$ given by
\begin{equation*}
\mathcal{E}(f,g) = \nrm{g - \pi^0 g_{h\tau}}{L^2(0,T;V_a^{\ast})}^2 + %
\left(\nrm{f-f_{h\tau}}{L^{\infty}(0,T;V_a^{\ast})} + %
\nrm{\partial_t(f-f_{h\tau})}{L^{1}(0,T;V_a^{\ast})}\right)^2,
\end{equation*}
and the framework data, space and time estimators are defined, respectively, as 
\begin{align}
\mathcal{E}_{\text{data}} &= \nrm{u_0 - u_{0h}}{a}^2 + \nrm{p_0 - p_{0h}}{c}^2%
	 + \mathcal{E}(f,g),\label{sec:aposteriori:estimators:dat}\\
\mathcal{E}_{\text{space}} &= \nrm{\mathcal{G}_d}{L^2(0,T;V_d^{\ast})}^2 + %
	\left(\nrm{\mathcal{G}_a}{L^{\infty}(0,T;V_a^{\ast})} + %
	\nrm{\partial_t \mathcal{G}_a}{L^1(0,T;V_a^{\ast})}\right)^2,%
	\label{sec:aposteriori:estimators:spc} \\
	\changed{\mathcal{E}_{\text{time}}} &= \changed{\nrm{p_{h\tau}-\pi^0 p_{h\tau}}{L^2(0,T;V_d)}^2}
	\label{sec:aposteriori:estimators:tim},
\end{align}
where we recall that the norms are now defined according to the
extended generalized poroelasticity spaces \eqref{eq:spaces} and
forms~\eqref{eq:forms}.  %
%
The following a posteriori error estimate for the general MPET equations holds:
\begin{prop}
  \label{eqn:err-upper-bound}
  For every time $t_n$, with $n\in\left\{1,2,\dots,N\right\}$, the following 
  inequality holds 
  \begin{align}
    &\nrm{u - u_{h\tau}}{L^{\infty}(0,t_n;V_a)}^2 + %
    \nrm{p-p_{h\tau}}{L^{\infty}(0,t_n;L_d)}^2 +%
    \nrm{p-p_{h\tau}}{L^2(0,t_n;V_d)}^2 \\ &\qquad + %
    \nrm{p-\pi^0 p_{h\tau}}{L^2(0,t_n;V_d)}^2 \ls %
    \mathcal{E}_{\text{data}} + \mathcal{E}_{\text{space}} %
    + \mathcal{E}_{\text{time}}
  \end{align}
\end{prop}
\begin{proof}
The proof follows from \cite[Thm.~4.1]{ern-meunier-2009} and the arguments of 
Section \ref{sec:general-setting:extension}. 
\end{proof}

\begin{rem}
	\changed{
	Note \cite[eqn.~(4.10)]{ern-meunier-2009} that \eqref{sec:aposteriori:estimators:tim} is equivalent to %
	\[
		\mathcal{E}_{\text{time}} = \frac13 \suml{m=1}{N}\tau_{m}\nrm{p_{h}^m - p_{h}^{m-1}}{d}^2,\quad m\in\left\{1,2,\dots,N\right\}.
	\]%
	The above expression will be used in Section \ref{sec:adaptivity:indicators} and follows, via the definition of $p_{h\tau}$ and $\pi^0$, from the calculation
	\begin{align*}
		\nrm{p_{h\tau}-\pi^0 p_{h\tau}}{L^2(0,T;V_d)}^2 &= \sum\limits_{m=1}^N \int_{t_{m-1}}^{t_m} \left(\frac{\xi-t_{m-1}}{\tau_m}\right)^2\nrm{p_h^m - p_h^{m-1}}{d}^2\,d\xi\\%
		&= \sum\limits_{m=1}^N \tau_m^{-2} \nrm{p_h^m - p_h^{m-1}}{d}^2 \int_{t_{m-1}}^{t_m}  \left(\xi-t_{m-1}\right)^2\,d\xi %
	\end{align*}%
	}
\end{rem}

\subsection{Element and edge residuals}
\label{sec:aposteriori:residuals}
In this section we state the definition of the element and edge
residuals (c.f.~\cite[Sec.~4.1]{ern-meunier-2009}) adapted to
generalized poroelasticity. We then define from these residuals a set
of a posteriori error indicators. These indicators can be used to
bound the Galerkin residuals defined in
Section~\ref{sec:aposteriori:gal-residuals}. The a posteriori error
indicators defined in this section will be used to carry out adaptive
refinement for the numerical studies in
Section~\ref{sec:numerical-experiments}.

\subsubsection{Element and edge residuals for the momentum equation}
\label{sec:aposteriori:residuals:momentum}
The residuals associated with the displacement are derived from the
Galerkin residual~\eqref{eq:residual:a}. We give them explicitly here
for the sake of clarity and to facilitate implementation. For $v \in
V_a$ and at time $t_n$, with $n \in \{1, 2, \dots, N \}$, we have
\begin{align*}
  \inner{\mathcal{G}^m_a}{v}  %
  &= \suml{K \in \mathcal{T}_h}{} \left( \dual{f^m_{h\tau}}{v}{K} - %
  \dual{\sigma(u^m_{h\tau})}{\epsilon(v)}{K} %
  + \suml{j=1}{J}\dual{\alpha_j p^m_{j,h\tau}}{\Div v}{K}\right), 
\end{align*}
where the notation $\dual{f}{g}{K} = \int_{K} f\,g\,dx$ denotes local integration
over a simplex $K \in \mathcal{T}_h$ and we have used that $u^n_{h \tau} = u^n_{h}$ and 
$p^n_{j, h \tau} = p^n_{j, h}$ for every $n \in \{1, 2, \dots, n \}$.  
Integrating the above by parts over each $K\in \mathcal{T}_h$ gives
\begin{equation}\label{sec:apriori:momentum-residual}
  \inner{\mathcal{G}^n_a}{v} = \suml{K\in\mathcal{T}_h}{}\dual{R^n_{uh,K}}{v}{K} %
  + \suml{e\in \Gamma_{int}}{}\dual{J^n_{uh,e}}{v}{e}, 
\end{equation}
where $\Gamma_{int}$ denotes the set of interior edges, $\inner{f}{g}_e$ denotes 
integration over the edge $e$ and where
\begin{equation}
  \label{eqn:momtm-residual-elem}
  R^n_{uh,K} = f^n_{h|K} + \Div \sigma(u^n_{h})_{K} - \suml{j=1}{J} \alpha_j \Grad p^n_{j,h|K},
\end{equation}
where the additional subscript denotes the restriction $K$. To define the term $J^n_{uh}$ 
above we use the standard notation, of \eqref{eqn:jump-operator}, and define 
\begin{equation}
  \label{eqn:momtm-residual-jump-simp}
  J^n_{uh,e}  
  = - \left[\sigma(u^n_{h})\right]_e n_e, 
\end{equation}
where $e$ is an edge and $n_e$ is the fixed choice of outward facing normal to 
that edge. The corresponding time-shifted local residual and jump operators are then
\begin{equation}\label{eqn:momtm-residual-elem-timeshift}
\delta_t R_{uh,K}^n = \tau_n^{-1}(R_{uh,K}^n - R_{uh,K}^{n-1}),\quad %
\delta_t J_{uh,e}^n = \tau_n^{-1}(J_{uh,e}^n - J_{uh,e}^{n-1}).
\end{equation}
To close, we note that the conditions in \citep{ern-meunier-2009} on the 
jump operator, $J_{uh,e}$ above, are general and other choices satisfying 
the abstract requirements can be used if desired.

\subsubsection{Element and edge residuals for the mass conservation equation}
\label{sec:aposteriori:residuals:mass}
The residuals associated with the network pressures are derived from
the Galerkin residual $\mathcal{G}_d$~\eqref{eq:residual:d}.
Integrating the diffusion terms by parts, over $T\in\mathcal{T}_h$,
gives
\begin{equation}\label{sec:apriori:mass-residual}
	\inner{\mathcal{G}_d}{q} = %
	\suml{K\in\mathcal{T}_h}{}\dual{R_{ph,K}^n}{q}{K} + %
	\suml{e\in\Gamma_{int}}{}\dual{J_{ph,e}^n}{q}{e}.
\end{equation}
In the context of the extended multiple-network poroelasticity framework the 
strong form of the mass conservation residual, $R_{ph,K}^n \in L_d$ of 
\eqref{sec:apriori:mass-residual}, has a $j^{th}$ component, for 
$j\in\left\{1,2,\dots,J\right\}$, with 
\begin{align}\label{eqn:mass-residual-elem}
  \left\{R^n_{ph,K}\right\}_j &= g^n_{j,h|K} - s_j\delta_t p^n_{j,h|K} 
  - \alpha_j (\Div \delta_t u^n_{h})_K 
  + (\Div \kappa_j \Grad p^n_{j,h})_K - T_{j,h|K}, 
\end{align}
recalling that $T_j$ is given by \eqref{eq:def:T_j}, and $T_{j, h}$ is
its discrete analogue. In \eqref{eqn:mass-residual-elem}, we have also
used that $\partial_t p_{j,h\tau}^n = \delta_t p_{j,h}^n =
\tau_{m}^{-1}(p_{j,h}^n - p_{j,h}^{n-1})$, $\partial_t u_{h,\tau}^n =
\delta_t u_{h}^n$, $p^n_{j,h\tau} = p_{j,h}^n$ and $u^n_{j,h\tau} =
u^n_{j,h}$. The corresponding jump term $J_{ph,e}^n$ for
$e\in\Gamma_{int}$ has $j^{th}$ component
\begin{equation}\label{eqn:mass-residual-jump}
  \left\{J_{ph,e}^n\right\}_j = -\left[\kappa_j \Grad p_{j,h}^n\right]_e\cdot n_e,
\end{equation}
and we once more remark that other jump operators satisfying the abstract 
conditions in \citep{ern-meunier-2009} can also be considered. We also have the 
analogous time-shifted versions of the above, $\delta_t R^n_{ph,K}$ and 
$\delta_t J^n_{ph,K}$, just as in %
\eqref{eqn:momtm-residual-elem-timeshift}.  

\subsection{Error indicators in space and time} 
We now define the element-wise error indicators; these indicators will
inform the construction of the a posteriori error indicators of
Section \ref{sec:adaptivity:indicators}. In turn, these indicators
will form the foundation of the adaptive refinement strategy of
section \ref{sec:adaptivity}. Specifically, we define element-wise
indicators denoted $\eta_{u,K}^n$ and $\eta_{p,K}^n$ such that the
following equalities hold for all $v \in V_a$ and $q\in V_d$
\begin{equation}
  \label{eqn:indicator-galerkin-equivalence}
  \inner{\mathcal{G}_a^n}{v} = \suml{K\in\mathcal{T}_h}{} \inner{\eta_{u,K}^n}{v},
  \quad \inner{\mathcal{G}_d^n}{q} = \suml{K\in\mathcal{T}_h}{} \inner{\eta_{p,K}^n}{q}.
\end{equation}

First, we define the following local error indicator associated with the
momentum equation:
\begin{align}
\eta_{u,K}^n &= h_K^2\nrm{R_{uh,K}}{K}^2 + h_K%
	\suml{e\in\partial K}{}\nrm{J_{uh,e}^n}{e}%
	\label{eqn:space-err-indic-momtm}\\
	&= h_K^2\nrm{f_h^n + \Div \sigma(u_h^n) %
	- \suml{j=1}{J}\alpha_j \Grad p_{j,h}^n}{K}^2
	+ h_K\suml{e\in\partial K}{}
	\nrm{\left[\sigma(u_h^n) \right]_e n_e}{e}^2, \nonumber 
\end{align}
where the norms $\nrm{\cdot}{K}$ and $\nrm{\cdot}{e}$ represent the
usual $L^2$, or or d-dimensional $L^2$, norm over a simplex, $K$, and
edge, $e$, respectively.  Likewise, the local error indicators
associated with the mass conservation equations are
\begin{equation}
\begin{split}
  \eta_{p,T}^n &= h_K^2\nrm{R_{ph,K}^n}{K}^2 + %
  h_K\suml{e\in\partial K}{}\nrm{J_{ph,e}^n}{e}%
  \label{eqn:space-err-indic-mass}\\
  &= h_K^2\suml{j=1}{J}\nrm{g_{j,h}^n - s_j\delta_t p_{j,h}^n - %
    \alpha_j \Div\delta_t u_{h}^n + k_j \Delta p_{j,h}^n - %
    \frac12 \suml{i=1}{J}\gamma_{ij}(p^n_{j,h}-p^n_{i,h})}{K}^2 \\
  & \quad\qquad + h_K\suml{e\in\partial K}{}%
  \nrm{\suml{j=1}{J}\left[k_j\Grad p_{j,h}^n\right]_e \cdot n_e}{e}
\end{split}
\end{equation}
Similar to \eqref{eqn:momtm-residual-elem-timeshift} we will use the time-shifted
version of the local spatial error indicator for the momentum equation.  This 
expression is given by
\begin{align*}
	\eta_{u,K}^n(\delta_t) &= h_K^2\nrm{\delta_tR_{uh,K}}{K}^2 %
	+ h_K\suml{e\in\partial K}{}\nrm{\delta_t J_{uh,e}}{e}^2,
\end{align*} 
where the right-hand is analogous to that of
\eqref{eqn:space-err-indic-momtm} by taking the time-shift of the
expressions appearing inside the norm. With the local indicators in
hand we immediately have the global indicators and their time-shifted
version given by
\begin{equation}
\eta_u^n = \suml{K\in\mathcal{T}_h}{} \eta^n_{u,K}, \quad %
	\quad \eta_p^n = \suml{K\in\mathcal{T}_h}{}\eta_{p,K}^n,\quad%
	\eta_u^n(\delta_t) = \suml{K\in\mathcal{T}_h}{} \eta^n_{u,K}(\delta_t)
	\label{eqn:global-err-indic}\\
\end{equation}
In Section \ref{sec:adaptivity} we will use the above expressions to define the 
a posteriori error indicators informing a simple adaptive refinement strategy 
for the numerical simulations of Section \ref{sec:numerical-experiments}.

\subsection{A posteriori error estimators}
\label{sec:adaptivity:indicators}

We close this section by defining the final a posteriori error
estimators:
\begin{equation}
  \begin{split}
  \eta_1 = \left(\suml{n=1}{N} \tau_n \eta_p^n \right)^{\frac12},&\quad %
  \eta_2 = \sup\limits_{0\leq n \leq N} \left(\eta_{u}^n\right)^{\frac12},\\%
  \eta_3 = \suml{n=1}{N}\tau_n \left(\eta_u^n(\delta_t)\right)^{\frac12},&\qquad%
  \eta_4 = \left(\suml{n=1}{N}\tau_n\nrm{p_h^n - p_h^{n-1}}{d}^2.\right)^{\frac12}%
  \end{split}
  \label{eq:etas}
\end{equation}
The summed term $\nrm{p_h^n - p_h^{n-1}}{d}^2$, in $\eta_4$ above, can be 
expanded using the definition of \eqref{eq:d} as
\begin{align*}
\nrm{p_h^n - p_h^{n-1}}{d}^2 = \suml{j=1}{J}\kappa_j\nrm{\Grad %
\left(p_{j,h}^n - p_{j,h}^{n-1}\right)}{L^2}^2 + \frac12 %
\suml{j=1}{J}\suml{i=1}{J} \gamma_{ij} %
\nrm{ (p_{j,h}^n-p_{i,h}^{n}) - (p_{j,h}^{n-1} - p_{i,h}^{n-1})}{L^2}^2.
\end{align*}
Finally, a bound on the MPET discretization errors in terms of the a posteriori
error estimators follows:
\begin{prop}
  \label{prop:error}
  For each time $t_n$, $n \in \{0, 1, \dots,N \}$, the following
inequality for the discretization error holds
\begin{equation*}
  \begin{split}
    \nrm{u - u_{h\tau}}{L^{\infty}(0,t_n;V_a)} + 
    \nrm{p-p_{h\tau}}{L^{\infty}(0,t_n;L_d)} +
    \nrm{p-p_{h\tau}}{L^2(0,t_n;V_d)}  
    + \nrm{p-\pi^0 p_{h\tau}}{L^2(0,t_n;V_d)} \\
    \ls \eta_1 +  \eta_2 + \eta_3 + \eta_4 
    + \mathcal{E}_{h0}(u_0,p_0) + \mathcal{E}_{h}(f,g)
  \end{split}
\end{equation*}
where $\mathcal{E}_{h0}(u_0,p_0)$ and $\mathcal{E}_{h}(f,g)$ are
determined by the fidelity in the approximation of the initial data
and source terms, respectively, as
\begin{align*}
\mathcal{E}_{h0}(u_0,p_0) &= \nrm{u_0 - u_{h0}}{a} + \nrm{p_0 - p_{0h}}{c},\\ 
\mathcal{E}_{h}(f,g) &= \nrm{g-\pi^0 g_{h\tau}}{L^2(0,T;V_d)} + %
	\nrm{f-f_{h\tau}}{L^{\infty}(0,T;V_a)} + %
	\nrm{\partial_t(f-f_{h\tau})}{L^1(0,T;V_a)}.
\end{align*}
\end{prop}
\begin{proof}
The above follows from the results of \cite[Thm.~4.1, Prop.~4.1,
  Thm~4.2]{ern-meunier-2009} applied in the context of generalized
poroelasticity in light of the results of
Section~\ref{sec:general-setting}.
\end{proof}

\begin{rem}
The framework result \cite[Prop.~4.2]{ern-meunier-2009} is stronger than the restatement given above; only the relevant left-hand side quantities for our computations have been restated.  \changed{It is also interesting to ask whether the framework of Ern and Meunier \cite{ern-meunier-2009} can be extended to yield a posteriori error estimators for higher-order time discretizations of elliptic-parabolic systems (e.g.~\eqref{eq:general}).  One might ponder, for instance, the use of the generalized $\theta$ scheme $\delta_t y_h^n = \theta f(y_h^n) + (1-\theta)f(y_h^{n-1})$ for which \eqref{eq:discrete} is $\theta=1$ and $\theta=1/2$ yields the trapezoidal time integration method.  Adapting \cite{ern-meunier-2009} to this context could be approached by generalizing Lemma 2.1 and Theorem 3.1 alongside extending the discrete scheme interpolation, Galerkin residuals, element and jump residuals, Theorem 4.1 and Proposition 4.1-4.3 of \cite[Sec.~4.1]{ern-meunier-2009}. Though higher-order time discretization schemes are of practical importance, the analytic extension of \cite{ern-meunier-2009} to this context is a topic for future work.} 
\end{rem}


\section{Numerical convergence and accuracy of error estimators}
\label{sec:numerical-experiments}

\begin{table}
  \begin{subtable}{1.0\textwidth}
  \centering
  \caption{$\nrm{u - u_{h\tau}}{L^{\infty}(0, T; (H^1_0)^d)}$}
  \begin{tabular}{c|ccccc|c}
    \toprule
    N/dt  & $\tau_0$ & $\tau_0/2$ & $\tau_0 / 4$ & $\tau_0 /8$ & $\tau_0/16$ & Rate ($h$) \\
    \midrule
4 & \num{1.82e-02} & \num{1.82e-02} & \num{1.82e-02} & \num{1.82e-02} & \num{1.82e-02} & \\
8 & \num{4.71e-03} & \num{4.64e-03} & \num{4.62e-03} & \num{4.61e-03} & \num{4.61e-03} & 1.98 \\
16 & \num{1.44e-03} & \num{1.24e-03} & \num{1.18e-03} & \num{1.16e-03} & \num{1.16e-03} & 1.99 \\
32 & \num{8.51e-04} & \num{5.29e-04} & \num{3.63e-04} & \num{3.10e-04} & \num{2.96e-04} & 1.97\\
64 & \num{7.86e-04} & \num{4.50e-04} & \num{2.40e-04} & \num{1.36e-04} & \num{9.07e-05} & 1.70 \\
    \midrule
    Rate ($\tau$) & &  0.81 & 0.90 & 0.82 & 0.59 & \textbf{1.77}\\
    \bottomrule 
  \end{tabular} 
  \label{tab:exp:1:a}
  \end{subtable}
  \begin{subtable}{1.0\textwidth}
  \centering
  \caption{$\nrm{p-p_{h\tau}}{L^{\infty}(0, T; (L^2)^J)}$}
  \begin{tabular}{c|ccccc|c}
    \toprule
    N/dt  & $\tau_0$ & $\tau_0/2$ & $\tau_0/4$ & $\tau_0/8$ & $\tau_0/16$ & Rate ($h$) \\
    \midrule
    4 & \num{8.69e-02} & \num{8.93e-02} & \num{8.66e-02} & \num{8.52e-02} & \num{8.46e-02}  & \\
    8 & \num{3.97e-02} & \num{3.29e-02} & \num{2.73e-02} & \num{2.47e-02} & \num{2.36e-02}  & 1.84 \\
    16 & \num{3.06e-02} & \num{1.97e-02} & \num{1.23e-02} & \num{8.74e-03} & \num{7.10e-03} & 1.73\\
    32 & \num{2.89e-02} & \num{1.69e-02} & \num{9.13e-03} & \num{5.14e-03} & \num{3.16e-03} & 1.14 \\
    64 & \num{2.86e-02} & \num{1.63e-02} & \num{8.46e-03} & \num{4.39e-03} & \num{2.33e-03} & 0.44 \\
    \midrule
    Rate ($\tau$) & &  0.81 & 0.95 & 0.95 & 0.91 & \textbf{1.14} \\
    \bottomrule
  \end{tabular}
  \label{tab:exp:1:b}
  \end{subtable}
  \begin{subtable}{1.0\textwidth}
  \centering
  \caption{$\nrm{p-p_{h\tau}}{L^{2}(0, T; (H^1_0)^J)}$}
  \begin{tabular}{c|ccccc|c}
    \toprule
    N/dt  & $\tau_0$ & $\tau_0/2$ & $\tau_0 / 4$ & $\tau_0 /8$ & $\tau_0/16$ & Rate ($h$) \\
    \midrule
    4 & \num{4.42e-01} & \num{5.26e-01} & \num{5.39e-01} & \num{5.41e-01} & \num{5.41e-01} &  \\
    8 & \num{2.37e-01} & \num{2.73e-01} & \num{2.77e-01} & \num{2.78e-01} & \num{2.78e-01} & 0.96 \\
    16 & \num{1.38e-01} & \num{1.45e-01} & \num{1.42e-01} & \num{1.40e-01} & \num{1.40e-01} & 0.99 \\
    32 & \num{9.81e-02} & \num{8.54e-02} & \num{7.49e-02} & \num{7.13e-02} & \num{7.03e-02} & 0.99 \\
    64 & \num{8.53e-02} & \num{6.23e-02} & \num{4.46e-02} & \num{3.77e-02} & \num{3.57e-02} & 0.98\\
    \midrule
    Rate ($\tau$) & &  0.45 & 0.48 & 0.24 & 0.08 & \textbf{1.00} \\
    \bottomrule
  \end{tabular}
  \end{subtable}
  \begin{subtable}{1.0\textwidth}
  \centering
  \caption{$\nrm{p-\pi^0 p_{h \tau}}{L^{2}(0, T; (H^1_0)^J)}$}
  \begin{tabular}{c|ccccc|c}
    \toprule
    N/dt  & $\tau_0$ & $\tau_0/2$ & $\tau_0 / 4$ & $\tau_0 /8$ & $\tau_0/16$ & Rate ($h$) \\
    \midrule
    4 & \num{9.41e-01} & \num{6.94e-01} & \num{6.03e-01} & \num{5.67e-01} & \num{5.53e-01} & \\
    8 & \num{7.89e-01} & \num{4.73e-01} & \num{3.49e-01} & \num{3.02e-01} & \num{2.87e-01} & 0.95\\
    16 & \num{7.44e-01} & \num{3.94e-01} & \num{2.40e-01} & \num{1.74e-01} & \num{1.50e-01} & 0.93 \\
    32 & \num{7.32e-01} & \num{3.71e-01} & \num{2.03e-01} & \num{1.21e-01} & \num{8.66e-02} & 0.80\\
    64 & \num{7.29e-01} & \num{3.65e-01} & \num{1.93e-01} & \num{1.04e-01} & \num{6.09e-02} & 0.51\\
    \midrule
    Rate ($\tau$) & & 1.00 & 0.92 & 0.89 & 0.77  & \textbf{0.99} \\
    \bottomrule
  \end{tabular}
  \end{subtable}
  \caption{Displacement and pressure approximation errors (in
    different norms) and their rates of convergence for the smooth
    3-network test case under uniform refinement in space
    (horizontal) and time (horizontal). $T = 0.4$, $\tau_0 =
    T/2$. Rate ($\tau$) is the rate for the finest mesh, under time
    step refinement. Rate ($h$) is the rate for the finest time step,
    under mesh refinement. The diagonal rate (in bold) is the final
    space-time (diagonal) rate.}
  \label{tab:exp:1}
\end{table}
\begin{table}[hbt!]
  \begin{subtable}{1.0\textwidth}
  \centering
  \caption{$\eta_1$}
  \begin{tabular}{c|ccccc|c}
    \toprule
    N/dt  & $\tau_0$ & $\tau_0/2$ & $\tau_0 / 4$ & $\tau_0 /8$ & $\tau_0/16$ & Rate ($h$) \\
    \midrule
4 & \num{3.30e+00} & \num{3.19e+00} & \num{3.13e+00} & \num{3.09e+00} & \num{3.08e+00} & \\
8 & \num{1.73e+00} & \num{1.67e+00} & \num{1.64e+00} & \num{1.63e+00} & \num{1.62e+00} & 0.93\\
16 & \num{8.80e-01} & \num{8.52e-01} & \num{8.37e-01} & \num{8.29e-01} & \num{8.25e-01} & 0.97\\
32 & \num{4.43e-01} & \num{4.29e-01} & \num{4.22e-01} & \num{4.18e-01} & \num{4.16e-01} & 0.99 \\
64 & \num{2.22e-01} & \num{2.15e-01} & \num{2.12e-01} & \num{2.10e-01} & \num{2.09e-01} & 1.00 \\
    \midrule
    Rate ($\tau$) & &  0.05 & 0.03 & 0.01 & 0.01 & \textbf{1.00} \\
    \bottomrule 
  \end{tabular} 
  \end{subtable}
  \begin{subtable}{1.0\textwidth}
  \centering
  \caption{$\eta_2$}
  \begin{tabular}{c|ccccc|c}
    \toprule
    N/dt  & $\tau_0$ & $\tau_0/2$ & $\tau_0 / 4$ & $\tau_0 /8$ & $\tau_0/16$ & Rate ($h$) \\
    \midrule
4 & \num{1.76e+00} & \num{1.76e+00} & \num{1.76e+00} & \num{1.76e+00} & \num{1.76e+00} & \\
8 & \num{4.51e-01} & \num{4.51e-01} & \num{4.51e-01} & \num{4.51e-01} & \num{4.51e-01} & 1.97\\
16 & \num{1.14e-01} & \num{1.14e-01} & \num{1.14e-01} & \num{1.14e-01} & \num{1.14e-01} & 1.99\\
32 & \num{2.84e-02} & \num{2.84e-02} & \num{2.84e-02} & \num{2.84e-02} & \num{2.84e-02} & 2.00\\
64 & \num{7.12e-03} & \num{7.12e-03} & \num{7.12e-03} & \num{7.12e-03} & \num{7.12e-03} & 2.00\\
\midrule
Rate ($\tau$) & &  -0.00 & -0.00 & -0.00 & -0.00 & \textbf{2.00} \\
    \bottomrule
  \end{tabular}
  \end{subtable}
  \begin{subtable}{1.0\textwidth}
  \centering
  \caption{$\eta_3$}
  \begin{tabular}{c|ccccc|c}
    \toprule
    N/dt  & $\tau_0$ & $\tau_0/2$ & $\tau_0 / 4$ & $\tau_0 /8$ & $\tau_0/16$ & Rate ($h$) \\
    \midrule
4 & \num{1.76e+00} & \num{1.77e+00} & \num{1.77e+00} & \num{1.77e+00} & \num{1.77e+00} & \\
8 & \num{4.51e-01} & \num{4.52e-01} & \num{4.52e-01} & \num{4.52e-01} & \num{4.52e-01} & 1.97\\
16 & \num{1.14e-01} & \num{1.14e-01} & \num{1.14e-01} & \num{1.14e-01} & \num{1.14e-01}& 1.99\\
32 & \num{2.85e-02} & \num{2.85e-02} & \num{2.85e-02} & \num{2.85e-02} & \num{2.85e-02}& 2.00\\
64 & \num{7.12e-03} & \num{7.13e-03} & \num{7.13e-03} & \num{7.13e-03} & \num{7.13e-03} & 2.00\\
\midrule
    Rate ($\tau$) & &  -0.00 & -0.00 & -0.00 & -0.00 & \textbf{2.00} \\
    \bottomrule
  \end{tabular}
  \end{subtable}
  \begin{subtable}{1.0\textwidth}
  \centering
  \caption{$\eta_4$}
  \begin{tabular}{c|ccccc|c}
    \toprule
    N/dt  & $\tau_0$ & $\tau_0/2$ & $\tau_0 / 4$ & $\tau_0 /8$ & $\tau_0/16$ & Rate ($h$) \\
    \midrule
4 & \num{1.25e+00} & \num{6.65e-01} & \num{3.39e-01} & \num{1.70e-01} & \num{8.54e-02 } & \\
8 & \num{1.28e+00} & \num{6.81e-01} & \num{3.47e-01} & \num{1.75e-01} & \num{8.76e-02} & -0.04\\
16 & \num{1.29e+00} & \num{6.85e-01} & \num{3.49e-01} & \num{1.76e-01} & \num{8.81e-02} & -0.01\\
32 & \num{1.29e+00} & \num{6.86e-01} & \num{3.50e-01} & \num{1.76e-01} & \num{8.83e-02} & -0.0\\
64 & \num{1.29e+00} & \num{6.86e-01} & \num{3.50e-01} & \num{1.76e-01} & \num{8.83e-02} & -0.0\\
    \midrule
    Rate ($\tau$) & &  0.91 & 0.97 & 0.99 & 1.00  & \textbf{1.00} \\
    \bottomrule
  \end{tabular}
  \end{subtable}
    \begin{subtable}{1.0\textwidth}
  \centering
  \caption{$\tilde{I}_{\rm eff}$}
  \begin{tabular}{c|ccccc}
    \toprule
    N/dt  & $\tau_0$ & $\tau_0/2$ & $\tau_0 / 4$ & $\tau_0 /8$ & $\tau_0/16$ \\
    \midrule
4 & \num{5.42e+00} & \num{5.56e+00} & \num{5.61e+00} & \num{5.61e+00} & \num{5.59e+00}\\
8 & \num{3.65e+00} & \num{4.16e+00} & \num{4.39e+00} & \num{4.44e+00} & \num{4.40e+00}\\
16 & \num{2.62e+00} & \num{3.15e+00} & \num{3.58e+00} & \num{3.80e+00} & \num{3.82e+00}\\
32 & \num{2.08e+00} & \num{2.47e+00} & \num{2.88e+00} & \num{3.29e+00} & \num{3.50e+00}\\
64 & \num{1.81e+00} & \num{2.06e+00} & \num{2.34e+00} & \num{2.74e+00} & \num{3.14e+00}\\
    \bottomrule
  \end{tabular}
  \label{tab:exp:2:I}
    \end{subtable}
  \caption{Error estimators $\eta_1, \eta_2, \eta_3, \eta_4$ and their
    rates of convergence, and Bochner efficiency indices
    $\tilde{I}_{\rm eff}$ for the smooth 3-network test case under
    uniform refinement in space (horizontal) and time (horizontal) $T
    = 0.4$, $\tau_0 = T/2$. Rate ($\tau$) is the rate for the finest
    mesh, under time step refinement. Rate ($h$) is the rate for the
    finest time step, under mesh refinement. The diagonal rate (in
    bold) is the final space-time (diagonal) rate.}
  \label{tab:exp:2}
\end{table}

\afterpage{\clearpage}

To examine the accuracy of the computed error estimators and resulting
error estimate, we first study an idealized test case with a
manufactured smooth solution over uniform meshes. We will consider
adaptive algorithms and meshes in the subsequent sections. All
numerical experiments were implemented using the FEniCS Project finite
element software~\citep{alnaes2015fenics}.

Let $\Omega = [0, 1]^2$ with coordinates $(x, y) \in \Omega$, and let
$T = 0.4$. We consider the case of three fluid networks ($J = 3$),
first with $\mu = 1.0$, $\lambda = 10.0$, $\alpha_j = 0.5$, $s_j =
1.0$, $\kappa_j = 1.0$, for $j = 1, 2, 3$ and $\gamma_{12} = \gamma_{23} =
\gamma_{13} = 1.0$. We define the following smooth solutions
to~\eqref{eq:mpet}:
\begin{align*}
  u(x, y) &= (0.1 \cos(\pi x) \sin(\pi y) \sin(\pi t), 0.1 \sin(\pi x) \cos(\pi y) \sin(\pi t)), \\
  p_1(x, y) &= \sin(\pi x) \cos(\pi y) \sin(2 \pi t), \\
  p_2(x, y) &= \cos(\pi x) \sin(\pi y) \sin(\pi t), \\
  p_3(x, y) &= \sin(\pi x) \sin(\pi y) t,
\end{align*}
with compatible Dirichlet boundary conditions, initial conditions and
induced force and source functions $f$ and $g_j$ for $j = 1, 2, 3$.

We approximate the solutions using Taylor--Hood type elements relative
to given families of meshes; i.e.~continuous piecewise quadratic
vector fields for the displacement and continuous piecewise linears
for each pressure. The exact solutions were approximated using
continuous piecewise cubic finite element spaces in the numerical
computations.

\subsection{Convergence and accuracy under uniform refinement}

We first consider the convergence of the numerical solutions, their
approximation errors and error estimators $\eta_1, \eta_2, \eta_3,
\eta_4$ under uniform refinement in space and time. We define the
meshes by dividing the domain into $N \times N$ squares and
dividing each subsquare by the diagonal. The errors and convergence
rates for the displacement and pressure approximations, measured in
natural Bochner norms, are listed in Table~\ref{tab:exp:1}. We observe
that both the spatial and the temporal discretization contributes to
the errors, and that all variables converges at at least first order
in space and time - as expected with the implicit Euler scheme. For
coarse meshes, we observe that the displacement converges at the
optimal second order under mesh refinement (Table~\ref{tab:exp:1:a}).

We next consider the convergence and accuracy of the error estimators
$\eta_1, \eta_2, \eta_3, \eta_4$ for the same set of discretizations
(Table~\ref{tab:exp:2}). We observe that each error estimator converge
at at least first order in space-time, with $\eta_2$ and $\eta_3$
converging at second order in space and $\eta_4$ converging at first
order in time\footnote{\changed{We observe that the error estimators
      $\eta_2$ and $\eta_3$ are nearly (but not quite) identical for
      this test case, and conjecture that this may be not entirely
      coincidental but related to the choice of the exact solution.}}.

\changed{We also define two efficiency indices $\tilde{I}_{\rm eff}$ and
$I_{\rm eff}$ with respect to the Bochner and energy norms,
respectively, for the evaluation of the approximation error:
\begin{equation}
  \label{eq:index}
  \tilde{I}_{\rm eff} = \frac{\eta}{\tilde{E}}, \quad  
  I_{\rm eff} = \frac{\eta}{E} , 
\end{equation}
}
where
\begin{align*}
  \label{eq:def:etaE}
  \eta & \equiv \eta_1 + \eta_2 + \eta_3 + \eta_4, \\
  \changed{\tilde{E}} &\equiv \nrm{u - u_{h\tau}}{L^{\infty}(0, T; H^1_0)} + 
  \nrm{p-p_{h\tau}}{L^{\infty}(0, T; L^2)} +
  \nrm{p-p_{h\tau}}{L^2(0, T; H^1_0)}  
  + \nrm{p-\pi^0 p_{h\tau}}{L^2(0, T;H^1_0)}, \\
  \changed{E } &\changed{\equiv \nrm{u - u_{h\tau}}{L^{\infty}(0, T; V_a)} + 
  \nrm{p-p_{h\tau}}{L^{\infty}(0, T; L_d)} +
  \nrm{p-p_{h\tau}}{L^2(0, T; V_d)}  
  + \nrm{p-\pi^0 p_{h\tau}}{L^2(0, T; V_d)}}.
\end{align*}
\changed{Note that we use both Bochner- and energy norms to
  investigate the practical quality and efficiency of the
  approximations and estimators as well as in terms of the
  energy/parameter-weighted norms appearing in the theoretical bound
  (Proposition~\ref{prop:error}).} For this test case, we find Bochner
efficiency indices between $1.8$ and $5.7$, with little variation in
this efficiency index between time steps for coarse meshes, and
efficiency indices closer to $1$ for finer meshes.

\subsection{Variations in material parameters}

\begin{figure}
  \begin{center}
    \includegraphics[width=0.49\textwidth]{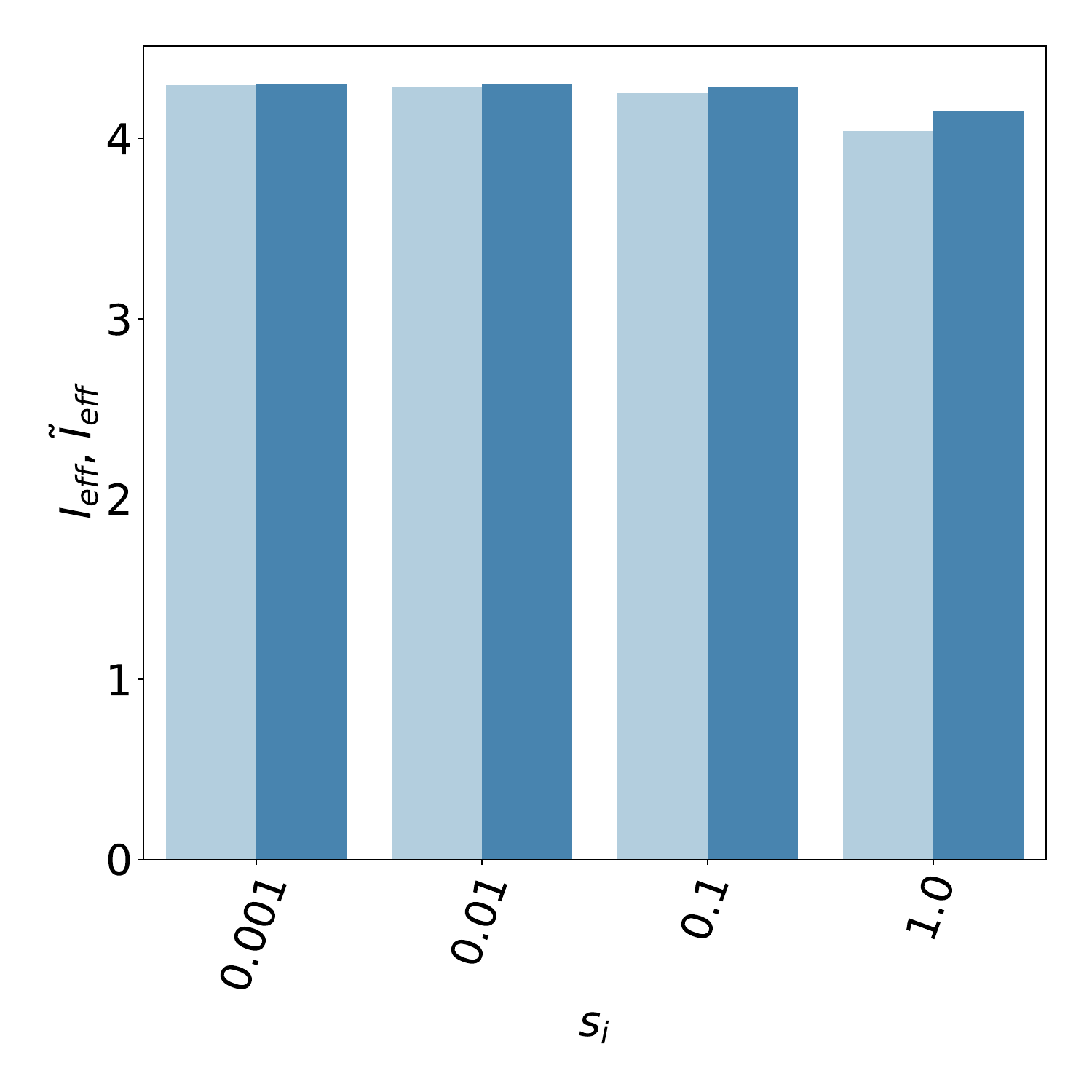}
    \includegraphics[width=0.49\textwidth]{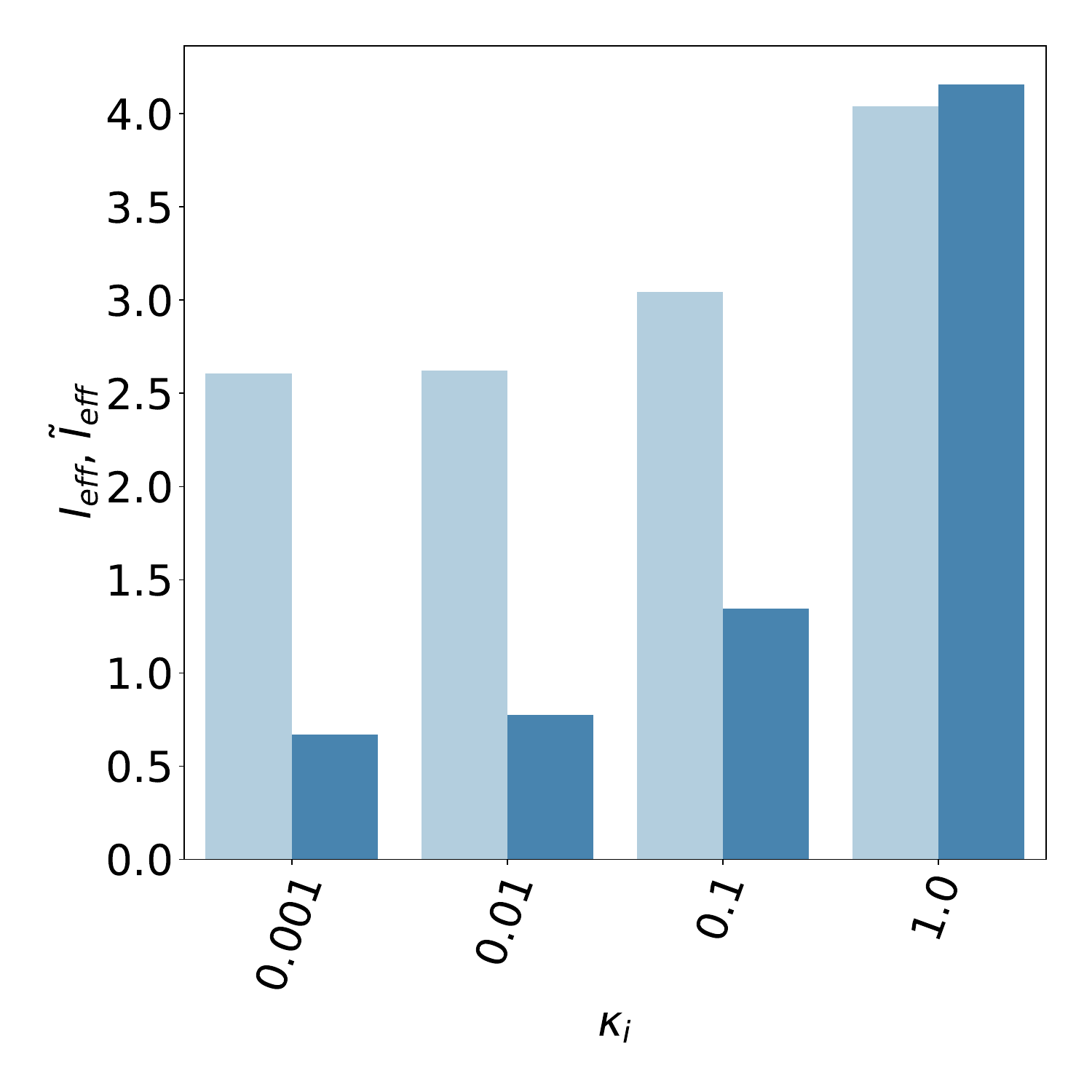} \\
    \includegraphics[width=0.49\textwidth]{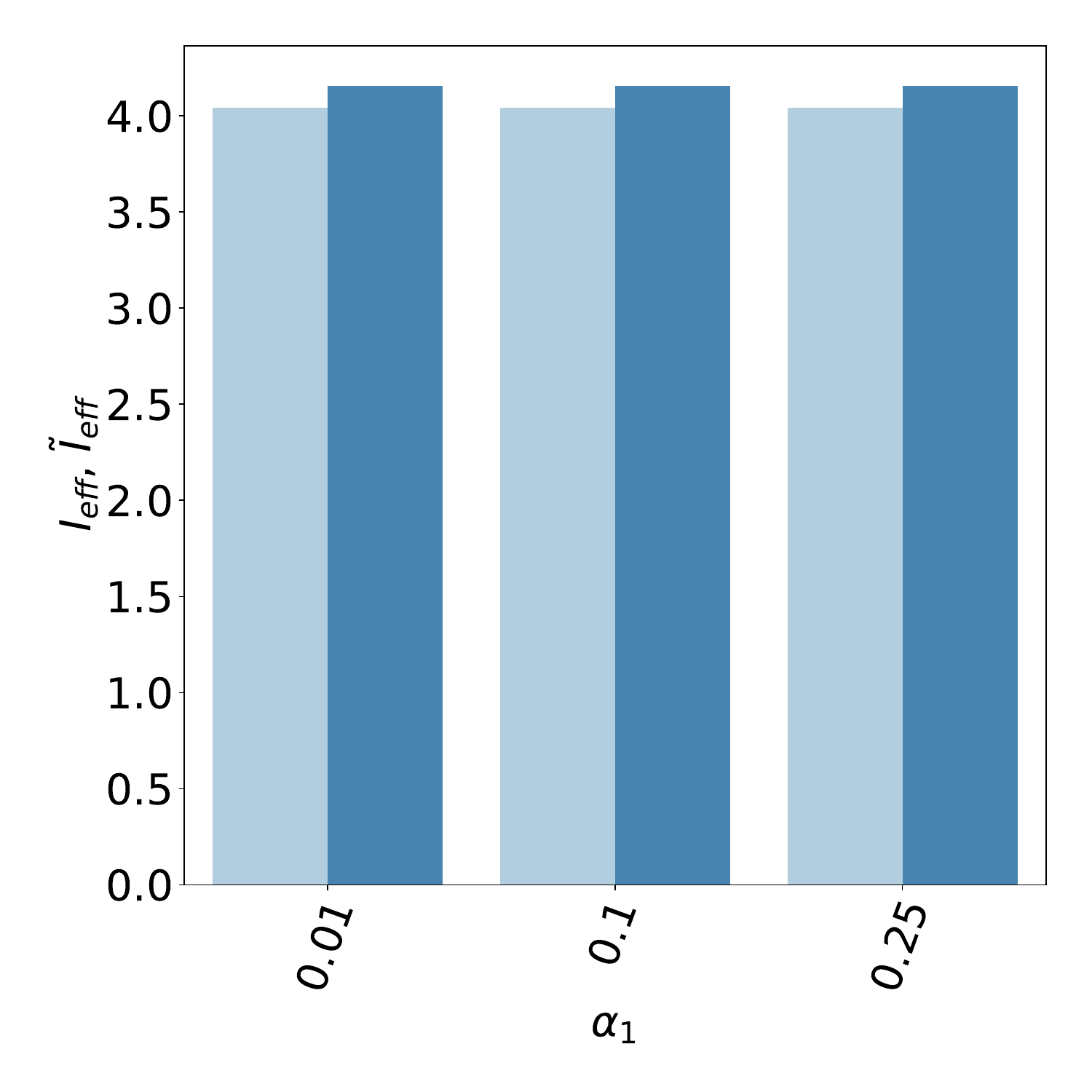}
    \includegraphics[width=0.49\textwidth]{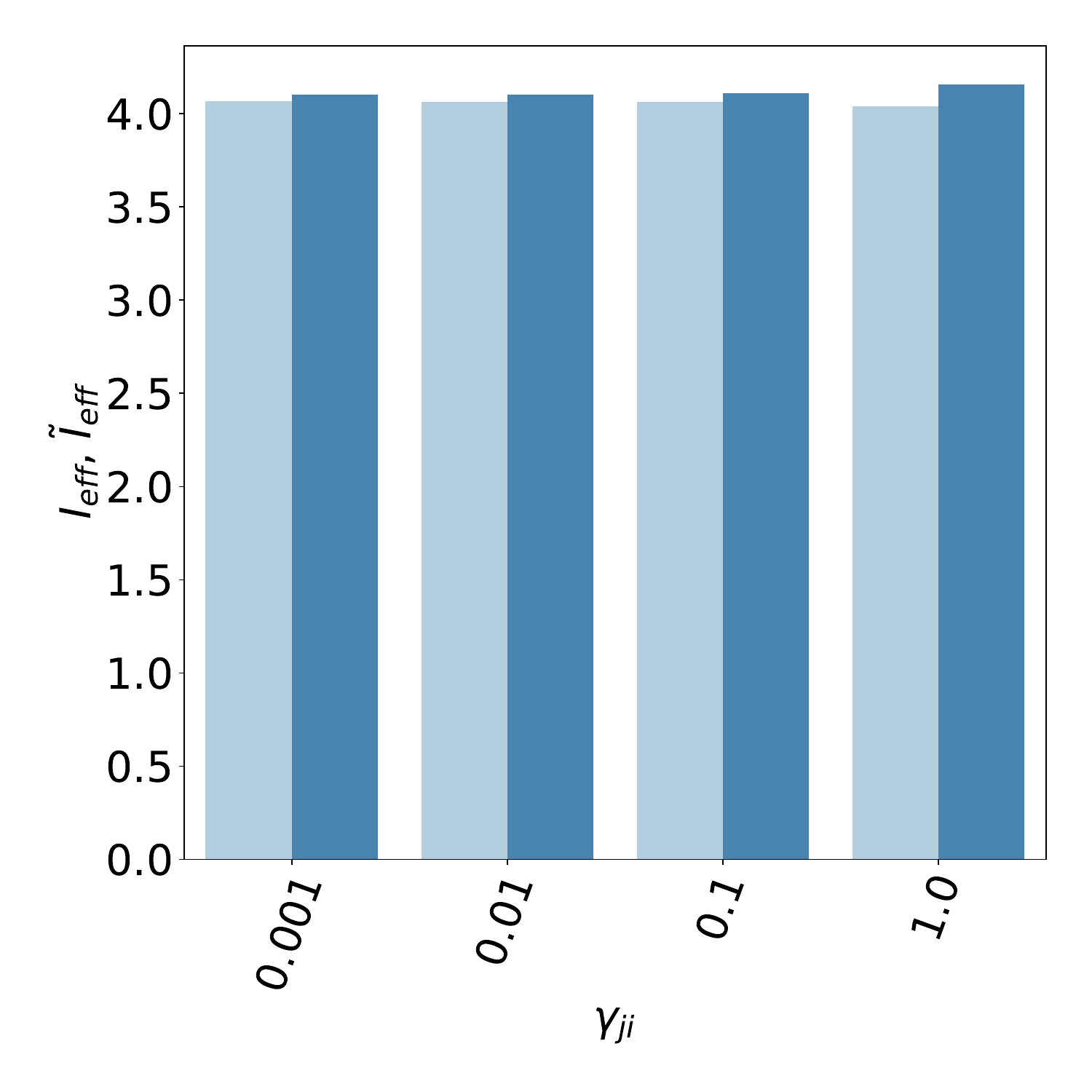} \\
    \includegraphics[width=0.49\textwidth]{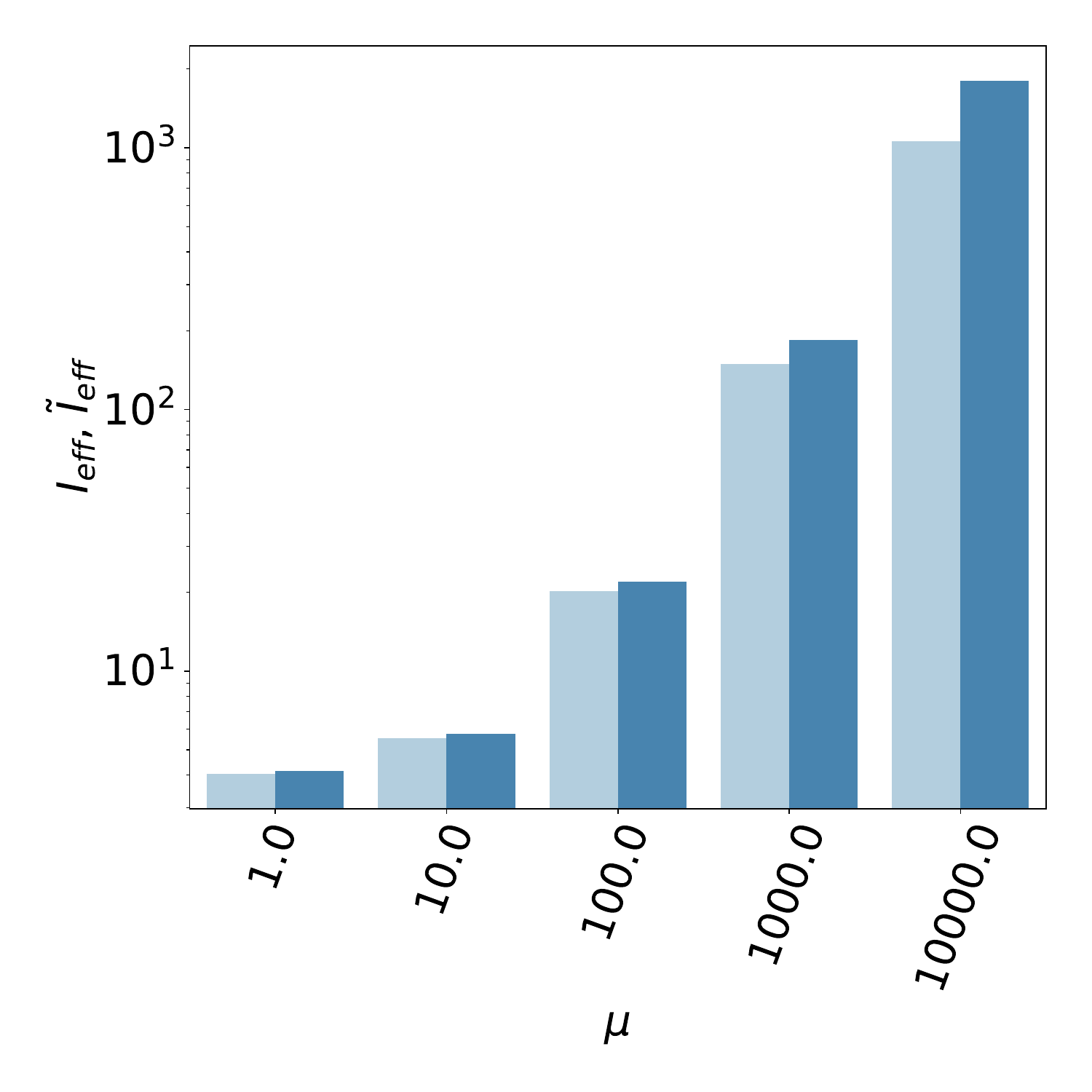}
    \includegraphics[width=0.49\textwidth]{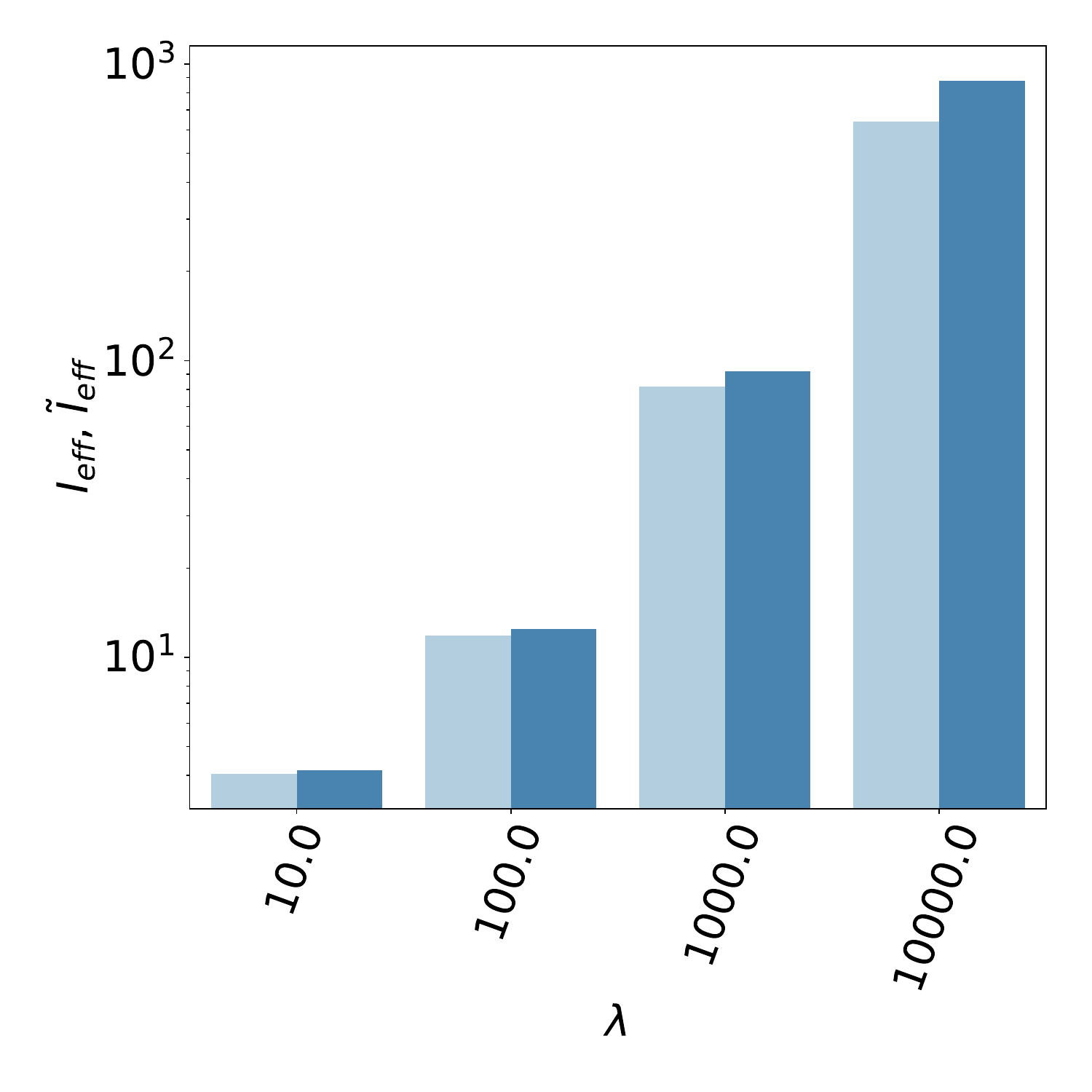} 
  \end{center}
  \caption{Efficiency indices for the smooth 3-network test case for
    different variations of the MPET material parameters $s_i$,
    $\kappa_i$, $\alpha_i$, $\gamma_{ji}$, $\mu$, and $\lambda$;
    energy-norm indices $I_{\rm eff}$ (light blue) and Bochner norm indices
    $\tilde{I}_{\rm eff}$ (middle blue). Numerical resolution parameters
    were kept fixed ($T = 0.4$, $\tau = 0.1$, $n=8$).}
  \label{fig:material}
\end{figure}
We also study how variations in the material parameters affect the
effectivity of the error estimates, \changed{measured in terms of the
  effectivity index $I_{\rm eff}$~\eqref{eq:index} with respect to the
  energy norm(s). We consider a set of default parameters: $\mu =
  1.0$, $\lambda = 10.0$, $\alpha_1 = 0.25$, $\alpha_2 = \alpha_3 -
  \alpha_1$, $\alpha_3 = 0.5$, $s_i = \kappa_i = \gamma_{ji} = 1.0$
  for $i = 1, 2, 3$, $j \not = i$, and subsequently independent
  variations in the material parameters representing} increased
stiffnesses $E$, reduced compressibilities $\nu$, lower transfer
$\gamma$, lower hydraulic conductances $\kappa$, and lower specific
storage coefficients $s$. \changed{Specifically, we consider $\alpha_1
  \in \{0.01, 0.1, 0.25\}$, $s_i, \kappa_i, \gamma_{ji} \in \{0.001,
  0.01, 0.1, 1.0 \}$ for $i = 1, 2, 3$, $j \not = i$, and $\mu \in
  \{1, 10, 100, 1000, 10000 \}$, $\lambda \in \{10, 100, 1000, 10000
  \}$. Both energy-norm and Bochner efficiency indices $I_{\rm eff}$
  and $\tilde{I}_{\rm eff}$ for the different variations are shown in
  Figure~\ref{fig:material}. }

\changed{The energy-norm efficiency indices $I_{\rm eff}$ are above 1
  for all material variations considered. Variations in the specific
  storage coefficients, Biot-Willis coefficients and transfer
  coefficients have minimal effect on both the energy- and Bochner
  norm efficiency indices: the efficiency indices are $\sim4$ for all
  variations in each of these parameters. For variations in the
  permeabilities $\kappa_i$, we observe some reduction in the
  energy-norm efficiency indices as the permeability is reduced (from
  $4.1$ to $2.6$), but that the index value stabilizes around $2.5$
  for the smaller permeabilities. We observe similar behaviour for the
  Bochner-norm efficiency index, but with index values of $\sim0.8$
  for the smaller permeabilities, and thus efficiency indices below
  $1.0$.} \changed{For the elastic parameters, the results are quite
  different. Both the energy-norm and Bochner efficiency indices
  increase substantially with increasing Lam\'e parameters $\mu$ and
  $\lambda$, though with Bochner efficiency indices increasing more.}

\changed{%
\begin{rem}
	The boundedness of efficiency indices, e.g.~$I_{\rm eff}$ and $\tilde{I}_{\rm eff}$, is canonically provided by the reverse inequality of Proposition~\ref{prop:error}, whereby the the error indicators are bounded above in terms of a constant times the norm of the discretization error. That is, one establishes 
	\begin{align*}
		\eta_1 + \eta_2 + \eta_3 + \eta_4 \lesssim \nrm{u-u_{h\tau}}{L^{\infty}(0,T;V_a)} &+ \nrm{p-p_{h\tau}}{L^{\infty}(0,T;L_d)} \\&+ \nrm{p-p_{h\tau}}{L^{2}(0,T;V_d)} + \nrm{p-\pi^0 p_{h\tau}}{L^2(0,T;V_d)}.
	\end{align*}
	If the constant of proportionality, in the above inequality, does not involve specific material parameters, then the efficiency indices are robust with respect to variations in those parameters.  However, as in the case of the use of higher-order time discretizations, the general Euler-Galerkin elliptic-parabolic framework of Ern and Meunier \cite{ern-meunier-2009} does not provide this bound and, as a result, its extension to the equations of generalized poroelasticity (MPET), presented herein, is limited in this same regard.  The computational experiments of this section suggest that such an estimate will entail a constant of proportionality that scales strongly with both $\mu$ and $\lambda$.  
\end{rem}
}

\section{Adaptive strategy: algorithmic considerations and numerical evaluation}
\label{sec:adaptivity}

\changed{We now turn to consider and evaluate two components of an overall
adaptive strategy: (i) temporal adaptivity (only) and (ii) temporal
and spatial adaptivity.} Our choices for the adaptive strategy can be
viewed in light of the observations on the convergence of $\eta_1,
\eta_2, \eta_3, \eta_4$ for the previous test case, as well as the
following characteristics of MPET problems arising in e.g.~biological
applications:
\begin{itemize}
\item
  Mathematical models of living tissue are often associated with a
  wide range of uncertainty e.g.~in terms of modelling assumptions,
  material parameters, and data fidelity. Simulations are therefore
  often not constrained by a precise numerical error tolerance, but
  rather by the limited availability of computational resources.
\item
  Living tissue often feature heterogeneous material parameters, but
  typically with small jumps, and in particular smoother variations
  than e.g.~in the geosciences. The corresponding MPET solutions are
  often relatively smooth.
\item
  Even for problems with a small number of networks $J$ such as single
  or two-network settings, the linear systems to be solved at each
  time step are relatively large already for moderately coarse meshes.
\end{itemize}
\changed{In light of these points, our target is an adaptive algorithm
  robustly reducing the error(s) given limited computational
  resources. We therefore consider an \emph{error balancing strategy}
  in which we adaptively refine time steps such that the estimated
  temporal and spatial contributions to the error is balanced, then
  refine the spatial mesh to reduce the overall error, and
  repeat. This approach to spatial adaptivity seeks to balance the
  computational gains associated with adaptively refined meshes and
  the computational and implementational overhead costs associated
  with more sophisticated time-space adaptive methods, see
  e.g.~\citep{ahmed2019adaptive, bendahmane2010multiresolution}. While
  a full space-time adaptive algorithm could yield time-varying meshes
  of lower computational cost, the computational costs associated with
  finite element matrix assembly over separate meshes, and
  interpolation of discrete fields between different meshes are often
  substantial. Without time-varying meshes, the blocks of the MPET
  linear operator can be reused (using the time steps $\tau_n$ as
  weights) which may reduce assembly time, and potentially linear
  system solver times. We note though that there is ample room for
  more sophisticated time step control methods than we consider here,
  see e.g.~\citep{soderlind2002automatic} and related works.}

\subsection{Time adaptivity}

We consider the time-adaptive scheme listed in
Algorithm~\ref{alg:time}. Overall, for a given mesh $\mathcal{T}_h$,
we step forward in time, evaluate (an approximation to) the error
estimators at the current time step, compare the spatial and temporal
contributions to the error estimators, and coarsen (or refine) the
time step if the spatial (or temporal) error dominates. 
\begin{algorithm}
\begin{algorithmic}[1]

  \State Define adaptive parameters \changed{$\alphafac \in [0, 1)$} and $\beta
  \geq 1$, $\tau_{\max} > 0$ and $\tau_{\min} \geq 0$.

  \State Assume that a mesh $\mathcal{T}_h$ and an initial time step size
  $\tau_0$ is given. Set $t^0$ and set the time step iterator $n = 0$.

  \While{$t^n < T$}
  \While{True}
  \State Set $n = n + 1$.
  \State Set $t^{\ast} = t^{n-1} + \tau_n$, and
  solve~\eqref{eq:discrete} over $\mathcal{T}_h$ for $(u^{\ast}_h,
  p^{\ast}_h)$ with time step size $\tau_n$
  \State Compute error estimator approximations at the current time step  
  \begin{equation*}
    \eta_1^n = (\tau_n \eta^n_p)^{\frac12}, \quad
    \eta_2^n = (\sup_{0 \geq m \geq n} \eta^m_u)^{\frac12}, \quad 
    \eta_3^n = \tau_n (\eta_u^n(\delta_t))^{\frac12}, \quad 
    \eta_4^n = (\tau_n \| p_h^n - p_h^{n-1} \|_d^2)^{\frac12},
  \end{equation*}
  \State and set $\eta_h^n = \eta_1^n + \eta_2^n + \eta_3^n$, $\eta_{\tau}^n
  = \eta_4^n$.
  \If{$\eta_{\tau}^n \leq (1 - \changed{\alphafac}) \eta_h$ and $\beta \tau_n \leq \tau_{\max}$}

  \State Set $t^n = t^{\ast}$, $(u^n_h, p^n_h) = (u^{\ast}_h,
  p^{\ast}_h)$, coarsen the next $\tau^{n+1} = \beta \tau^n$, and break loop.

  \ElsIf{$\eta_{\tau}^n \geq (1 + \changed{\alphafac}) \eta_h^n$ and $\tau_n/\beta \geq \tau_{\min}$}
  \State
  Discard the solution and refine the time step: set $t^n = t^{n-1}$, $n = n - 1$, $\tau^n = \tau^n/\beta$.
  \Else
  \State Set $t^n = t^{\ast}$, $(u^n_h, p^n_h) = (u^{\ast}_h,
  p^{\ast}_h)$, $\tau^{n+1} = \tau^n$, and break loop.
  \EndIf
  \EndWhile
  \EndWhile
\end{algorithmic}
\caption{Time-adaptive algorithm}
\label{alg:time}
\end{algorithm}

\subsubsection*{Evaluation of the time-adaptive algorithm on a smooth numerical test case}

We evaluate Algorithm~\ref{alg:time} using the numerical test case with smooth
solutions defined over 3 networks as introduced in
Section~\ref{sec:numerical-experiments}, the default material
parameters, and different uniform meshes (defined by $2 \times N
\times N$ triangles as before). We also verified the adaptive solver
by comparing the solutions at each time step and error estimators
resulting from rejected coarsening and refinement (resulting in
$\tau_n = 0.2$ for each $n$) with the solutions and error estimators
computed with a uniform time step ($\tau = 0.2$). We let $T = 1.0$,
and considered an initial time step of $\tau_0 = 0.2$, adaptive weight
$\changed{\alphafac} = 0$, a coarsening/refinement factor $\beta = 2$, and time
step bounds $\tau_{\max} = T$ and $\tau_{\min} = 0.0$.

The discrete times $t^n$ resulting from the adaptive algorithm, error
estimators $\eta_h^n$ and $\eta_{\tau^n}$ are shown in
Figure~\ref{fig:adaptive:time:1} for different uniform mesh
resolutions. For $N = 8$ (Figure~\ref{fig:adaptive:time:1:a}), we
observe that the adaptive algorithm estimates the initial time step of
$0.2$ to be unnecessarily small in light of the dominating spatial
error, and coarsens the time step to $0.4$ before quickly reaching the
end of time ($T$). The $\|u - u_{h \tau} \|_{L^{\infty}}$, $\| p -
p_{h \tau} \|_{L^{\infty}}$, and $\| p - p_{h \tau} \|_{L^2}$ errors
are $4.61 \times 10^{-3}$, $3.86 \times 10^{-2}$ and $6.83 \times
10^{-1}$. For comparison, with a uniform time step $\tau = 0.2$, the
$\|u - u_{h \tau} \|_{L^{\infty}}$, $\| p - p_{h \tau}
\|_{L^{\infty}}$, $\| p - p_{h \tau} \|_{L^2}$, and $\| p - \pi^0 p_{h
  \tau}\|_{L^2}$ (as listed in Table~\ref{tab:exp:1}) are $4.71 \times
10^{-3}$, $4.38 \times 10^{-2}$, $4.96 \times 10^{-1}$, and $1.33$
respectively, and thus the errors with the adaptively defined coarser
time step are very comparable - as targeted by our error balancing
principle. The picture changes for $N = 16$
(Figure~\ref{fig:adaptive:time:1:b}), in this case the temporal error
initially dominates the spatial error, and the time step is reduced
substantially initially before a subsequent increase and plateau at
$0.1-0.2$. The value of the adaptive error estimator $\eta_4$ is lower
than for the uniform solution ($1.23$ vs $2.17$), but the exact errors
are comparable between the uniform and adaptive scheme in this
case. By setting $\tau_{\min} = \tau^0/4$, the unnecessarily high
initial time step refinement is limited
(Figure~\ref{fig:adaptive:time:1:c}), and again comparable errors as
for the uniform time step are observed. For higher spatial resolution
and thus lower spatial errors ($N = 32$), similar observations hold
(Figure~\ref{fig:adaptive:time:1:d}), but now the adaptive solutions
approximately halve the exact errors compared to the uniform $\tau^0 =
0.2$ case (as expected). We conclude that the time adaptive scheme
efficiently balances the temporal and spatial error, but does little
for reducing the overall error -- as the spatial error dominates this
case. For $N = 64$ and the same configurations, the adaptive time step
reduces to the minimal threshold $\tau_0/4 = 0.05$ and remains there
until end of time $T$, with the expected quartering of the exact
errors compared to the $\tau_0 = 0.2$ case (and the first order
accuracy of the temporal discretization scheme).
\begin{figure}
  \begin{subfigure}[b]{0.49\textwidth}
    \centering
    \includegraphics[width=\textwidth]{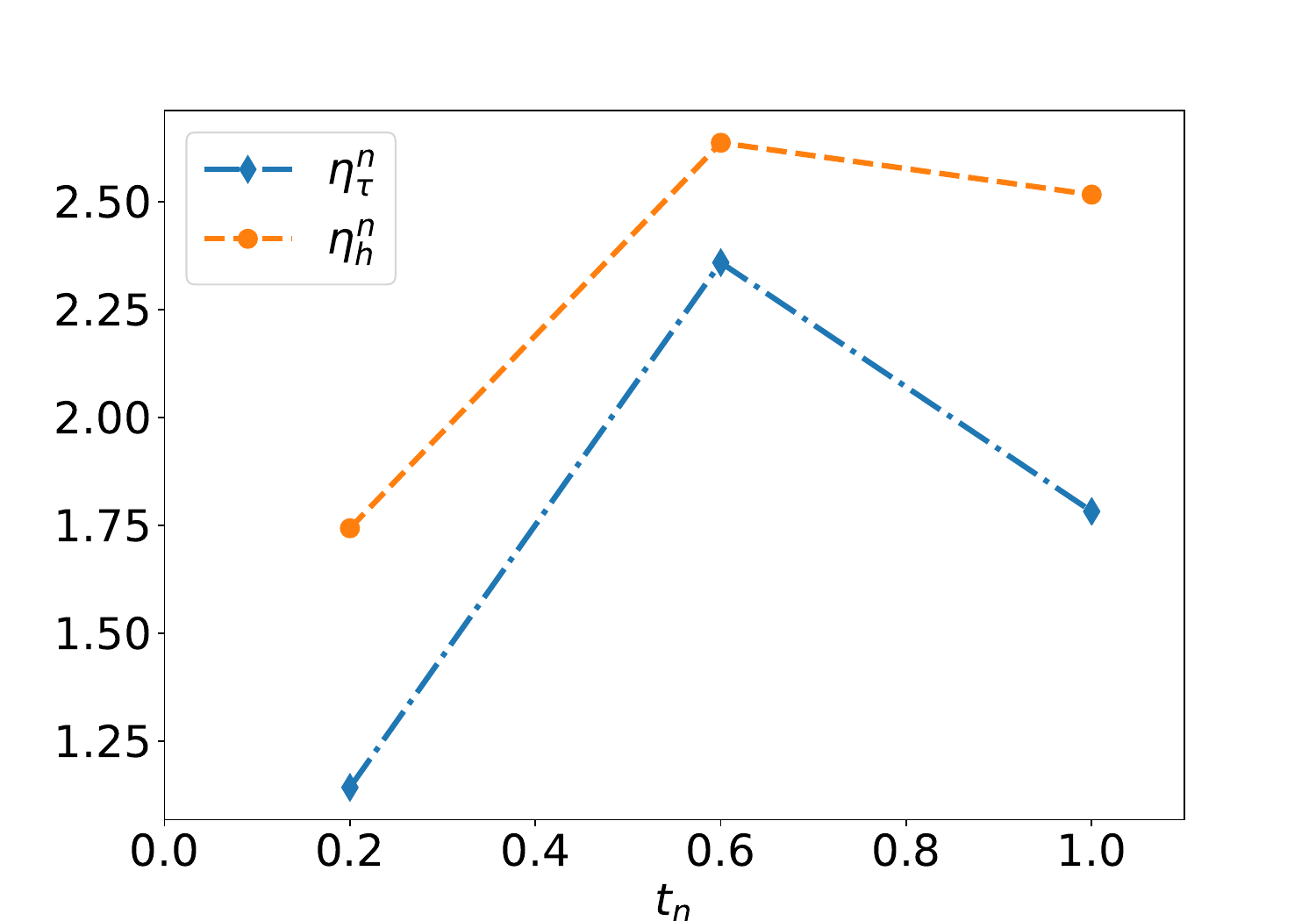}
    \caption{$N = 8$.}
  \label{fig:adaptive:time:1:a}
  \end{subfigure}
  \begin{subfigure}[b]{0.49\textwidth}
    \centering
    \includegraphics[width=\textwidth]{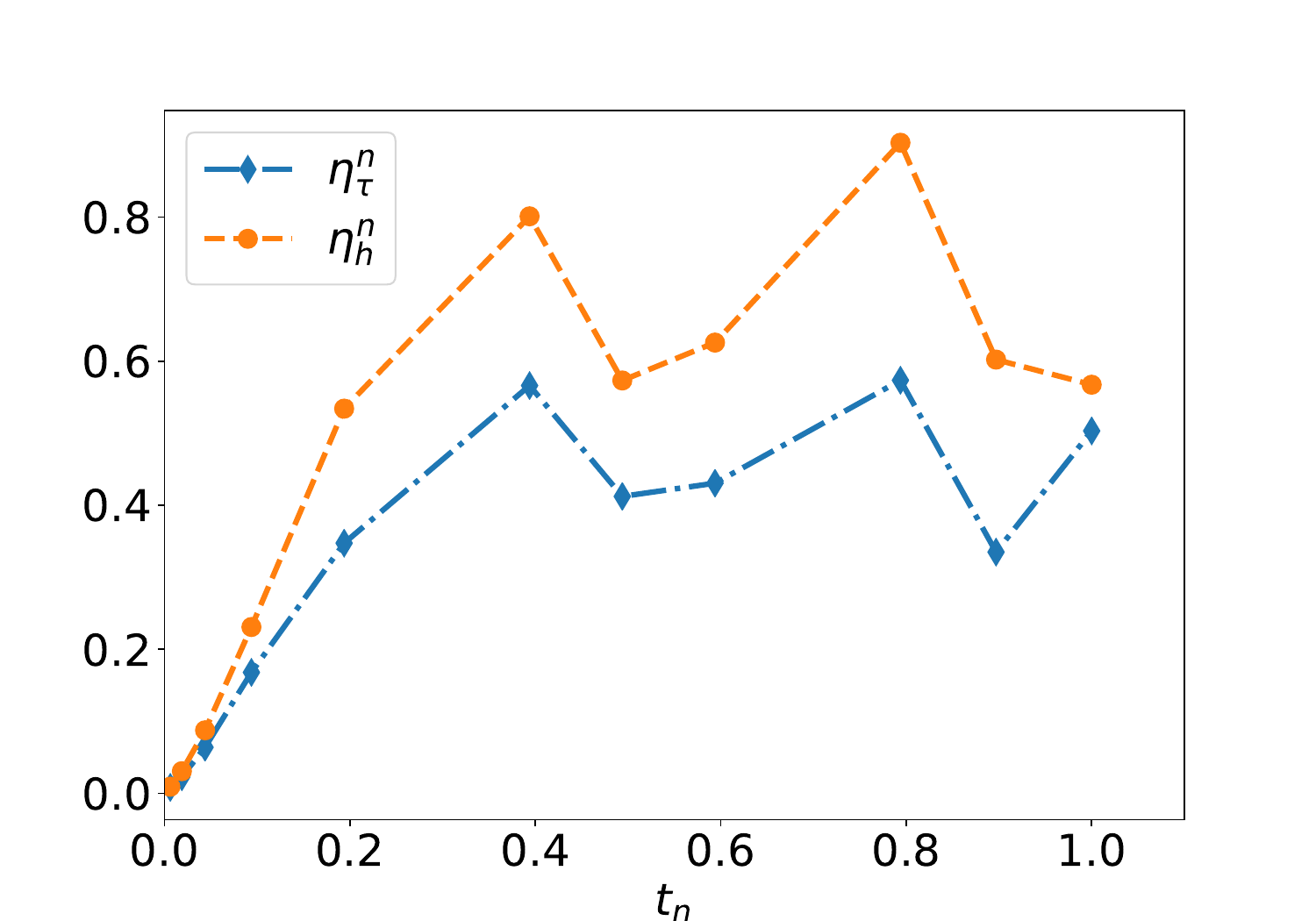}
    \caption{$N = 16$.}
  \label{fig:adaptive:time:1:b}
  \end{subfigure}
  \begin{subfigure}[b]{0.49\textwidth}
    \centering
    \includegraphics[width=\textwidth]{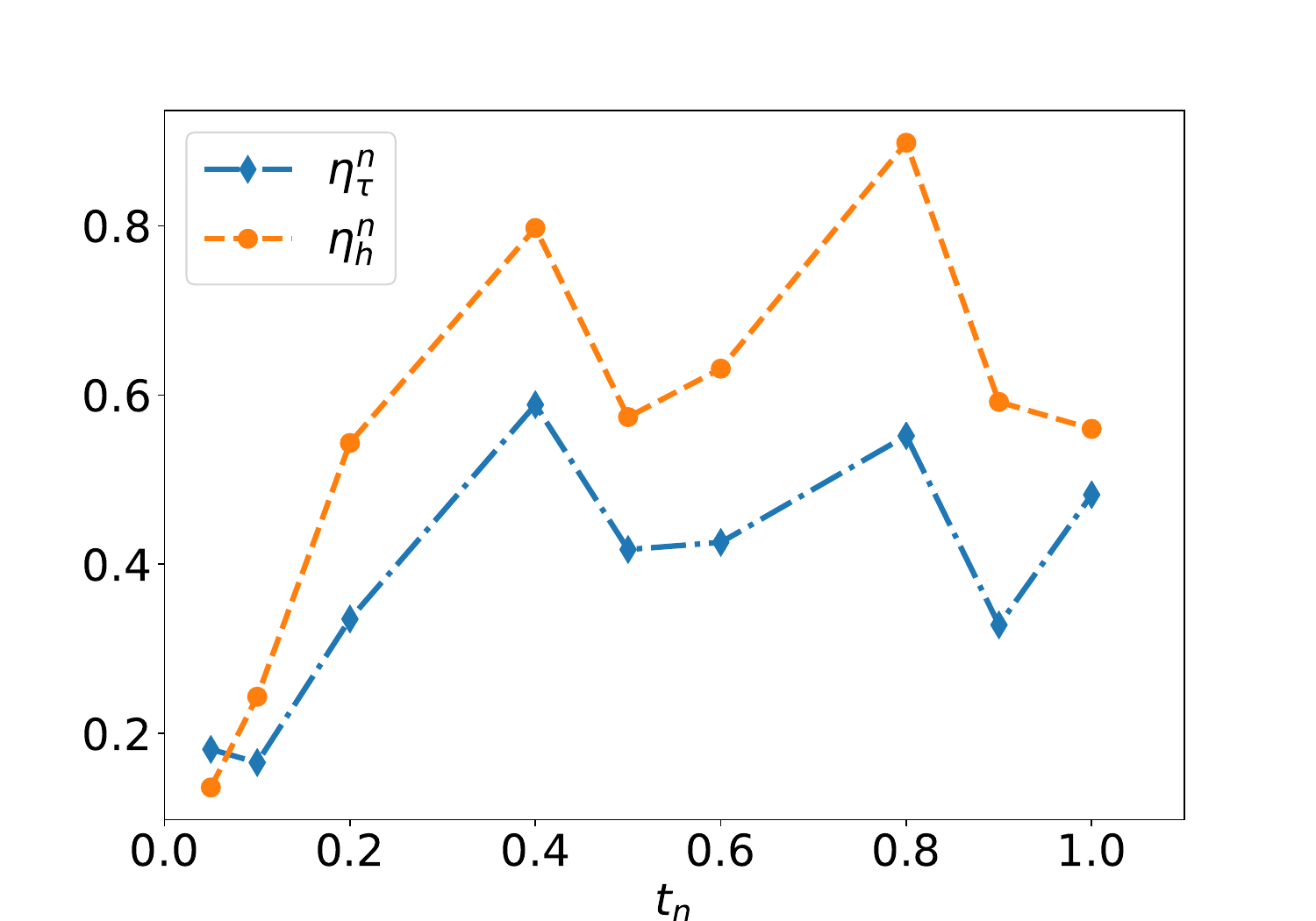}
    \caption{$N = 16$ and $\tau_{\min} = \tau^0/4$.}
  \label{fig:adaptive:time:1:c}
  \end{subfigure}
  \begin{subfigure}[b]{0.49\textwidth}
    \centering
    \includegraphics[width=\textwidth]{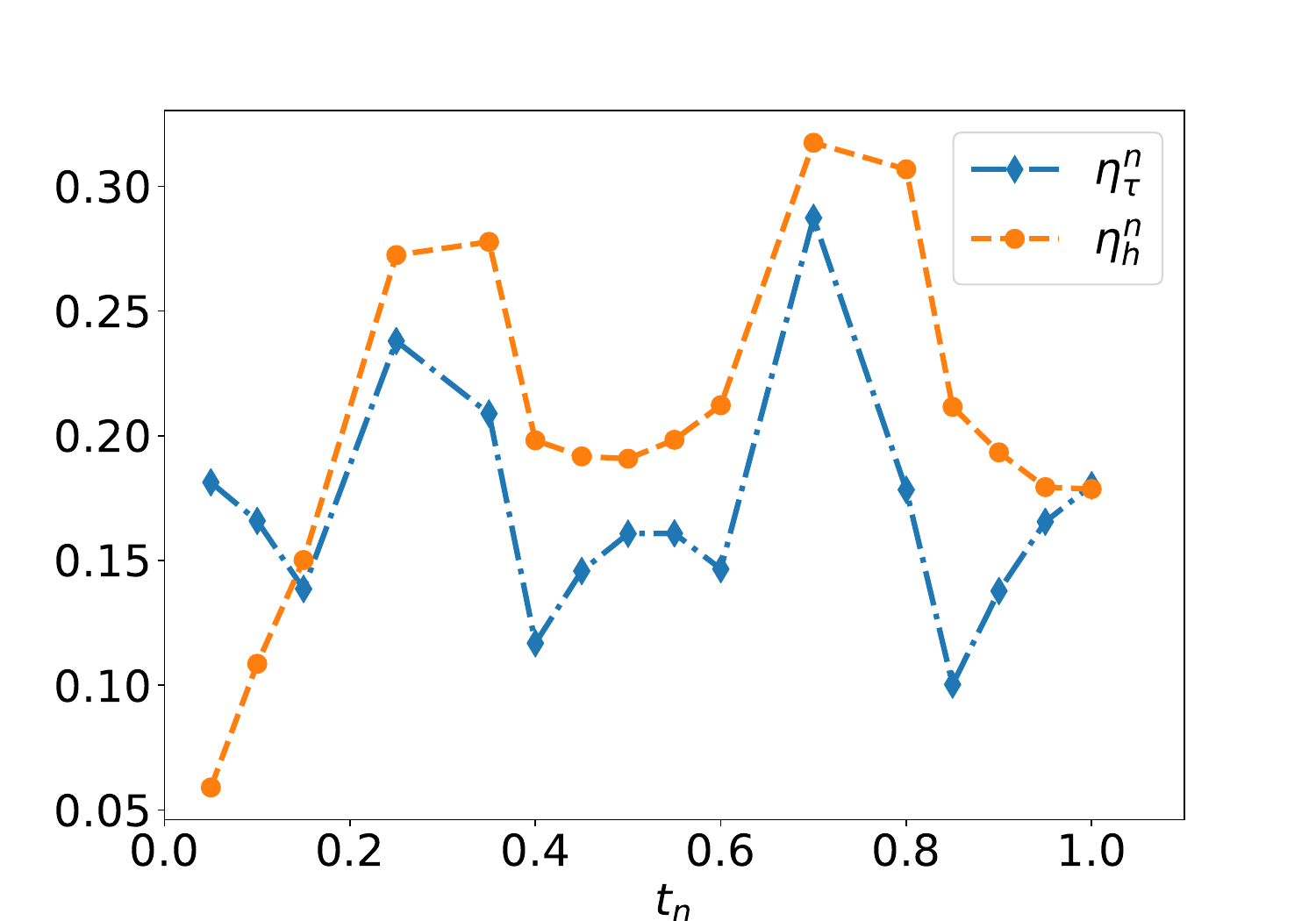}
    \caption{$N = 32$ and $\tau_{\min} = \tau^0/4$.}
  \label{fig:adaptive:time:1:d}
  \end{subfigure}
  \caption{Evaluation of adaptive time stepping for a smooth numerical
    test case, given uniform meshes and different adaptive parameter
    configurations. All plots show the approximated error estimators
    $\eta_h^n$ and $\eta_{\tau}^n$ at each time step versus adaptive
    times $t^n$. }
  \label{fig:adaptive:time:1}
\end{figure}

\subsection{Spatial adaptivity}

For the spatial adaptivity, we use adaptive mesh (h-)refinement
based on local error indicators $\{ \eta_K \}_{K \in \mathcal{T}_h}$
derived from the global error estimators~\eqref{eq:etas}. In light of
the theoretical and empirical observation that $\eta_4$ primarily
contributes to the temporal error, we will rely on local contributions
to $\eta_1, \eta_2$ and $\eta_3$ only for the local error
indicators. Specifically, we will let
\begin{equation}
  \eta_K = \eta_{1, K} + \eta_{2, K} + \eta_{3, K}
  \label{eq:eta_K}
\end{equation}
where
\begin{align}
  \eta_{1, K} = \left ( \sum_{n = 1}^N \tau_n \eta_{p, K}^n \right )^{\frac12}, 
  \quad
  \eta_{2, K} = \left ( \sup_{0 \leq n \leq N} \eta_{u, K}^n \right )^{\frac12},
  \quad
  \eta_{3, K} = \sum_{n = 1}^N \tau_n \left (\eta_{u, K}^n (\delta_t) \right )^{\frac12}.
\end{align}
for each $K \in \mathcal{T}_h$. The complete space-time adaptive
algorithm is given in Algorithm~\ref{alg:space}. We here choose to use
D\"orfler marking~\citep{dorfler1996convergent} or a maximal marking
strategy in which the $\gamma_M$ of the total number of cells with the
largest error indicators are marked for refinement, but other marking
strategies could of course also be used.
\begin{algorithm}
\begin{algorithmic}[1]
  \State Assume that an error tolerance $\epsilon$ and/or a resource
  limit $L$ and an initial mesh $\mathcal{T} = \mathcal{T}_h^0$ are
  given. Set a marking fraction parameter $\gamma_M \in (0, 1]$.

  \While{True}

   \State Set the parameters ($\tau_0, \changed{\alphafac}, \beta, \tau_{\max},
   \tau_{\min}$) required by Algorithm~\ref{alg:time}.
   
  \State Solve~\eqref{eq:discrete} over $\mathcal{T}$ via the time-adaptive scheme defined by Algorithm~\ref{alg:time}.

  \State Estimate the error $\eta = \eta_1 + \eta_2 + \eta_3 + \eta_4$
  where $\eta_i$ for $i = 1, 2, 3, 4$ are given by~\eqref{eq:etas}.

  \If{$\eta < \epsilon$}
  \State Break
  \EndIf

  \State Compute spatial error estimators $\eta_K$ for $K \in
  \mathcal{T}$ via~\eqref{eq:eta_K}.

  \State From $\{ \eta_K \}$, define Boolean refinement markers $\{
  y_K \}_{K \in \mathcal{T}_h}$ via D\"orfler or maximal marking (with
  $\gamma_M$).

  \State Refine $\mathcal{T}$ (locally) based on the markers $\{ y_K
  \}$.

  \If{$|\mathcal{T}| > L$}
    \State Break
  \EndIf
  
  \EndWhile
\end{algorithmic}
\caption{Space-time adaptive algorithm}
\label{alg:space}
\end{algorithm}
\changed{
\begin{rem}
  In Algorithm~\ref{alg:space}, we suggest adapting the mesh in each
  outer iteration via only (local) mesh refinements. One could equally
  well consider a combination of local mesh refinement and coarsening,
  and/or other adaptive mesh techniques such as r-refinement. Indeed,
  this could be particularly relevant in connection with complex
  geometries, for which the initial mesh may be overly fine (in terms
  of approximation power) in geometrically involved local
  regions. 
\end{rem}
}

\subsubsection*{Evaluation of the space-time-adaptive algorithm on a smooth numerical test case}
We evaluate Algorithm~\ref{alg:space} using the numerical test case
with smooth solutions defined over 3 networks as introduced in
Section~\ref{sec:numerical-experiments} with the default material
parameters. As this is a smooth test case in a regular domain, we
expect only moderate efficiency improvements (if any) from adaptive
mesh refinement, and therefore primarily evaluate the accuracy of the
error estimators on adaptively refined meshes and the balance between
temporal and spatial adaptivity.

\begin{figure}
  \begin{subfigure}[t]{0.49\textwidth}
    \centering
    \captionsetup{width=.8\linewidth}%
    \includegraphics[width=\textwidth]{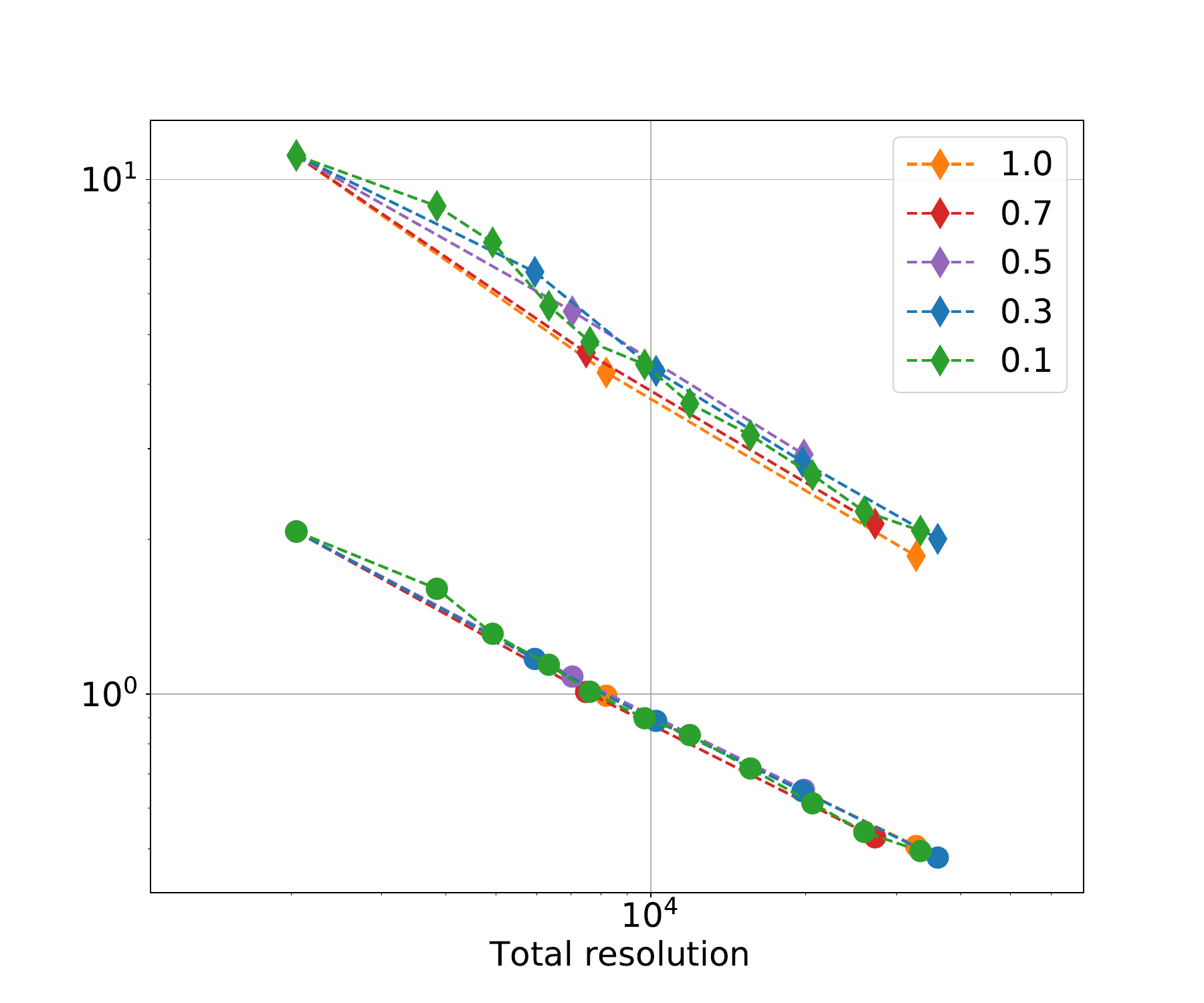}
    \caption{Approximation error $E$ (dots) and error estimates $\eta$
      (diamonds) versus total resolution ($|\mathcal{T}| \times N$) at
      different adaptive iterations with a uniform timestep $\tau =
      1/64$ for different D\"orfler marking fractions ($1.0, 0.7, 0.5,
      0.3, 0.1$).}
    \label{fig:adaptive:space:1:a}
  \end{subfigure}
  \begin{subfigure}[t]{0.49\textwidth}
    \centering
    \captionsetup{width=.8\linewidth}%
    \includegraphics[width=\textwidth]{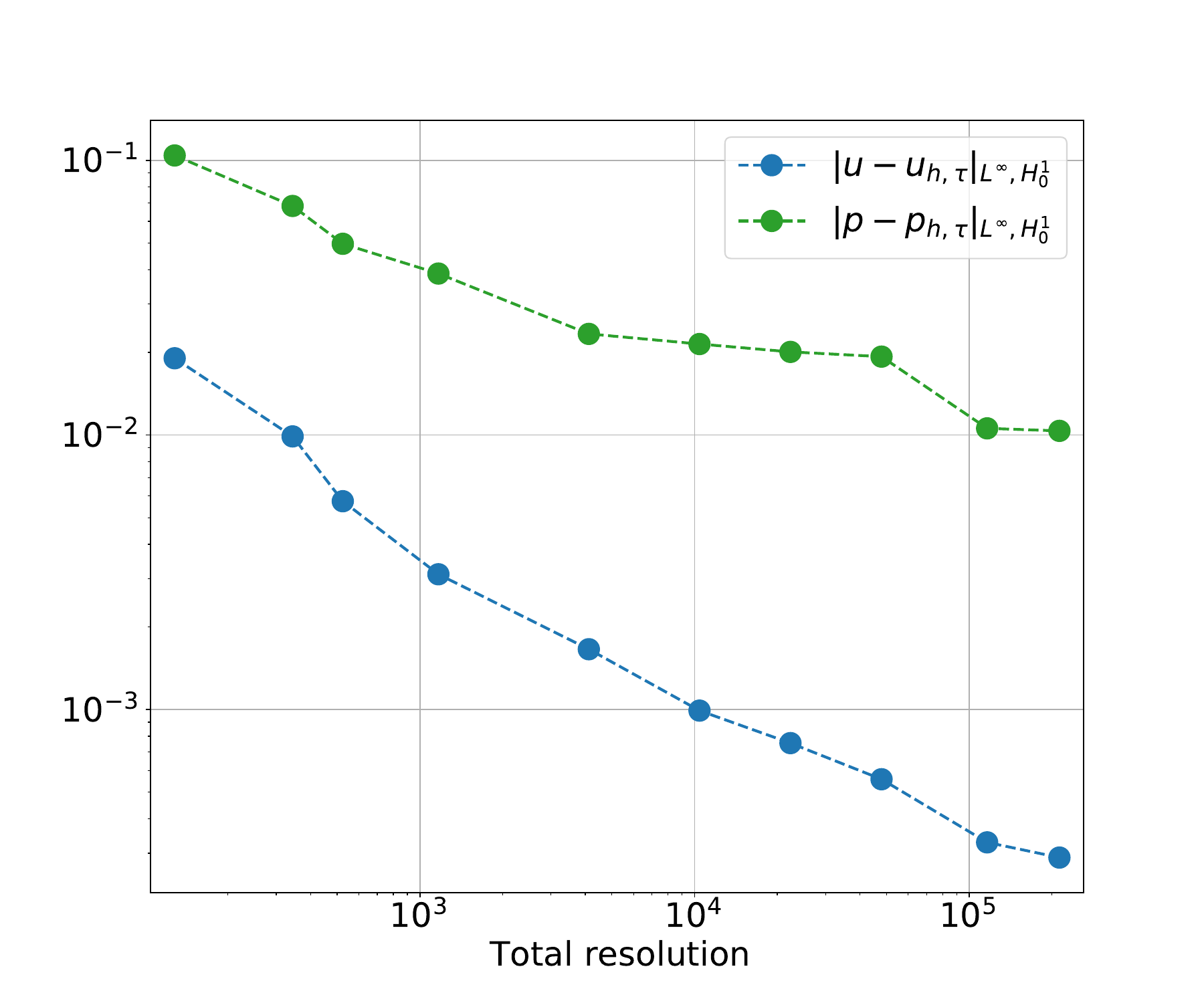}
    \caption{Displacement and pressure approximation errors under
      adaptive refinement as generated by the space-time adaptive
      algorithm with $\tau_0 = T/4$, $\changed{\alphafac} = 0.3$, $\beta = 2.0$,
      $\gamma_M = 0.3$, $\tau_{\max} = \tau_0$, $\tau_{\min} =
      \tau_0/16$, and $L \approx 8000$.}
    \label{fig:adaptive:space:1:b}
  \end{subfigure}
  \begin{subfigure}[t]{0.49\textwidth}
    \centering
    \captionsetup{width=.8\linewidth}%
    \includegraphics[width=0.8\textwidth]{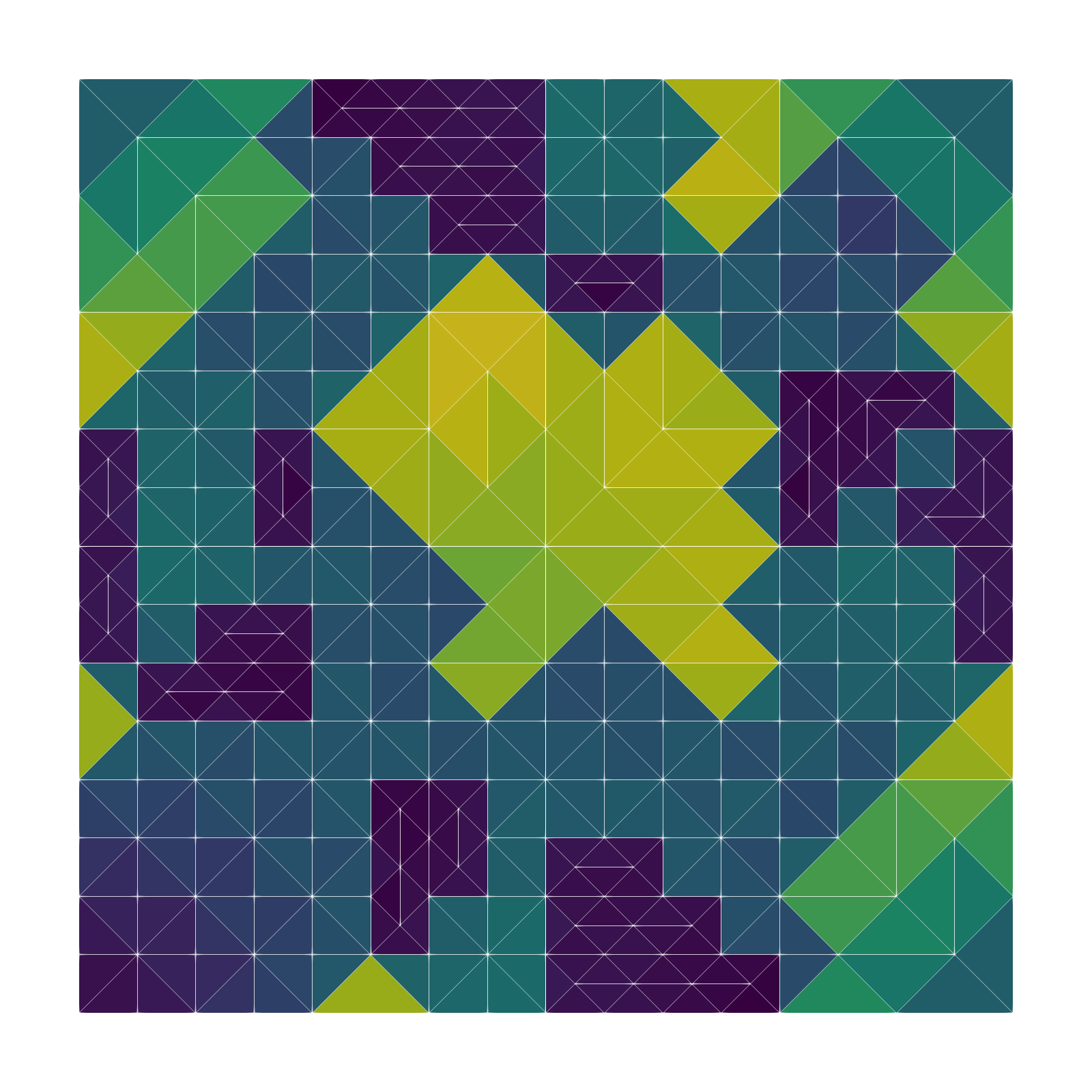}
    \caption{Error indicators on final mesh refinement level with
      yellow values indicating high error indicators (colormap:
      viridis) (same parameters as in
      Figure~\ref{fig:adaptive:space:1:b}).}
    \label{fig:adaptive:space:1:c}
  \end{subfigure}
  \begin{subfigure}[t]{0.49\textwidth}
    \centering
    \captionsetup{width=.8\linewidth}%
    \includegraphics[width=\textwidth]{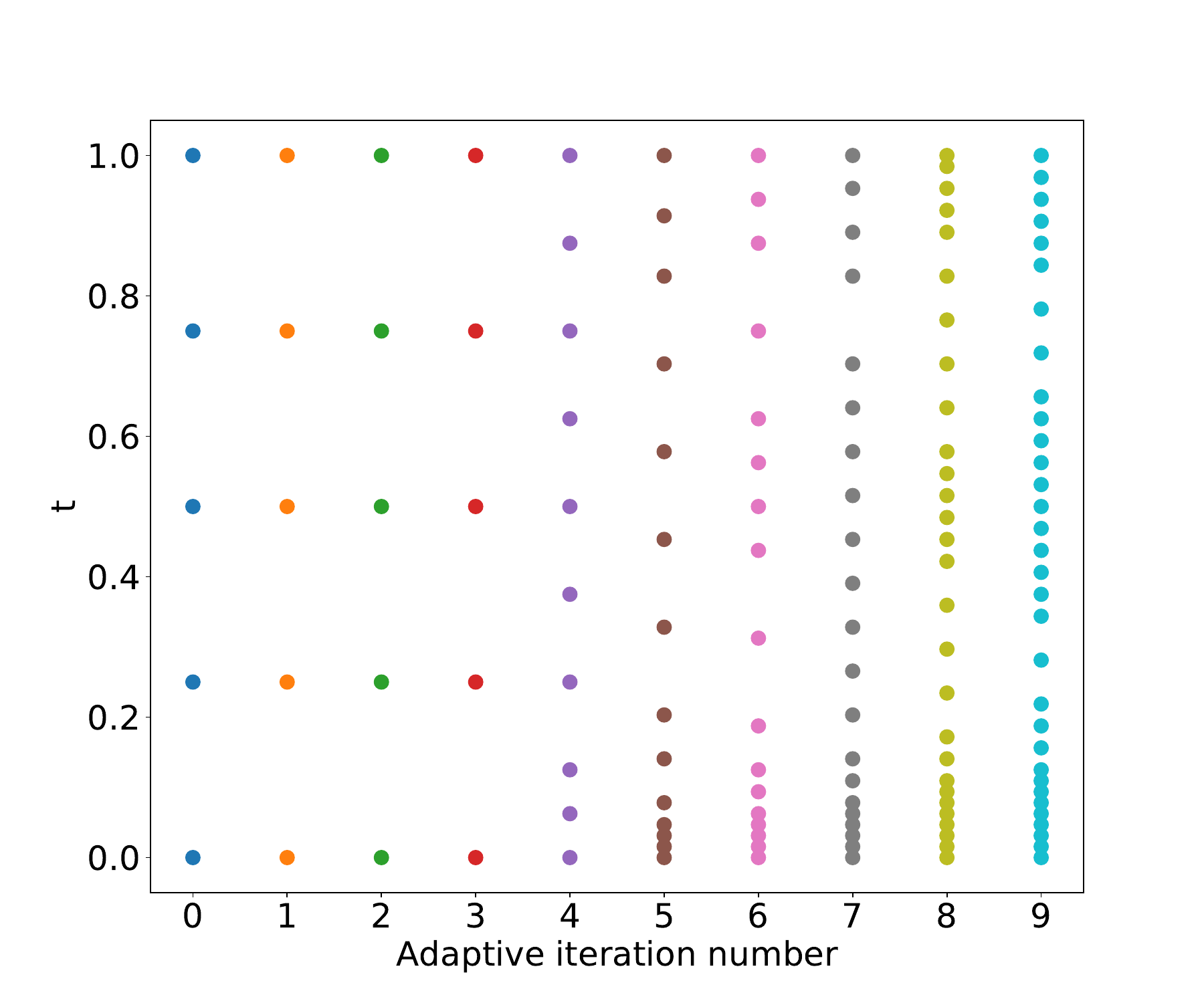}
    \caption{Discrete times (and time steps) generated by the
      space-time adaptive algorithm (same parameters as in
      Figure~\ref{fig:adaptive:space:1:b}).}
    \label{fig:adaptive:space:1:d}
  \end{subfigure}
  \caption{Evaluation of the space-time adaptive algorithm on a smooth test case.}
  \label{fig:adaptive:space:1}
\end{figure}
We set $T = 1.0$, $\tau_0 = 0.5$, and begin with a $2 \times 4 \times
4$ mesh of the unit square as $\mathcal{T}_h^0$. We set $\epsilon =
0$, but instead prescribe a resource tolerance \changed{$L$}. We first set a
fine initial time step $\tau_0 = T/64$, and let $\beta = 2.0$, $\changed{\alphafac}
= 0.3$, $\tau_{\min} = \tau_0/16$, and $\tau_{\max} = \tau_0$ in
Algorithm~\ref{alg:time}. We note that a D\"orfler marking fraction
$\gamma_M$ of $1.0$ yields a series of uniformly refined meshes. For
marking fractions between 0.0 and 1.0, we obtain adaptively refined
meshes, yet for this test case, the time step remains uniform
throughout the adaptive loop. The resulting errors and error estimates
at each adaptive refinement iteration are shown in
Figure~\ref{fig:adaptive:space:1:a}. We observe that the errors decay
as expected, and that the error estimates provide upper bounds for the
errors $E$ at each refinement level for all marking fractions tested.

We next let $\tau_0 = T/4$ and $\gamma_M = 0.3$ (and all other
parameters as before), and consider the results of the adaptive
algorithm
(Figure~\ref{fig:adaptive:space:1:b}--\ref{fig:adaptive:space:1:d}). We
find that the adaptive algorithm keeps the initial time step and
refines the mesh only for the first 4 iterations, which substantially 
\changed{reduces} the $L^{\infty} H^1_0$ displacement approximation 
error and moderately reduces the $L^{\infty} H^1_0$ pressure 
approximation error. For the next iterations, both the mesh and the 
time step is refined. The pressure errors seem to plateau before 
continuing to reduce given sufficient mesh refinement, while the 
displacement errors steadily decrease.

\section{Adaptive brain modelling and simulation}
\label{sec:brain}

We turn to consider a physiologically and computationally realistic
scenario for simulating the poroelastic response of the human
brain. Human brains form highly non-trivial, non-convex domains
characterized by narrow gyri and deep sulci, and as such represent a
challenge for mesh generation algorithms. Therefore, brain meshes are
typically constructed to accurately represent the surface geometry,
without particular concern for numerical approximation properties. We
therefore ask whether the adaptive algorithm presented here can
effectively and without further human intervention improve the
numerical approximation of key physiological quantities of interest
starting from a moderately coarse initial mesh and initial time
step.

Specifically, we let $\Omega$ be defined by a subject-specific left
brain hemisphere mesh (Figure~\ref{fig:brain:a}) generated from
MRI-data via FreeSurfer~\citep{fischl2012freesurfer} and SVMTK as
described e.g.~in~\citep{mardal2021mathematical}. The domain boundary
is partitioned in two main parts: the semi-inner boundary enclosing
the left lateral ventricle $\partial \Omega_v$ and the remaining
boundary $\partial \Omega_s$ (Figure~\ref{fig:brain:b}).

Over this domain, we consider the MPET equations~\eqref{eq:weak} with
$J = 3$ fluid networks representing an arteriole/capillary network ($j
= 1$), a low-pressure venous network ($j = 2$), and a perivascular
space network ($j = 3$). We assume that the two first networks are
filled with blood, while the third network is filled with
cerebrospinal fluid (CSF). 

\subsection{Pulsatility driven by fluid influx}
\label{sec:brain:source}

We consider a scenario in which fluid influx is represented by a
pulsatile uniform source in the arteriole/capillary network ($j = 1$):
\begin{equation}
  g_1(x, t) = g_1(t) =  \frac12 (1 - \cos (2 \pi t)),
\end{equation}
while we set $g_2 = g_3 = f = 0$. From the arteriole/capillary
network, fluid can transfer either into the venous network or into the
perivascular network with rates $\gamma_{12}, \gamma_{13} > 0$, while
$\gamma_{23} = 0$. All material parameters are given in
Table~\ref{tab:brain:materials}.
\begin{table}
  \centering
  \begin{tabular}[width=\linewidth]{lll}
    \toprule 
    Parameter & Value & Note \\
    \midrule
    $E$ (Young's modulus) & $1.642 \times 10^3$ Pa & \citep{budday2015mechanical} (gray/white average)\\
    $\nu$ (Poisson's ratio) & $0.497$ & $\ast$ \\
    $s_1$ (arteriole storage coefficient) & $2.9 \times 10^{-4}$ Pa$^{-1}$ & \citep{guo2019validation} (arterial network) \\ 
    $s_2$ (venous storage coefficient) & $1.5 \times 10^{-5}$ Pa$^{-1}$ & \citep{guo2019validation} (venous network) \\ 
    $s_3$ (perivascular storage coefficient) & $2.9 \times 10^{-4}$ Pa$^{-1}$ & \citep{guo2019validation} (arterial network) \\ 
    $\alpha_1$ (arteriole Biot-Willis parameter) & $0.4$ & $\ast$ \\
    $\alpha_2$ (venous Biot-Willis parameter) & $0.2$ & $\ast$ \\
    $\alpha_3$ (perivascular Biot-Willis parameter) & $0.4$ & $\ast$ \\
    $\kappa_1$ (arteriole hydraulic conductance) & $3.75 \times 10^{-2}$ mm$^2$ Pa$^{-1}$ s$^{-1}$ & $k_1$/$\mu_1$, see below \\
    $\kappa_2$ (venous hydraulic conductance) & $3.75 \times 10^{-2}$ mm$^2$ Pa$^{-1}$ s$^{-1}$ & $k_2$/$\mu_1$, see below \\
    $\kappa_3$ (perivascular hydraulic conductance) & $1.43 \times 10^{-1}$ mm$^2$ Pa$^{-1}$ s$^{-1}$ & $k_3$/$\mu_3$, see below  \\
    $\gamma_{12}$ (arteriole-venous transfer) & $1.0 \times 10^{-3}$  Pa$^{-1}$ s$^{-1}$ & $\ast$ \\
    $\gamma_{13}$ (arteriole-perivascular transfer) & $1.0 \times 10^{-4}$ Pa$^{-1}$ s$^{-1}$ & $\ast$ \\
    $C$ (environment compliance) & $10$ & $\ast$ \\ 
    $R$ (environment resistance) & $79.8$ Pa / (mm$^3$ / s) & \citep{vinje2020intracranial} \\
    \midrule
    $k_1$ (arteriole permeability) & $1.0 \times 10^{-10}$ m$^2$ & \citep{guo2019validation} (arterial network) \\
    $k_2$ (venous permeability) & $1.0 \times 10^{-10}$ m$^2$ & \citep{guo2019validation} (venous network) \\
    $k_3$ (perivascular permeability) & $1.0 \times 10^{-10}$ m$^2$ & Estimate, vascular permeability \\
    $\mu_1$ (blood dynamic viscosity)  & $2.67 \times 10^{-3}$ Pa s & \citep{guo2019validation} (arterial network) \\
    $\mu_3$ (CSF dynamic viscosity) & $6.97 \times 10^{-4}$ Pa s & \citep{daversin2020mechanisms} (water at body temperature) \\
    \bottomrule
  \end{tabular}
  \caption{Material parameters corresponding to a human brain at body
    temperature. The hydraulic conductances $\kappa$ are defined in
    terms of the permeabilities and the fluid viscosities $\kappa_j =
    k_j/\mu_j$, $\mu_2 = \mu_1$. Values marked by the $\ast$ are
    estimates, yielding physiologically reasonable brain
    displacements, fluid pressures, and fluid velocities.}
  \label{tab:brain:materials}
\end{table}

In terms of boundary conditions for the momentum equation, we set
\begin{subequations}
  \begin{align}
    u &= 0 \quad &&\text{ on the outer  boundary } \partial \Omega_{s}, \\
    (\sigma - \ssum_{j} \alpha_j p_j I) \cdot n &= - p_{\rm csf} \, n \quad &&\text{ on the inner boundary } \partial \Omega_{v}.
  \end{align}
  \label{eq:brain:bcs}%
\end{subequations}
for a spatially-constant $p_{\rm csf}$ to be defined below. For the
arteriole space, we assume no boundary flux:
\begin{equation}
  \Grad p_1 \cdot n = 0 \quad \text{ on  } \partial \Omega. 
\end{equation}
We assume that the venous network is connected to a low (zero)
pressure compartment and set:
\begin{equation}
  p_2 = 0 \quad \text{ on } \partial \Omega. 
\end{equation}
We assume that the perivascular space is in direct contact with its environment, and set:
\begin{equation}
  p_3 = p_{\rm csf} \quad \text{ on } \partial \Omega.
  \label{eq:brain:bcs3}
\end{equation}
Last, we model $p_{\rm csf}$ via a simple Windkessel model at the boundary:
\begin{equation}
  C \dot{p}_{\rm csf} = Q - \frac{p_{\rm csf}}{R}
\end{equation}
with compliance $C$ and a resistance $R$ (see
Table~\ref{tab:brain:materials}), and where $Q$ is the outflow: $Q =
\int_{\partial \Omega} u \cdot n \ds$. After an explicit time
discretization, we define at each time step
\begin{equation}
  \label{eq:pcsf}
  C p_{\rm csf}^{n+1} = \tau_n Q^n + (C - \frac{\tau_n}{R}) p_{\rm csf}^n,
\end{equation}
and use~\eqref{eq:pcsf} in~\eqref{eq:brain:bcs}
and~\eqref{eq:brain:bcs3}. Finally, we let all fields start at
zero. We let $T = 2.0$ corresponding to two cardiac cycles, and an
initial time step of $\tau_0 = 0.1$.
\begin{figure}
  \begin{subfigure}[t]{0.4\textwidth}
    \centering
    \captionsetup{width=.9\linewidth}%
    \includegraphics[height=3cm]{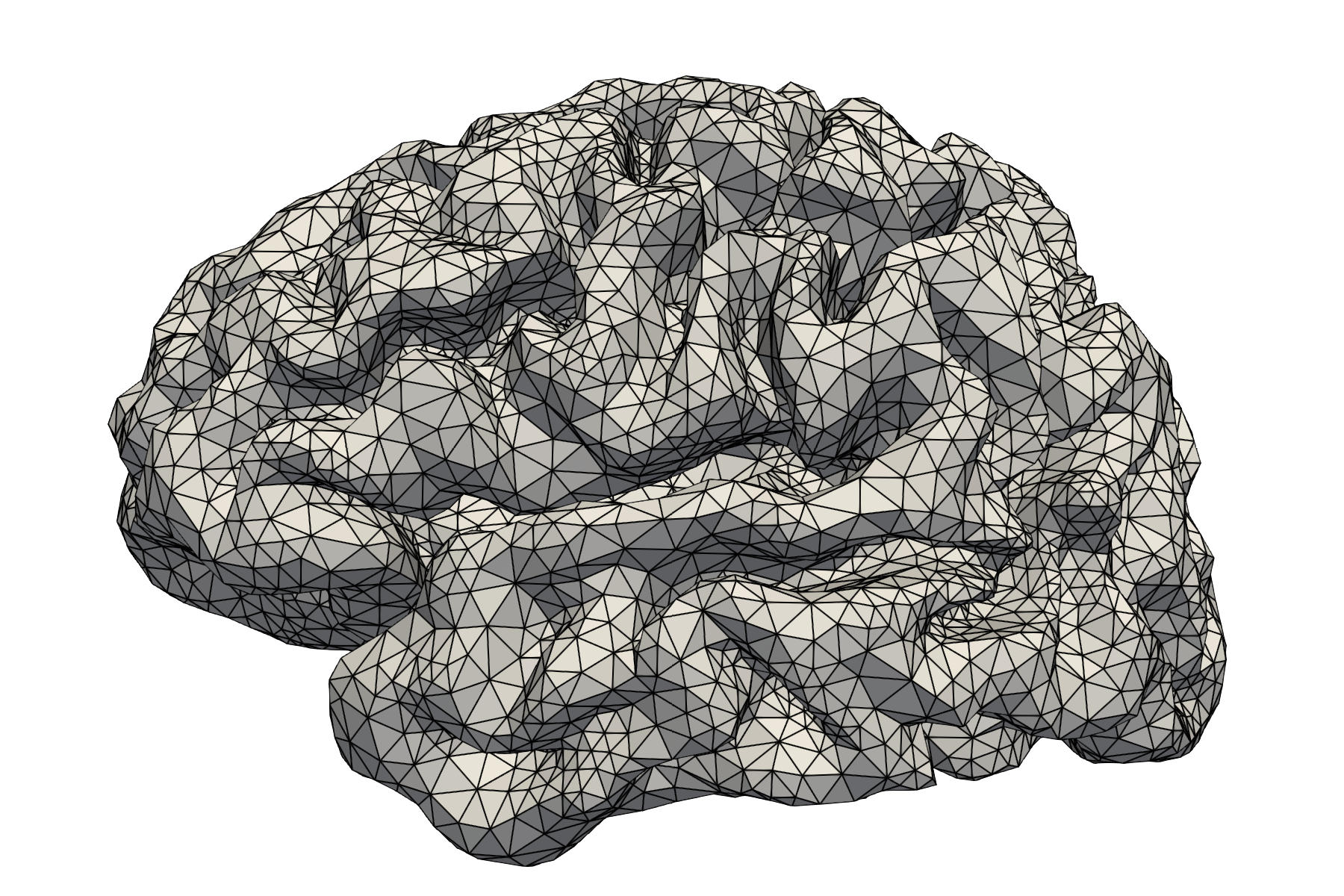}
    \caption{Initial mesh of a brain hemisphere (sagittal view, along
      positive x-axis) with 20911 cells and 6325 vertices, and a
      volume of $4.37 \times 10^5$ mm$^3$.}
    \label{fig:brain:a}
  \end{subfigure}
  \begin{subfigure}[t]{0.59\textwidth}
    \centering
    \captionsetup{width=.9\linewidth}%
    \includegraphics[height=3cm]{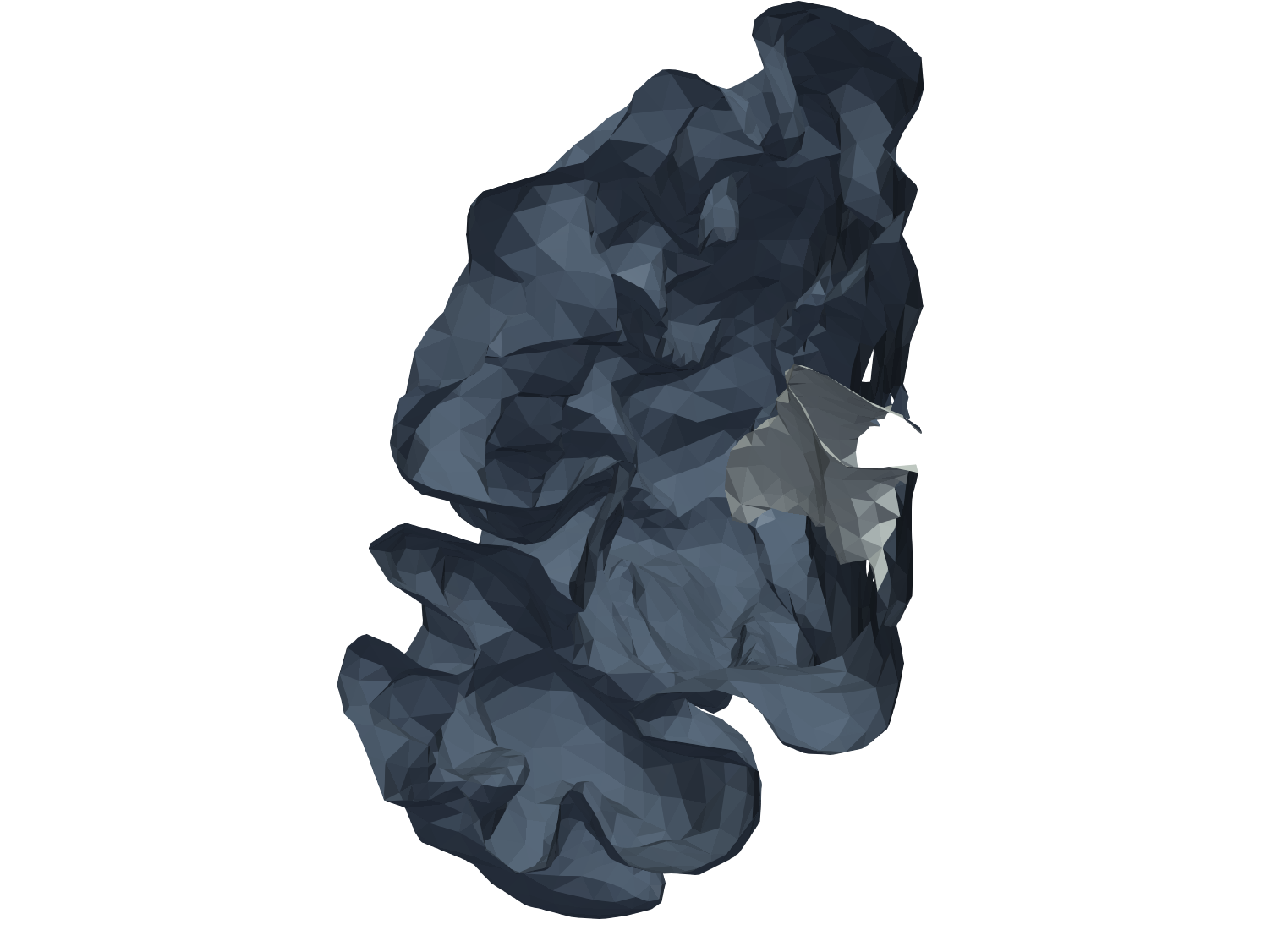}
    \includegraphics[height=3cm]{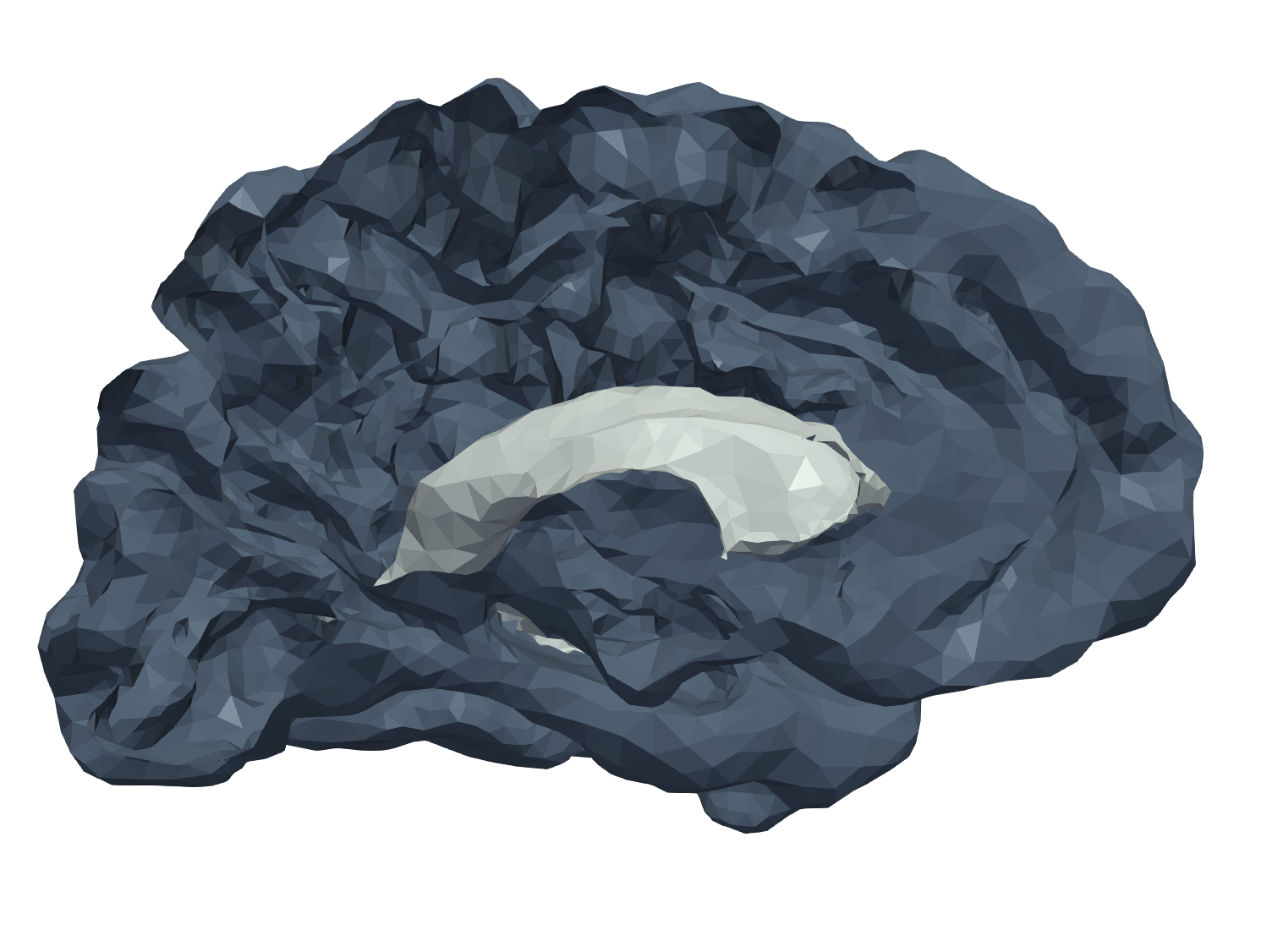}
    \caption{Illustration of the semi-inner ventricular boundary (in
      white), coronal and sagittal clips (view from the y- and
      negative x-axes, respectively).}
    \label{fig:brain:b}
  \end{subfigure}
  \begin{subfigure}[t]{1.0\textwidth}
    \centering
    \captionsetup{width=.9\linewidth}%
    \includegraphics[height=5cm]{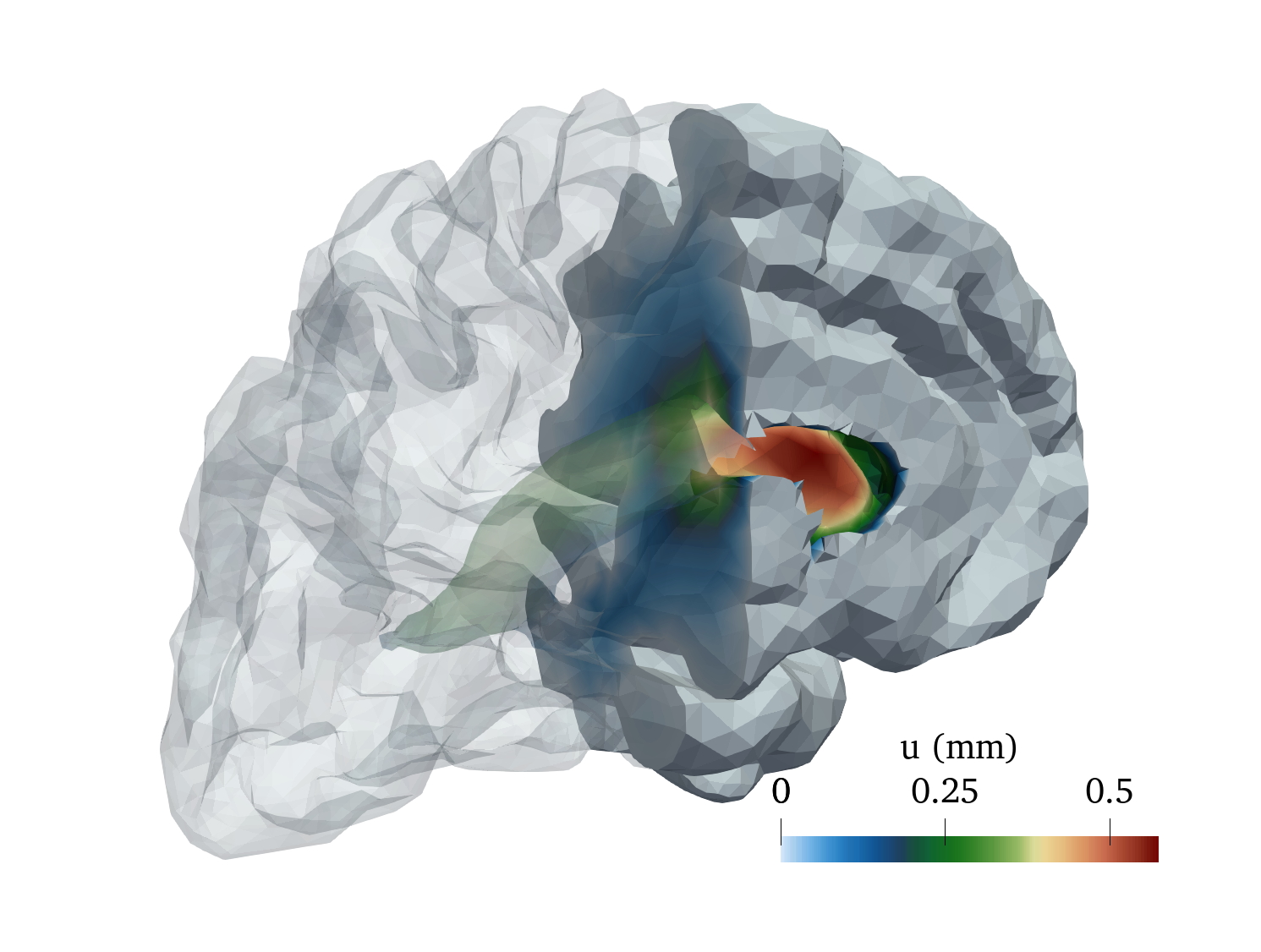}
    \includegraphics[height=5cm]{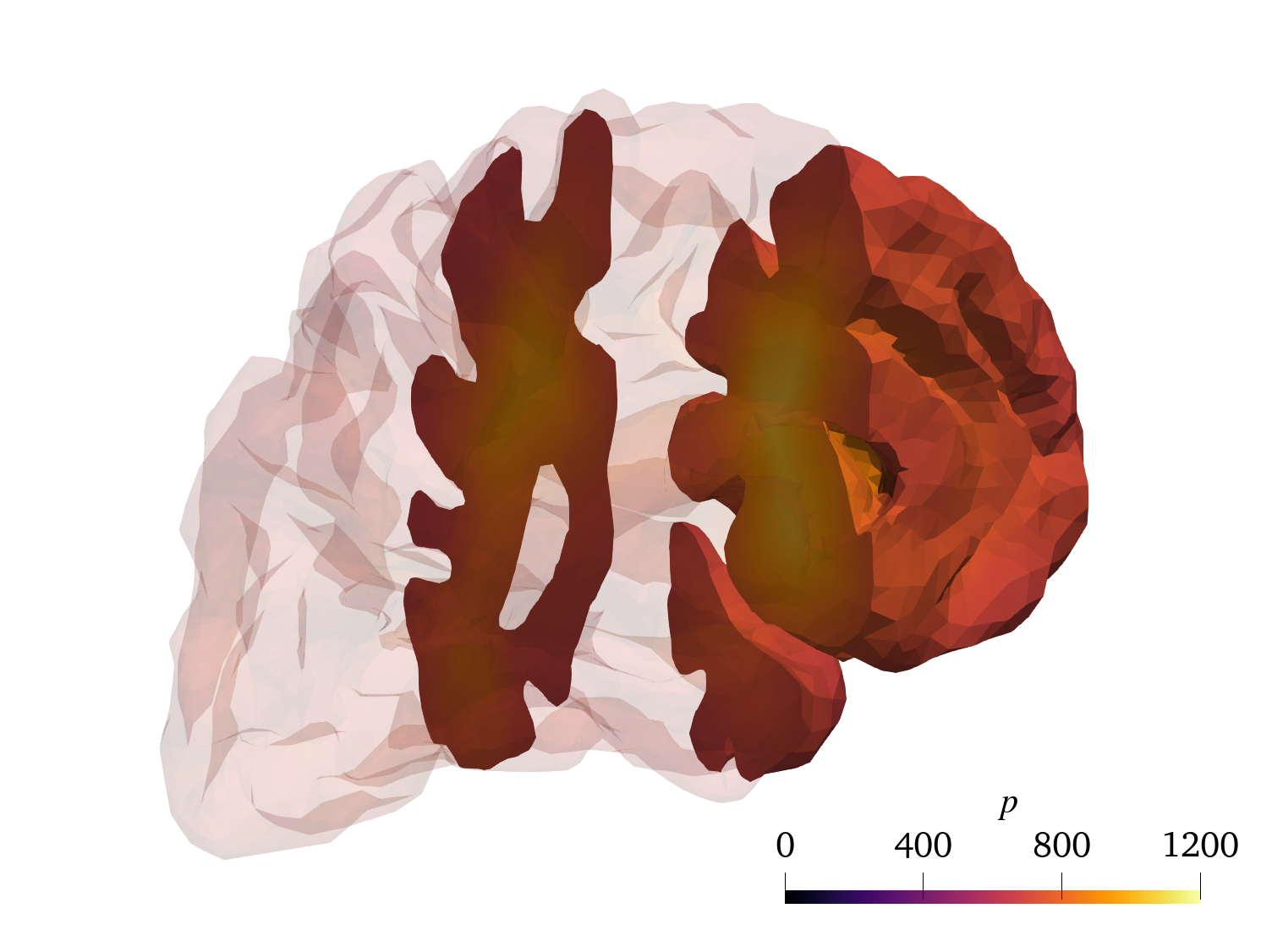} \\
    \includegraphics[height=5cm]{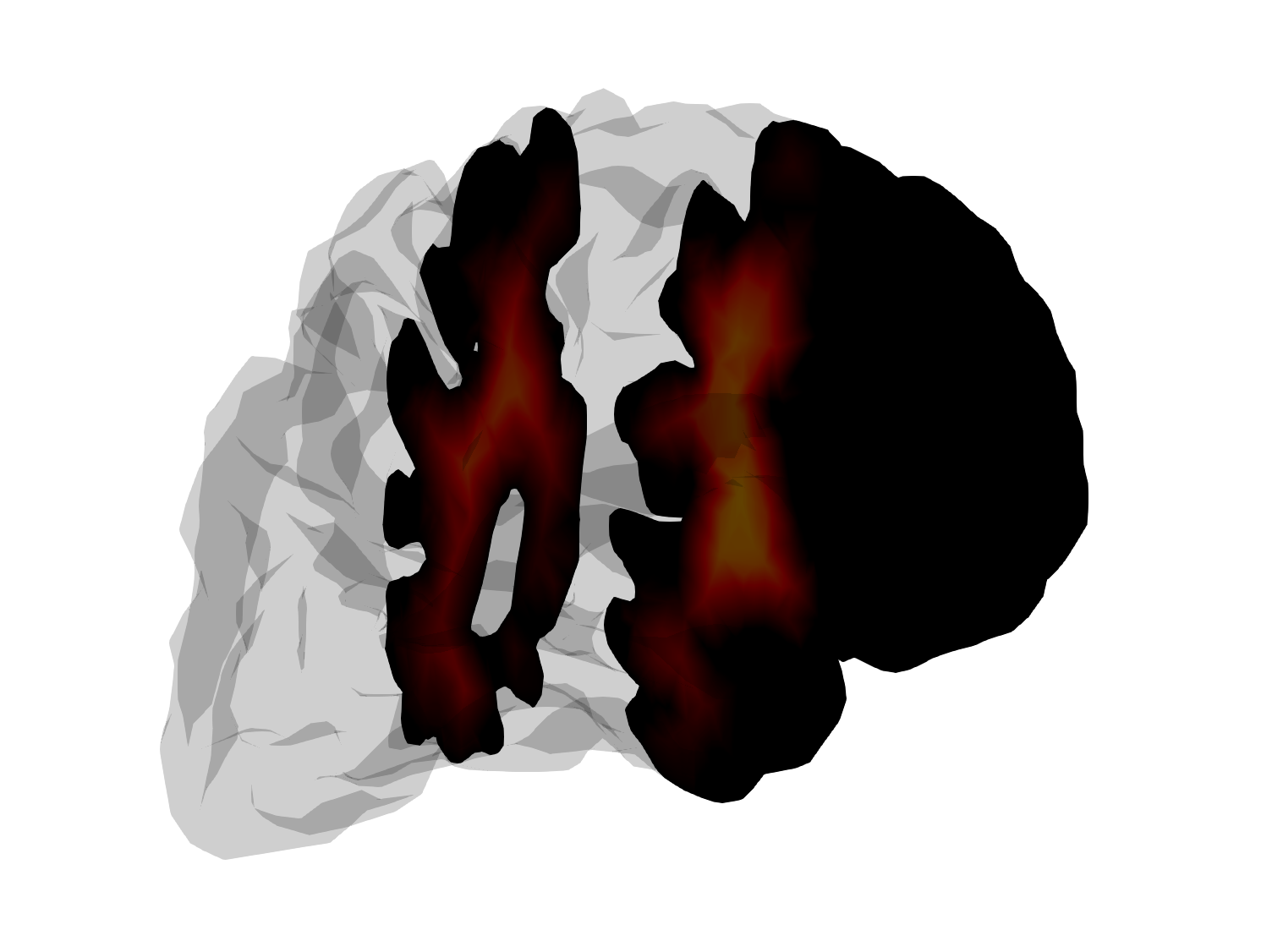}
    \includegraphics[height=5cm]{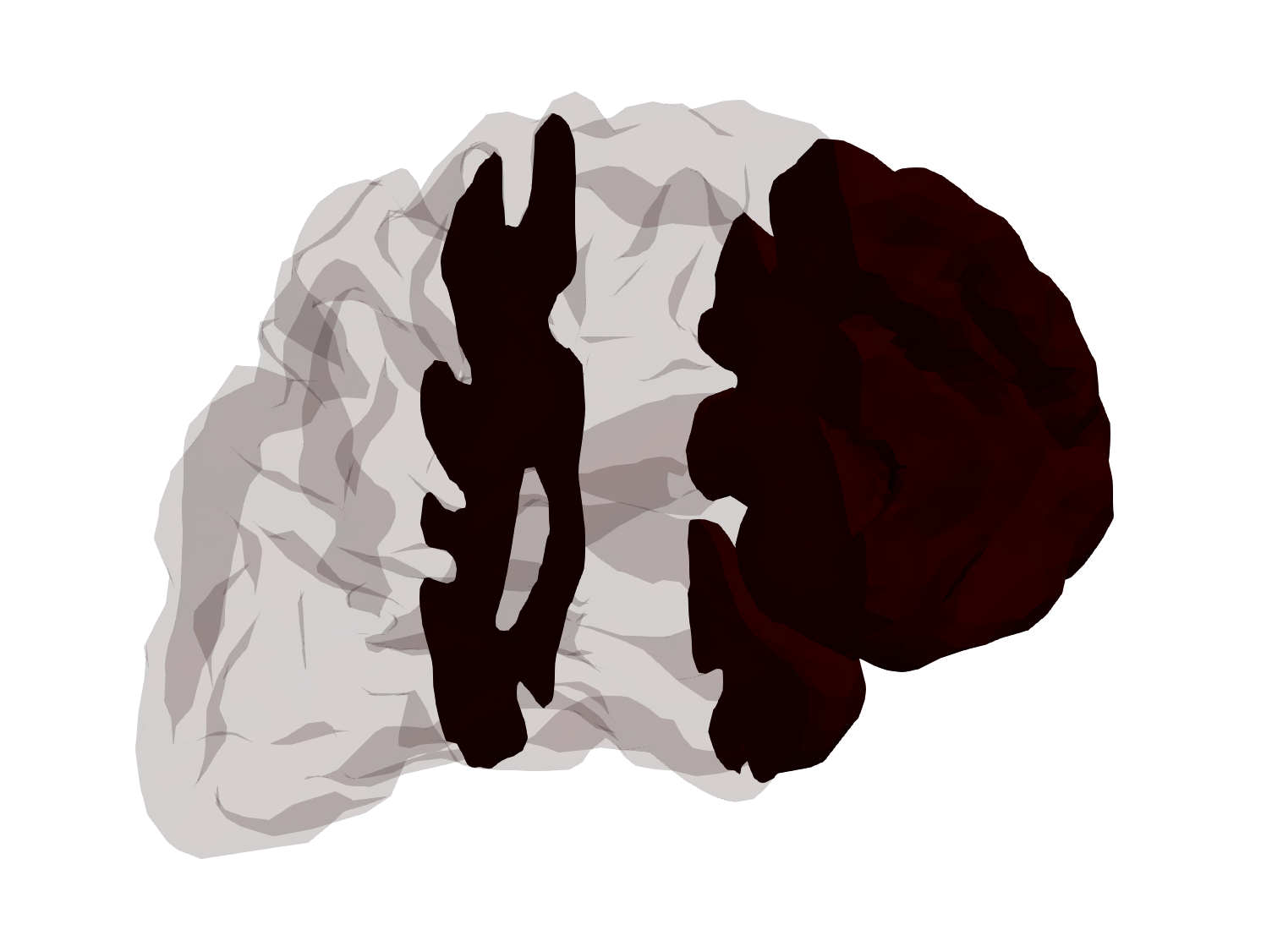}
    \caption{Left to right, top to bottom: Displacement $u$
      magnitude, arteriole/capillary pressure $p_1$, venous pressure
      $p_2$, and perivascular pressure $p_3$ at peak displacement ($T
      = 1.7$).}
    \label{fig:brain:c}
  \end{subfigure}
  \caption{The human brain as a poroelastic medium: meshes,
    boundaries, and snapshots of solution fields.}
\end{figure}

\begin{figure}
  \centering
  \captionsetup{width=.9\linewidth}%
  \includegraphics[width=0.49\textwidth]{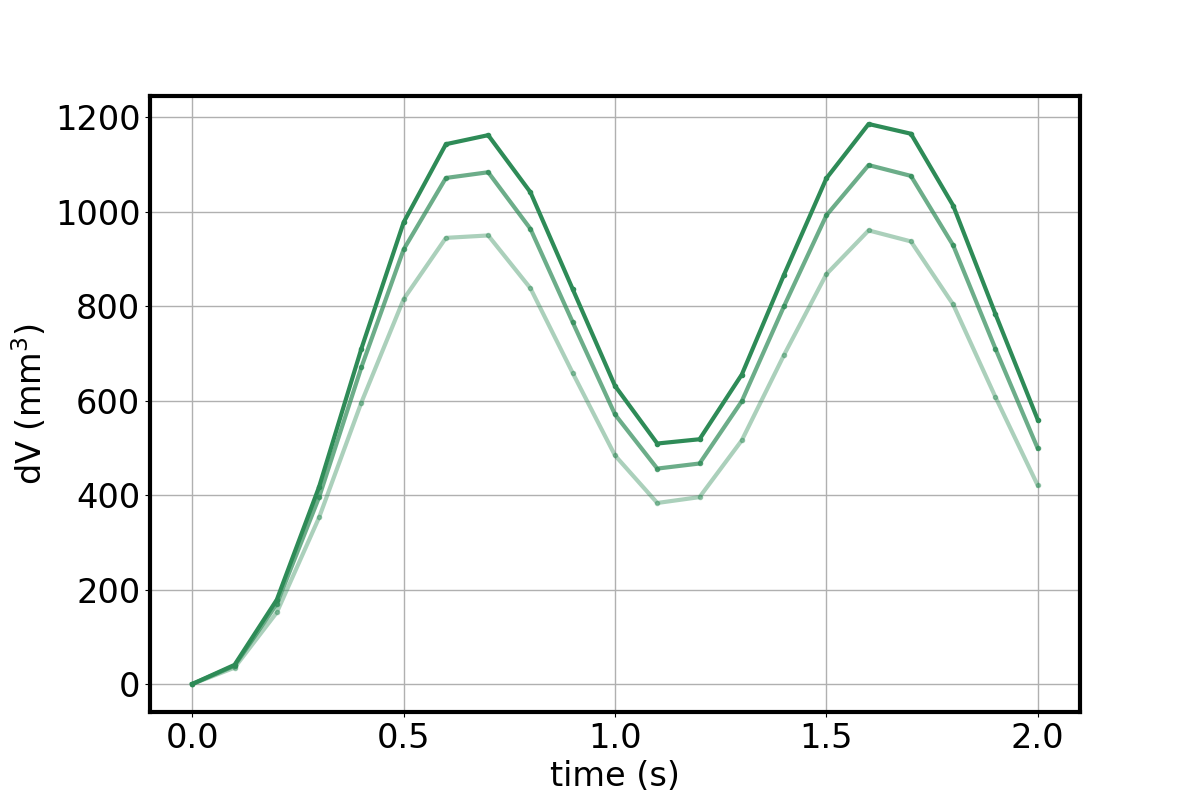}
  \includegraphics[width=0.49\textwidth]{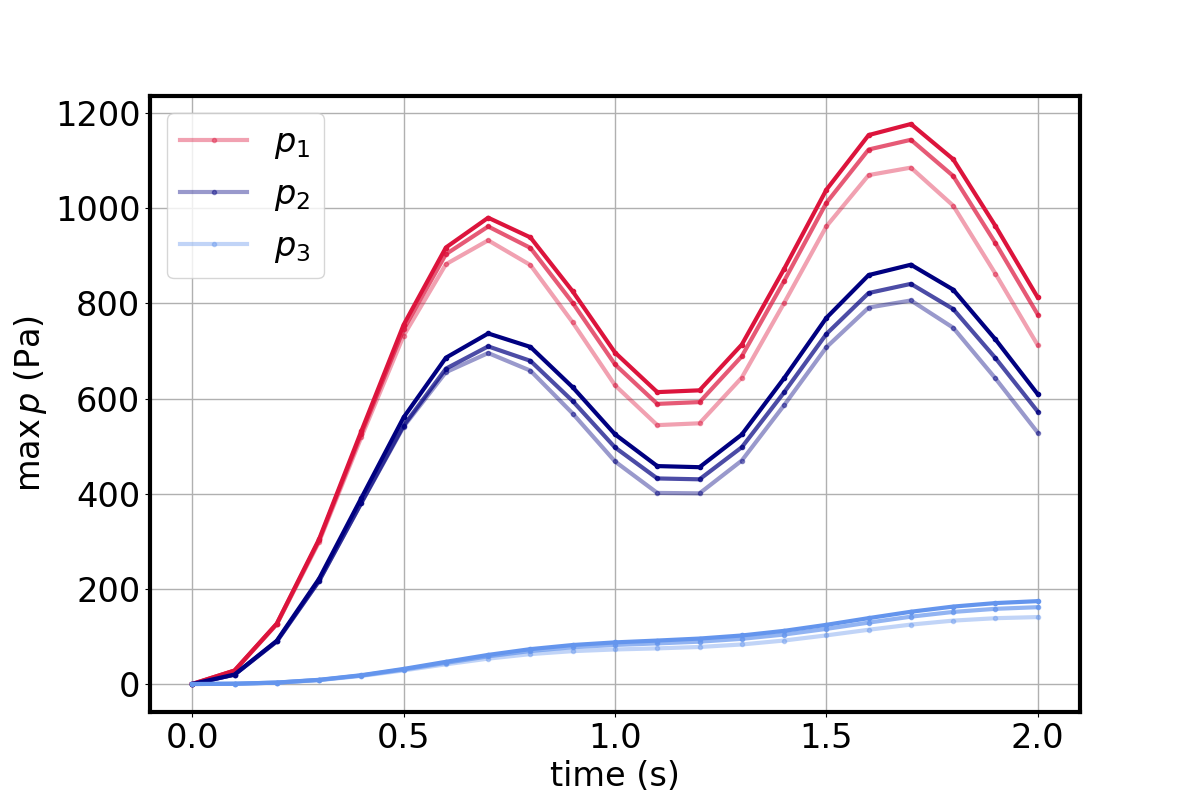}
  \includegraphics[width=0.49\textwidth]{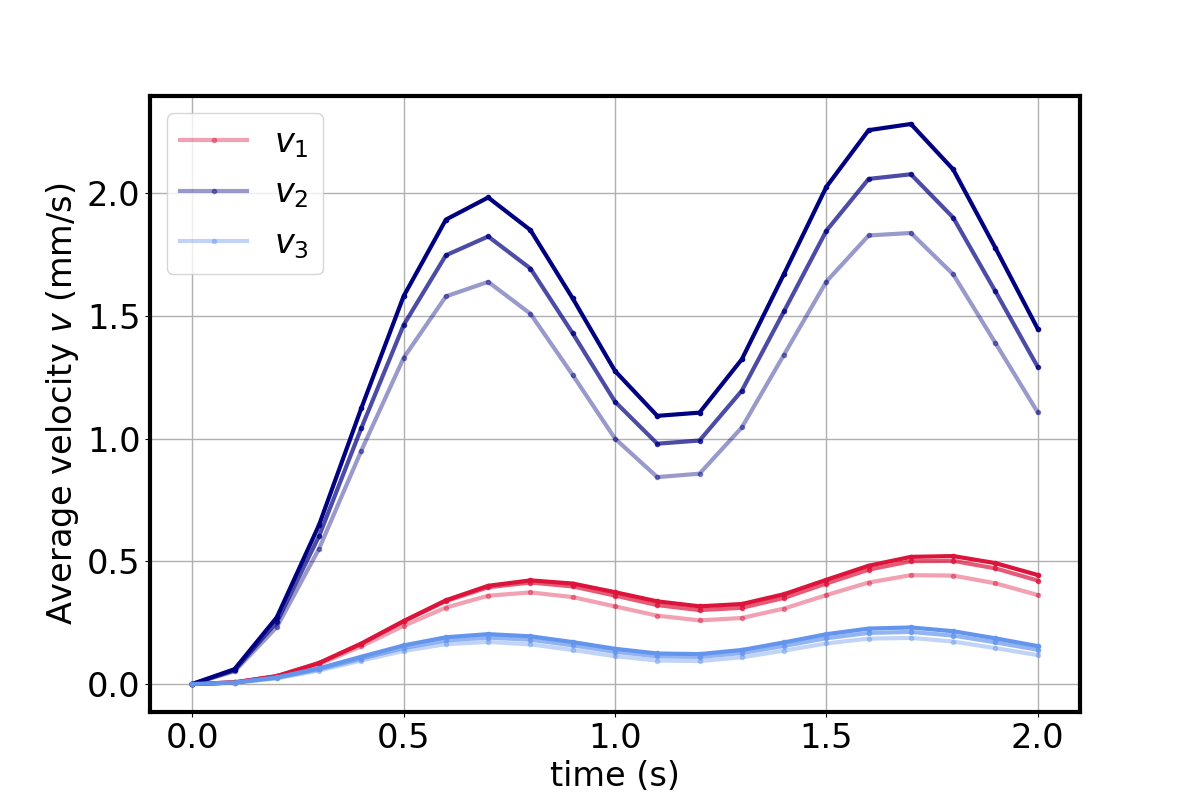}
  \includegraphics[width=0.49\textwidth]{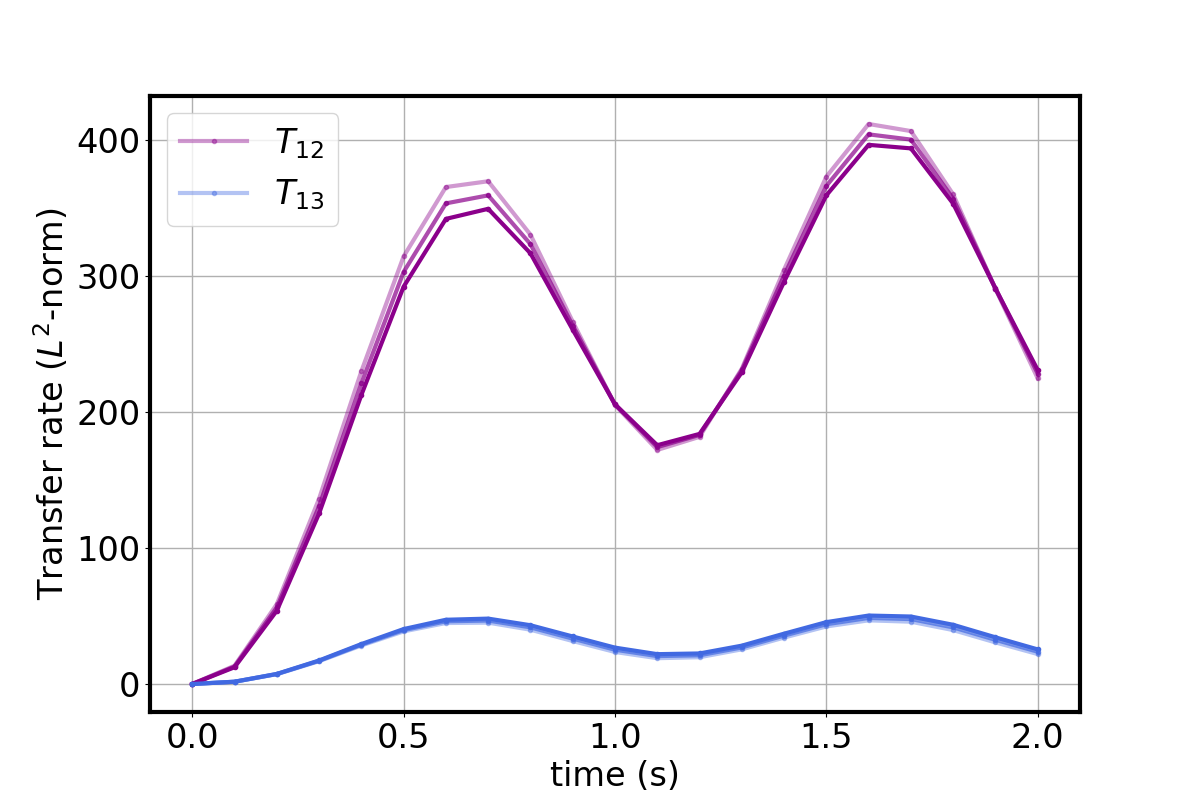}
  \caption{Left to right, top to bottom: Volume change $\textrm{d}V$,
    peak pressure $p_i$ and average velocity $v_i$ for $i = 1, 2, 3$
    and integrated transfer rates $T_{12}$ and $T_{13}$ over time for
    a series of adaptively refined meshes (a1, a2, a3). The opacity
    indicates the adaptive level: the more opaque, the finer the
    mesh.}
  \label{fig:brain:d}
\end{figure}
The given fluid influx induces pulsatile tissue displacements and
pressures in the different networks with varying temporal and spatial
patterns (Figure~\ref{fig:brain:c}, Figure~\ref{fig:brain:d}). The
brain hemisphere expands and contracts with peak changes in volume
\begin{equation*}
  \textrm{dV} = \int_{\Omega} \Div u \dx,
\end{equation*}
of up to 1200 mm$^3$. The largest displacements occur around the
lateral ventricle with peak displacement magnitudes of $\approx$0.5
mm. The arteriole/capillary pressure varies in space and time with a
peak pressure $\max p_1$ of up to 1200 Pa, a pressure pulse amplitude
$\Delta p_1$ of $\approx$ 560 Pa and a pressure difference in space of
$\approx$ 400 Pa. The venous pressure field show similar patterns,
though with lower temporal variations and higher spatial variability
inducing higher venous blood velocities of above 2.0 mm/s
(Figure~\ref{fig:brain:d}). The perivascular pressure shows a steady
increase of up to $\approx$ 200 Pa at $T = 2.0$, but only moderate
pulsatility and lower fluid velocities than both the
arteriole/capillary and venous networks.

\begin{table}
  \begin{subtable}{1.0\textwidth}
    \centering
    \begin{tabular}{c|rrrr|rrrrrr}
    \toprule
     & \#cells & $h_{\min}$ & $h_{\max}$ & \#dofs & $\textrm{d}V$ & $\Delta \textrm{d}V$ & $\max p_1$ & $\Delta p_1$ & $\max v_2$ & $\Delta v_2$ \\
    \midrule
    a1 & 20911 & 1.9 & 13.8 & 135774 & 961 & 577 & 1086 & 541 & 1.84 & 0.99 \\
    a2 & 66849 & 0.86 & 12.8 & 364416 & 1099 & 643 & 1144 & 555 & 2.08 & 1.10 \\
    a3 & 198471 & 0.43 & 11.4 & 1021749 & 1186 & 677 & 1177 & 564 & 2.28 & 1.19 \\
    \midrule 
    u2 & 167288 & 0.7 & 9.8 & 922350 & 1162 & 668 & 1160 & 559 & 2.24 & 1.18 \\
    \bottomrule
  \end{tabular}
    \captionsetup{width=.9\linewidth}
    \caption{Quantities of interest on the initial mesh (a1) and two
      adaptive refinement levels (a2, a3) and after uniform refinement
      (u2). Each row gives the number of mesh cells (\#cells), minimal
      and maximal cell size $h_{\min}$, $h_{\max}$ (mm), the number of
      degrees of freedom per time step (\# dofs), the computed peak
      volume change of the domain $\textrm{d}V$ (mm$^3$), and
      pulsatile volume change amplitude $\Delta \textrm{d}V$ (mm$^3$),
      peak arteriole/capillary pressure $\max p_1$ (Pa) and its
      pulsatile amplitude $\Delta p_1$ (Pa), peak venous fluid
      velocity $\max v_2$ (mm/s) and its pulsatile amplitude $\Delta
      v_2$ (mm/s).}
  \vspace{1em}
  \label{tab:adaptive:brain:a}
  \end{subtable}
  \begin{subtable}{1.0\textwidth}
    \centering
    \begin{tabular}{c|cccc|c}
    \toprule
     & $\eta_1$ & $\eta_2$ & $\eta_3$ & $\eta_4$ & $\eta$ \\
    \midrule
    a1 & $5.6 \times 10^3$ & $5.7 \times 10^5$ & $1.6 \times 10^6$ & $2.9 \times 10^3$ & $2.2 \times 10^6$ \\
    a2 & $4.6 \times 10^3$ & $2.7 \times 10^5$ & $7.8 \times 10^5$ & $2.9 \times 10^3$ & $1.1 \times 10^6$ \\
    a3 & $3.5 \times 10^3$ & $1.3 \times 10^5$ & $3.9 \times 10^5$ & $2.9 \times 10^3$ & $5.0 \times 10^5$ \\
    \bottomrule
  \end{tabular}
    \captionsetup{width=.9\linewidth}
    \caption{Computed error estimate $\eta$ and its partial
      contributions $\eta_1, \eta_2, \eta_3, \eta_4$ (see
      cf.~\eqref{eq:etas}) for the series of adaptively refined brain
      meshes (a1, a2, a3).}
  \label{tab:adaptive:brain:b}
  \end{subtable}
  \label{tab:adaptive:brain}
  \caption{} 
\end{table}

The local error indicators $\{ \eta_K \}_{K}$ as defined
by~\eqref{eq:eta_K} show substantial local error contributions with
substantial spatial variation (\Cref{fig:brain:indicators}): the
values range from the order of $10^3$ to $10^9$ on the initial mesh
$\mathcal{T}_h^0$. This large variation in error indicator magnitude
makes the choice of marking strategy important: the D\"orfler marking
strategy would lead to the marking of perhaps only a handful of cells
in this case as the local error indicators for a few cells would
easily add up to a significant percentage of the total
error. Therefore, we instead choose to employ a maximal marking
strategy with a marking fraction $\gamma_M = 0.03$ for this test
scenario.
\begin{figure}
  \includegraphics[width=0.49\textwidth]{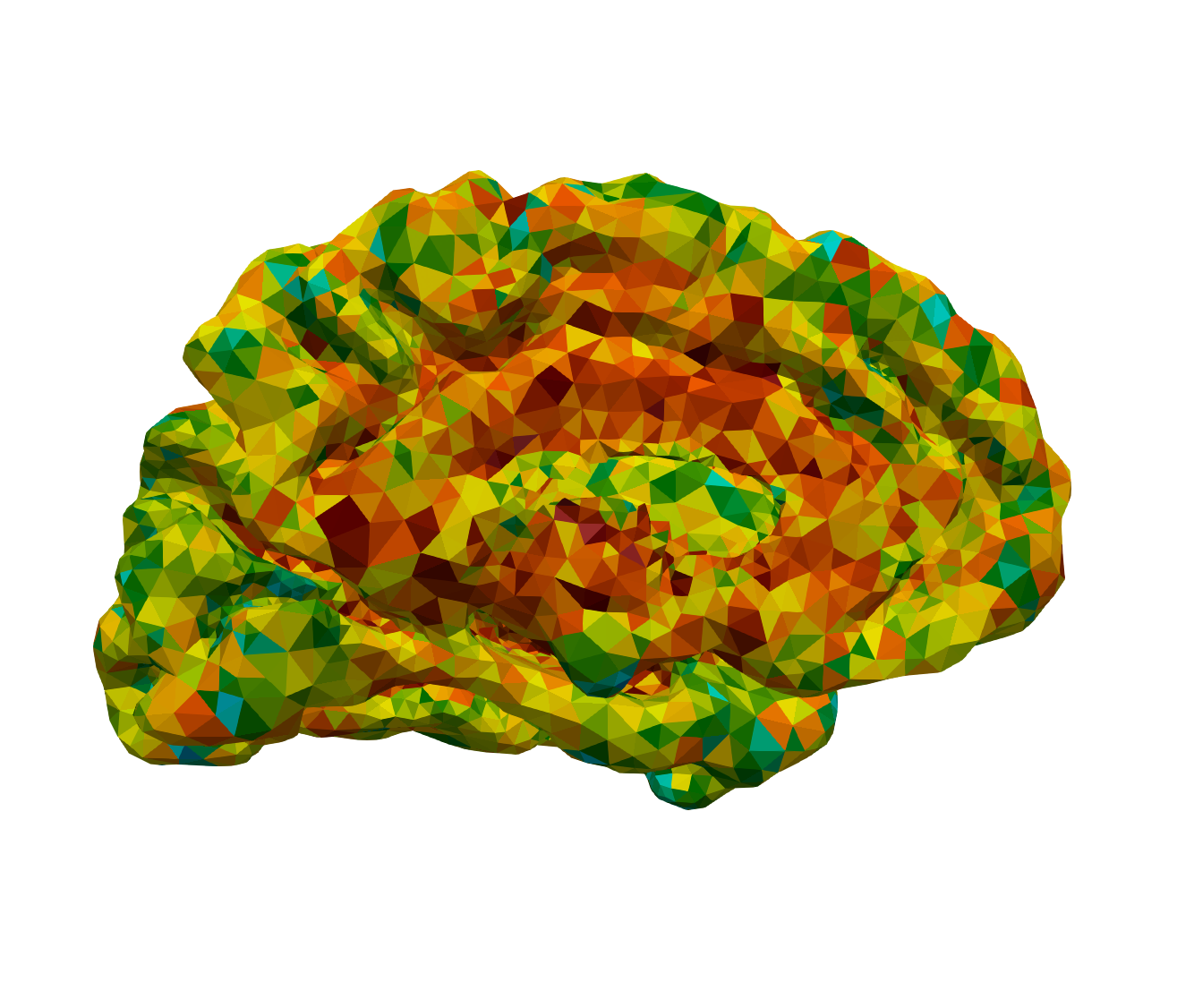}
  \includegraphics[width=0.49\textwidth]{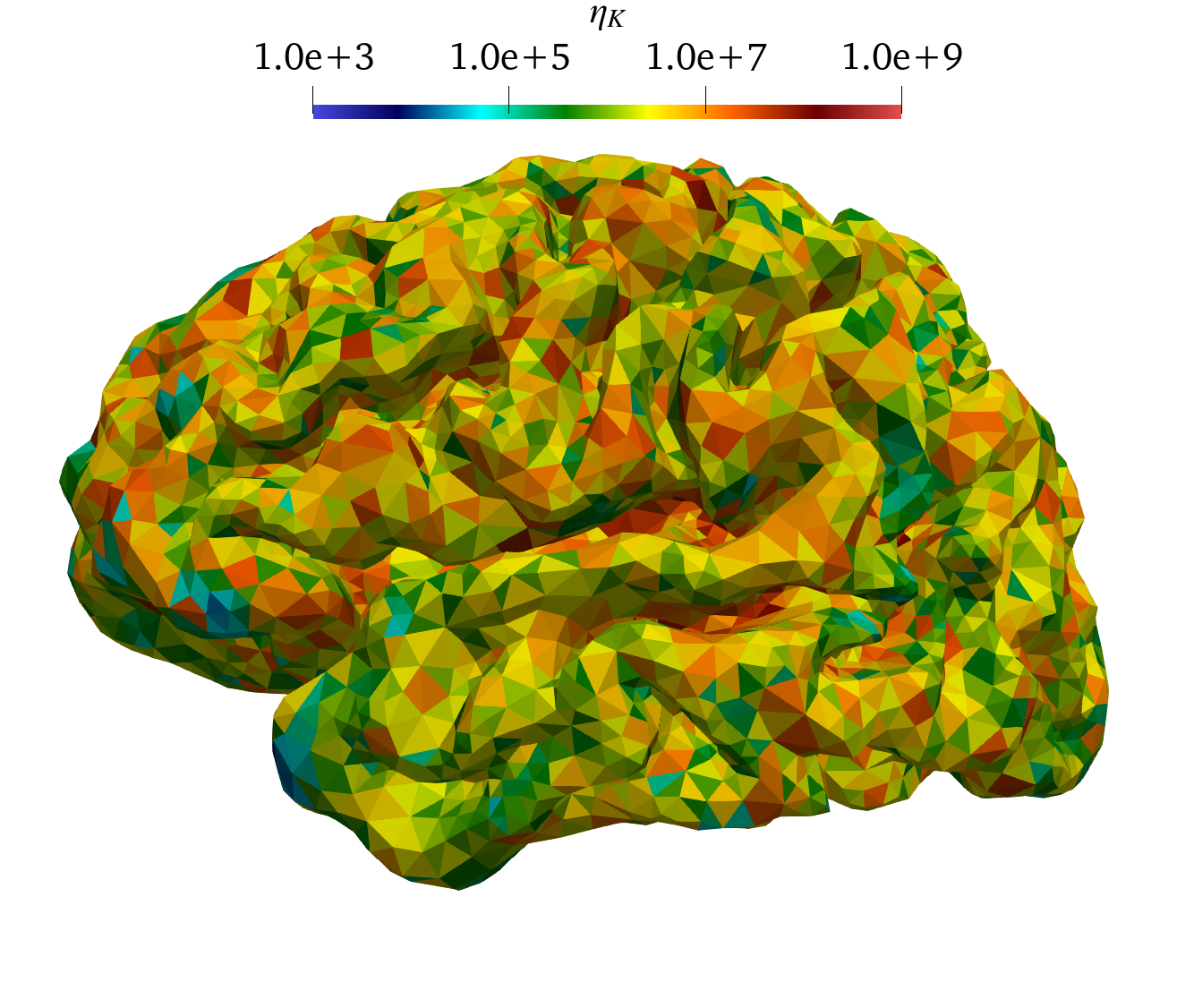} \\
  \vspace{-3em}
  \includegraphics[width=0.49\textwidth]{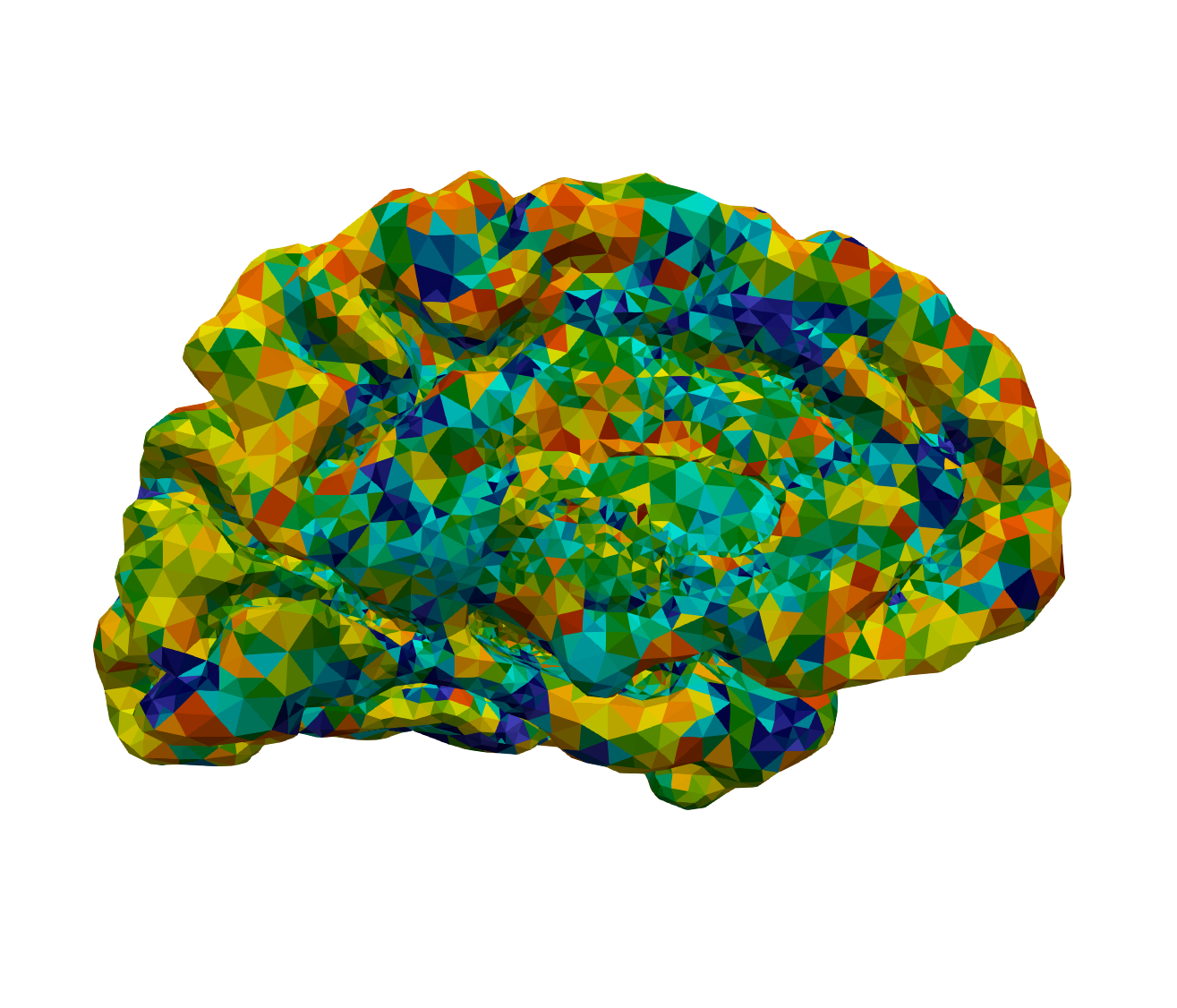}
  \includegraphics[width=0.49\textwidth]{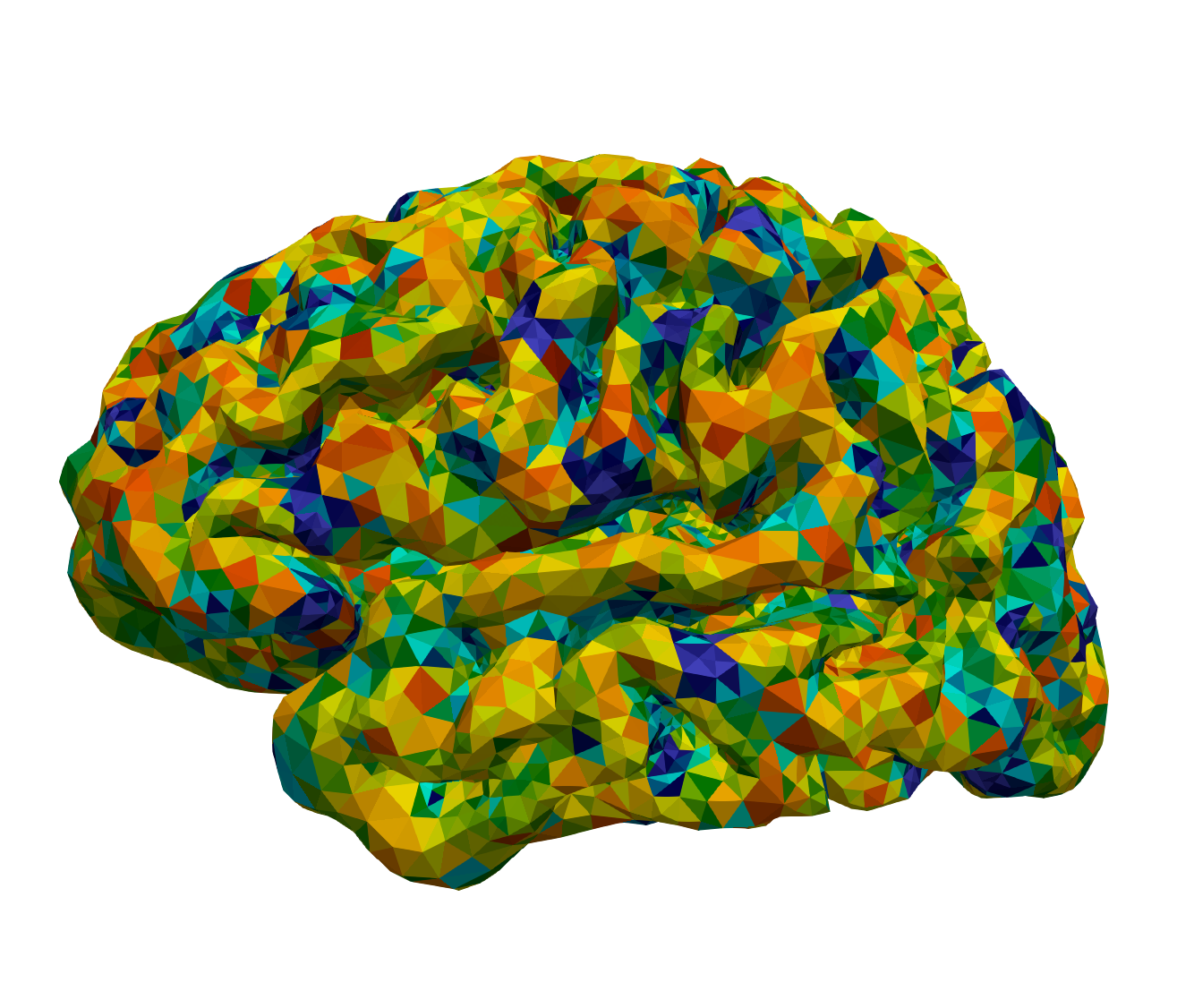} \\
  \vspace{-3em}
  \includegraphics[width=0.49\textwidth]{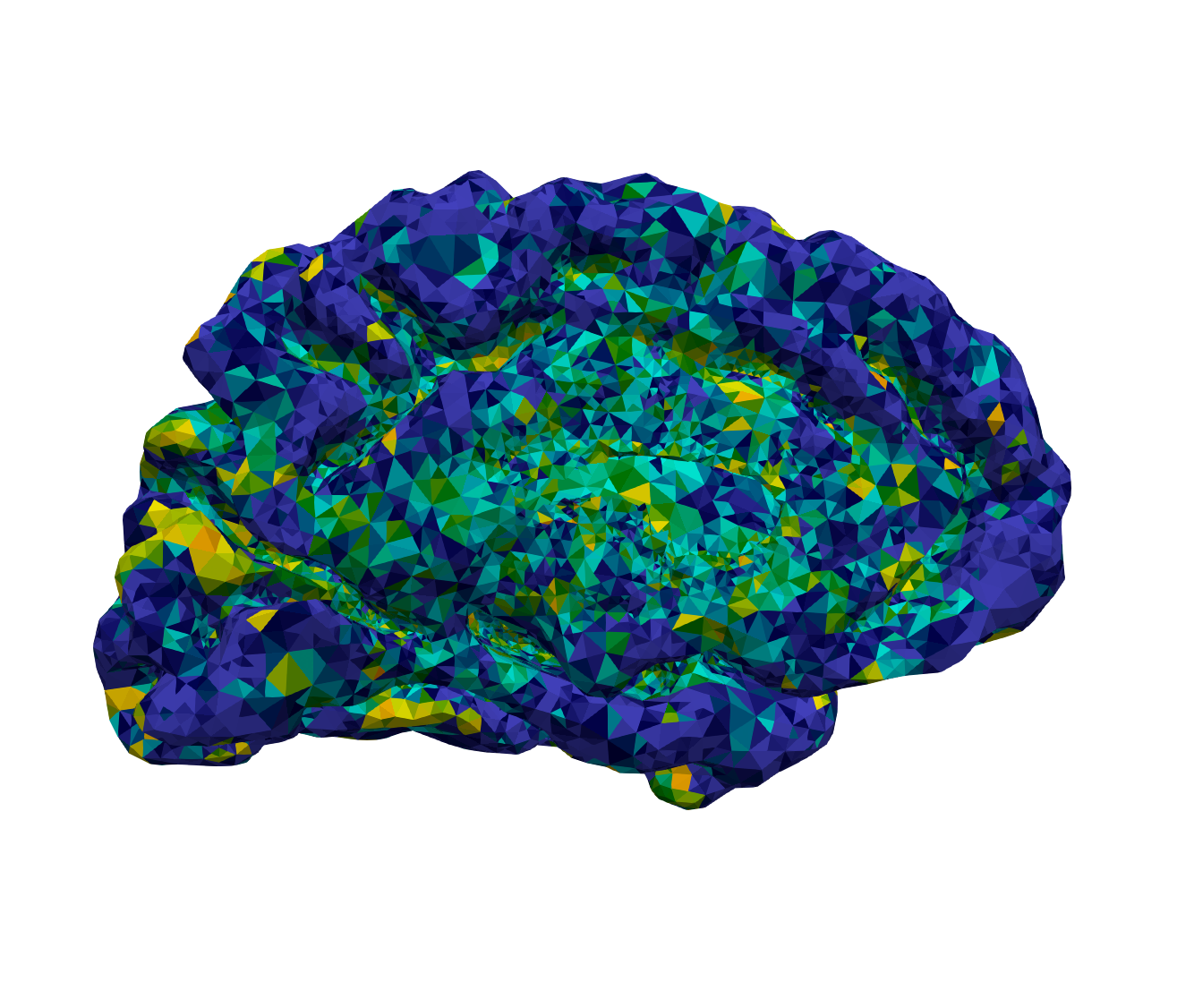}
  \includegraphics[width=0.49\textwidth]{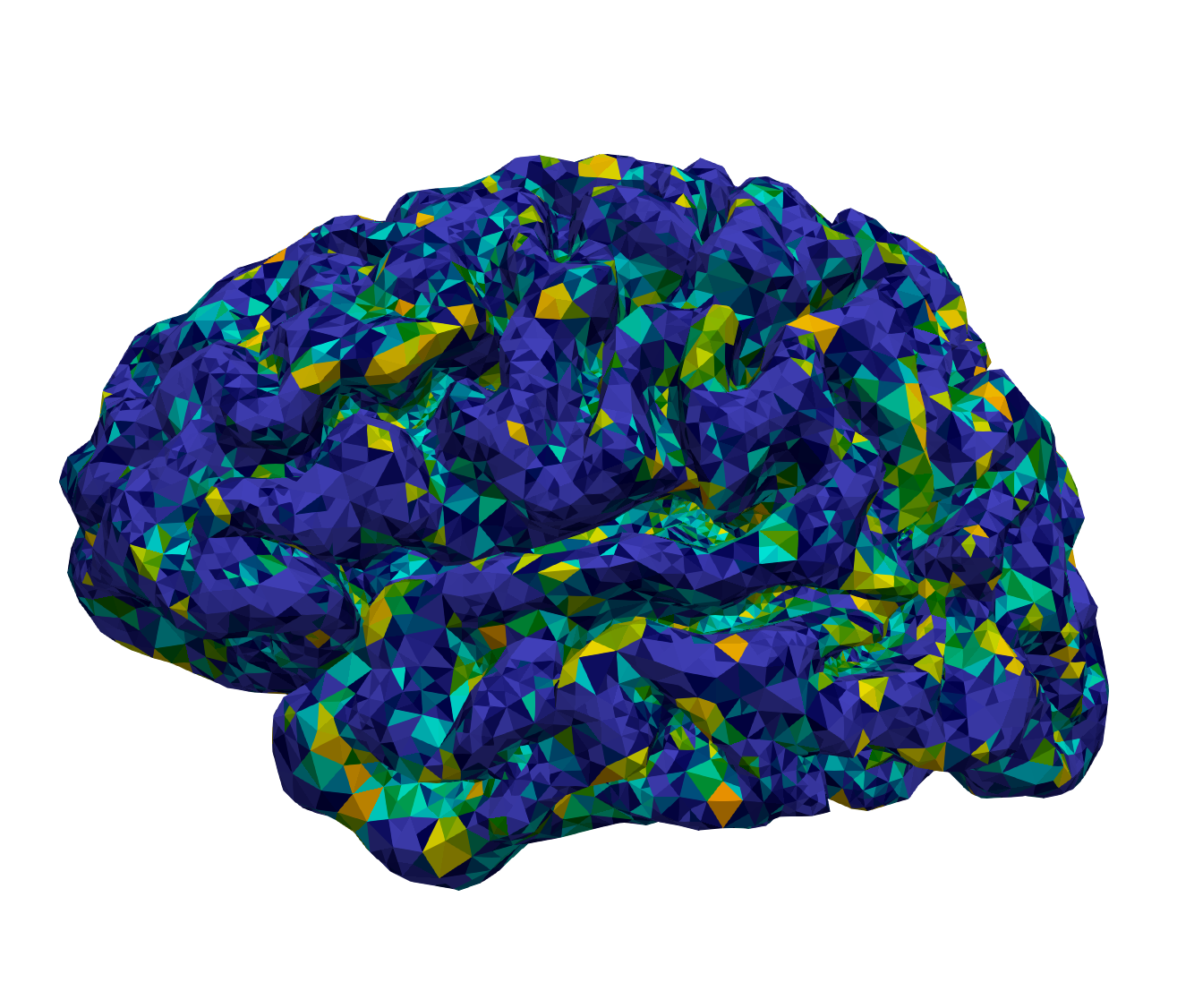}
  \caption{Error indicators $\{ \eta_K \}_K$ for three levels of
    adaptive refinement $\mathcal{T}_h^0, \mathcal{T}_h^1,
    \mathcal{T}_h^2$ the brain simulation scenario. Refinement levels
    from top to bottom (a1, a2, a3), sagittal views from right and
    left on the left and right.}
  \label{fig:brain:indicators}
\end{figure}

\changed{The adaptive algorithm yields locally refined meshes with
  around 67 000 cells after one refinement and 198 000 cells after two
  (\Cref{tab:adaptive:brain:a}). The error estimates
  cf.~\eqref{eq:etas} decrease with the adaptive refinement
  (\Cref{tab:adaptive:brain:b}). The contribution from $\eta_2$ and
  $\eta_3$ dominates the error estimate, and both of these as well as
  the total error estimate $\eta$ seem to halfen for each adaptive
  refinement level. We also note that in this simulation scenario, for
  all time steps $n$ and refinement levels, the spatial error
  contribution $\eta_h^n$ dominates the temporal contribution
  $\eta_{\tau}^n$ cf.~\Cref{alg:time}. Thus, the adaptive algorithm
  does not refine the time step and the uniform initial time step of
  $\Delta t = 0.1$ is kept throughout.}

\changed{We also inspect the computed quantities of physiological interest
(\Cref{tab:adaptive:brain:a}). Using the uniform refinement as an
intermediate reference value, we observe that the adaptive algorithm
seems to produce more accurate estimates of these quantities of
interest even after a single adaptive refinement, and that the
quantities of interest after two refinements are more accurate than
those of a uniform refinement. The adaptive procedure is therefore
able to drive more accurate computation of quantities of interest at a
lower or comparable cost as uniform refinement.}

\changed{Finally, we observe that a single uniform refinement yields a
  mesh with around 167 000 cells (\Cref{tab:adaptive:brain:a}). Thus,
  even with a small marking fraction of 3\%, the mesh growth in each
  adaptive iteration is substantial. In the Plaza
  algorithm~\citep{plaza2003mesh} and other similar conforming mesh
  refinement algorithms, both cells marked for refinement as well as
  neighboring cells will be refined to avoid mesh artefacts such as
  hanging nodes. Therefore, the domain geometry and initial mesh
  connectivity may strongly influence the adaptive mesh growth, and
  mesh growth may be more rapid than anticipated, especially in 3D. A
  more targeted adaptivity and more gradual growth could possibly be
  achieved with even lower marking fractions, though in the current
  case the propagation of cell refinement to neighboring cells seems
  to dominate. In any case, allowing for meshes with hanging nodes
  could be an effective albeit more disruptive strategy for reducing
  the computational complexity.}

\subsection{\changed{Pulsatility driven by boundary pressure}}

\changed{We also consider an alternative, more localized, scenario in
  which, instead of considering a Windkessel model for the CSF
  pressure and the directly coupled PVS pressure $p_3$, we directly
  prescribe a variation in the boundary PVS pressure. Concretely, we
  set
\begin{equation}
  p_{\rm csf} = - 2 \times 133 \sin (2 \pi t)
\end{equation}
and thus boundary CSF pressure variations of up to $\pm 2 \times 133$
Pa (corresponding to approx.~$\pm 2$ mmHg) in each cycle. In this
scenario, we set the bulk fluid influx to zero ($g_1 = g_2 = g_3 =
0.0$). We consider otherwise the same experiments as in
\Cref{sec:brain:source} and the same adaptive parameters. Also for
this case, we observe that the adaptive refinement -- even with a
small marking fraction and maximal marking yields non-localized
marking patterns and a relatively rapid growth in the number of mesh
cells. Three adaptive refinements yields meshes with $20\,911$,
$68\,608$ and $197\,975$ cells and no refinement of the time steps;
numbers which are comparable with the previous case.}

\changed{These results corroborate the observation that the adaptive
  algorithm drives distributed mesh refinement, and that the spatial
  errors overall dominate. Moreover, further studies may consider
  finer initial meshes and further refinements. A prerequisite for this
  would be robust parallel adaptive refinement algorithms including
  robust transfer of fields within the mesh hierarchy, and considered
  the topic of later work.}

\section*{Acknowledgments}

We wish to thank Dr.~Magne Nordaas, \changed{Dr.~Chris Richardson,} and
Prof.~Ragnar Winther for constructive discussions, as well as Dr.~Lars
Magnus Valnes for invaluable aid with the human brain mesh generation
pipeline.

\bibliographystyle{plainnat}
\bibliography{references}

\end{document}